\documentclass[11pt]{article}
\usepackage{amsmath,amsfonts,amsthm,amssymb,latexsym,amscd,amsopn,eucal,mathrsfs,graphics,color,enumerate,lscape,enumerate}
\usepackage[all,cmtip]{xy}
\usepackage[margin=1in]{geometry}

\usepackage[medium]{titlesec}
\usepackage[dvipsnames]{xcolor}
\usepackage[colorlinks]{hyperref}
 \hypersetup{
  colorlinks,
  citecolor=Green,
  linkcolor=red,
  urlcolor=magenta}

\DeclareFontEncoding{OT2}{}{} % to enable usage of cyrillic fonts

\def\bigstab{{\rm bigstab}}

\def\SL{{\rm SL}}
\def\GL{{\rm GL}}
\def\PGL{{\rm PGL}}

\def\Stab{{\rm Stab}}
\def\Sym{{\rm Sym}}
\def\Jac{{\rm Jac}}

\def\O{{\mathcal O}}
\def\SO{{\rm SO}}
\def\P{{\mathbb P}}
\def\Disc{{\rm Disc}}
\def\s{{\rm stab}}
\def\ls{{\rm ls}}
\def\rk{{\rm rk}}
\def\red{{\rm red}}

\def\Aut{{\rm Aut}}
\def\irr{{\rm irr}}
\def\inv{{\rm inv}}
\def\red{{\rm red}}
\def\Vol{{\rm Vol}}
\def\R{{\mathbb R}}
\def\F{{\mathbb F}}
\def\FF{{\mathcal F}}
\def\RR{{\mathcal R}}

\def\CC{{\mathcal C}}
\def\Q{{\mathbb Q}}

\def\H{{\mathcal H}}

\def\C{{\mathcal C}}

\def\W{{\mathcal W}}

\def\Z{{\mathbb Z}}
\def\P{{\mathbb P}}
\def\F{{\mathbb F}}
\def\Q{{\mathbb Q}}
\def\C{{\mathbb C}}

\def\H{{\mathcal H}}

\def\BB{{\mathcal B}}

\def\Gm{\mathbb{G}_\mathrm{m}}

\def\Zp{{\mathbb{Z}_p}}
\def\ZZp{\mathbb{Z}_p}
\def\Qp{{\mathbb{Q}_p}}
\def\Fp{{\mathbb{F}_p}}
\def\height{{\mathrm{ht}}}
\def\Pic{{\mathrm{Pic}}}

\def\equalcases{{$1$, $2$, $4$, and~$5$  of Table~$\ref{table:Invariants}$}}

\theoremstyle{plain}
  \newtheorem{theorem}{Theorem}[section]
  \newtheorem{proposition}[theorem]{Proposition}
  \newtheorem{corollary}[theorem]{Corollary}
  \newtheorem{lemma}[theorem]{Lemma}
  
\theoremstyle{definition}
  \newtheorem{definition}[theorem]{Definition}

\theoremstyle{remark}

\title{On average sizes of Selmer groups and ranks in \\
families of elliptic curves having  marked points}

\author{Manjul Bhargava and Wei Ho}

\date{}

\begin{document}
\maketitle
%\tableofcontents
%\begin{abstract}
%We determine average sizes/bounds for the $2$- and $3$-Selmer groups in various families of elliptic curves with marked points, thus confirming several cases of the Poonen--Rains heuristics. As a consequence, we deduce that the average ranks of the elliptic curves in all of these families are bounded. Our proofs are uniform and make use of parametrizations involving various forms of $2 \times 2 \times 2 \times 2$ and $3 \times 3 \times 3$ matrices that we studied in a previous paper.
%
%We also deduce that $100\%$ of genus one curves of the form $y^2 = Ax^4 + Bx^2 z^2 + Cz^4$ with $A, B, C \in \Z$, when ordered by $\max\{|B|^2,|AC|\}$, fail the Hasse principle. Other forthcoming applications include proofs that a positive proportion of integers are (respectively, are not) the sum of two rational cubes, and a positive proportion of genus one curves in $\P^1 \times \P^1$ over $\Q$ fail the Hasse principle.
%\end{abstract}

\section{Introduction} \label{sec:intro}
Let $F_0$, $F_1$, $F_1(2)$, $F_1(3)$, and $F_2$ denote the following
families of elliptic curves over $\Q$:
\begin{eqnarray*}
  F_0 & = & \left\{y^2 = x^3+a_4x+a_6 \;|\; a_4,a_6\in\Z, \;\Delta \neq 0 \right\}; \\[8pt]
  F_1 & = & \left\{y^2 + a_3 y = x^3+a_2x^2+a_4x \;|\; a_2,a_3,a_4\in\Z, \;\Delta \neq 0 \right\}; \\[8pt]
  F_1(2) & = &  \left\{y^2 = x^3+a_2x^2+a_4x \;|\; a_2,a_4 \in\Z, \;\Delta \neq 0 \right\}; \\[8pt]
  F_1(3) & = & \left\{y^2 + a_1 xy + a_3 y =x^3 \;|\; a_1,a_3 \in\Z, \;\Delta \neq 0 \right\}; \mbox{ and } \\[8pt]
  F_2 & = & \left\{y^2+a_1 x y + a_3 y = (x-a_2)(x-a_2')(x-a_2'') \;|\; \right.\\[2.25pt]
  &  &\left. \;\;\;\;\;\;\;\;\;\; a_1, a_2, a_2', a_2'', a_3 \in \Z, \; a_2 + a_2' + a_2'' = 0, \;\Delta \neq 0 \right\},
\end{eqnarray*}
where, in each case, $\Delta$ is the {\em discriminant} polynomial in the coefficients $a_i, a_i', a_i''$ whose
nonvanishing is equivalent to the curve being nonsingular.
Then we see that $F_0$ is the family of all elliptic curves;  $F_1$ is the family of all elliptic curves with a marked rational point (at (0,0)); 
$F_1(2)$ is the family of all elliptic curves with a marked rational point of order two (at (0,0));
$F_1(3)$ is the family of all elliptic curves with a marked rational point of order three (at (0,0)); and
$F_2$ is the family of all elliptic curves with two marked rational points (at, e.g., $(a_2, 0)$ and $(a_2',0)$). In fact,  we will show in \S\ref{sec:markedpts} that $100\%$ of the curves in $F_1$ (respectively, $F_2$) have rank at least $1$ (resp., $2$) when ordered by height, to be defined below.

The aim of this article is to demonstrate that the average rank of elliptic curves in each of these five families is bounded above.  In fact, we will prove the stronger result that average size of the $2$- and/or $3$-Selmer groups of elliptic curves in each of these families is bounded above.  

To state these results more precisely, define the {\em $($naive$)$ height} of a Weierstrass elliptic curve~$E$ in any of the families $F_0$, $F_1$, $F_1(2)$, or $F_1(3)$ by
$\height(E) := \max\{|a_i|^{12/i}\}$; similarly, define the {\em $($naive$)$ height} of an elliptic curve $E$ in the family~$F_2$ by $\height(E):=\max\{|a_i|^{12/i},|a_2'|^{6}, |a_2''|^{6}\}.$
We prove the following theorem. 

\begin{theorem} \label{thm:SelmerAverages}
When elliptic curves in the families $F_0$, $F_1$, $F_1(3)$, $F_1(2)$, $F_2$ are ordered by  height:
\begin{enumerate}[\quad \rm (a)]
\item
The average size of the $2$-Selmer group in $F_0$ is $3$.
\item
The average size of the $3$-Selmer group in $F_0$ is $4$.
\item
The average size of the $2$-Selmer group in $F_1$ is at most $6$.
\item
The average size of the $3$-Selmer group in $F_1$ is $12$.
\item
The average size of the $3$-Selmer group in $F_1(2)$ is $4$.
\item
The average size of the $2$-Selmer group in $F_1(3)$ is at most $3$.
\item
The average size of the $2$-Selmer group in $F_2$ is at most $12$.
\end{enumerate}
\end{theorem}

\noindent In particular, these Selmer bounds and averages agree with and confirm several instances of the Poonen--Rains heuristics for families with marked points (see \cite[Remark  1.9]{poonenrains}).

For cases (c), (f), and (g), the statement that the average size is at most a given integer~$N$ means that the limsup of the corresponding ratio is at most $N$.
Note that, in each of these averages, we are in fact counting certain
isomorphism classes of elliptic curves infinitely many times.
However, for each such family $F$, we observe that an elliptic curve $E\in
F$ always has a unique representative in $F$ of minimal height, which we then call a {\it minimal} element of
$F$.  For example, a curve $E$ in $F_0$ is minimal if and only if there is
no prime $p$ such that $p^4\mid a_4$ and $p^6 \mid a_6$.  In general, for each of 
our families $F$, a curve is minimal in $F$ exactly when, for every
prime $p$, it is not the case that $p^i\mid a_i$ (and $p^i\mid a_i'$ and $p^i \mid a_i''$)
for all $i$.

If we restrict ourselves only to the minimal curves in each of our
families $F$, then any $\Q$-isomorphism class of elliptic curves in any
such family will be represented exactly once.  We prove that the
averages and upper bounds in Theorem \ref{thm:SelmerAverages} do not change even when one averages only over
these minimal curves (i.e., over all isomorphism classes of curves in
these families, ordered by their minimal heights); moreover, these averages and upper bounds do not change even when any additional finite set of congruence conditions are imposed on the coefficients of the elliptic curves. 

More generally, for $F=F_0$, $F_1$, $F_1(2)$,
  $F_1(3)$, or $F_2$, let $\Phi$ be any subfamily of $F$ that is defined, for each prime $p$, by congruence conditions modulo some power of $p$.  
We say that such a subfamily~$\Phi$ of elliptic curves over $\Q$ is {\em large at $p$}
if the congruence conditions defining $\Phi$ at $p$ contains all elliptic curves $E$ in the family such that $p^2$ does not divide the {\em reduced discriminant} $\Delta_{\red}(E)$ of $E$; and we say that the subfamily $\Phi$ of $F$ is {\em large} if it is large at all but finitely many primes $p$. The {\em reduced discriminant} $\Delta_{\red}$ is the polynomial that is the squarefree part of the discriminant polynomial $\Delta$. We prove the following strengthening of Theorem~\ref{thm:SelmerAverages}:

\begin{theorem} \label{congmain} The average values and upper bounds given in Theorem
  $\ref{thm:SelmerAverages}$ for the families $F_0$, $F_1$, $F_1(2)$,
  $F_1(3)$, and $F_2$ remain the same even when one averages over any
  large subfamily of one of these families.
\end{theorem}

\noindent
Note that the sets of all curves and the sets of all minimal curves in
$F_0$, $F_1$, $F_1(2)$, $F_1(3)$, and~$F_2$ are large. 
So too are the subfamilies of all elliptic curves (resp., all minimal curves) 
in these families defined by any finite set of congruence conditions, as are those having semistable reduction in the families $F_0$, $F_1$, and $F_2$. 
Thus Theorem~\ref{congmain} applies to quite general subfamilies of
these families of elliptic curves.

Using the fact that the $p$-rank of the $p$-Selmer group of an elliptic
curve bounds its algebraic rank, we obtain the following:

\begin{theorem} \label{thm:AverageRanksBounded}
  When elliptic curves  in any large subfamily of one of the families $F_0$, $F_1$, $F_1(2)$, $F_1(3)$, and
  $F_2$ are ordered by height $($resp., minimal height$)$, the average
  rank is bounded.
\end{theorem}
Indeed, we obtain the explicit bounds of 7/6, 13/6, 7/6, 3/2, and 7/2
for the limsup of the average ranks of the curves in any large subfamily of $F_0$, $F_1$, $F_1(2)$, $F_1(3)$,
and $F_2$, respectively.

We note that considering only the $2$-Selmer group is {\it not} sufficient to prove a finite bound on the average rank in the family $F_1(2)$ of elliptic curves over $\Q$ with a marked rational $2$-torsion point.  As shown by Klagsbrun and Lemke-Oliver~\cite{KLO}, the average size of the $2$-Selmer group in the family $F_1(2)$ is infinite.  Moreover, if $\phi$ denotes the $2$-isogeny of elliptic curves $E\to E'$ associated to the marked rational $2$-torsion point on the elliptic curve $E$ in $F_1(2)$, and $\hat\phi$ denotes the dual isogeny, then Kane and Klagsbrun~\cite{kane-klagsbrun}  have shown that the average sizes of the $\phi$- and $\hat\phi$-Selmer groups in $F_1(2)$ are also each unbounded.  Since Theorem~\ref{thm:SelmerAverages}(e) shows that the average size of $3^{\rk(E)}$, and therefore $2^{\rk(E)}$, is bounded for the elliptic curves in $F_1(2)$, it follows that most (indeed, a density of 100\% of) elements in the $2$-Selmer groups, and the $\phi$- and $\hat\phi$-Selmer groups, of elliptic curves in $F_1(2)$ must correspond to nontrivial elements of the Tate--Shafarevich group.  

\begin{theorem}\label{noratlpt}
When all $2$-Selmer elements $($resp., $\phi$- or $\hat\phi$-Selmer elements$)$ of elliptic curves over~$\Q$ with a marked rational $2$-torsion point are ordered by the heights of these elliptic curves, a density of $100\%$ have no rational point, i.e., correspond to nontrivial elements of the Tate--Shafarevich group. 
\end{theorem}

As a consequence, we also obtain:

\begin{theorem}
When elliptic curves over $\Q$ with a marked rational $2$-torsion point are ordered by height, the average size of the $2$-torsion $($resp., $\phi$- or $\hat\phi$-torsion$)$ subgroup of the Tate--Shafarevich group is infinite. 
\end{theorem}

Theorem~\ref{noratlpt} can be made more explicit in the case of $\phi$-Selmer groups. Indeed, elements of the $\phi$-Selmer group of the elliptic curve $E:y^2=x^3+a_2x^2+a_4x$ in $F_1(2)$ can be represented by the genus one curves $Y_{A,B,C}:y^2=Ax^4+Bx^2z^2+Cz^4$ that have points locally at every place, where $A,B,C\in\Z$, $AC=a_4$, and $B=a_2$. We conclude from Theorem~\ref{noratlpt} that 100\% of such curves $Y_{A,B,C}$ in fact have no rational point when ordered by the heights of their Jacobians. 

\begin{theorem}\label{XABC}
When curves $y^2=Ax^4+Bx^2z^2+Cz^4$ $(A,B,C\in\Z)$ having points everywhere locally are ordered by the height $\max\{|B|^2,|AC|\}$, a density of $100\%$ have no rational point.
\end{theorem}

We remark that it is precisely the curves appearing in Theorem~\ref{XABC} that are studied in first courses on descent on elliptic curves having a rational $2$-torsion point.  Theorem~\ref{XABC} states that most (in fact, $100\%$ of) such torsors for elliptic curves have no rational points.

The above theorems on the average sizes of Selmer groups in these families have several other known applications, including that a positive proportion of integers are (respectively, are not) the sum of two rational cubes \cite{sumofcubes} and that a positive proportion of bidegree $(2,2)$ genus one curves in $\P^1 \times \P^1$ over $\Q$ fail the Hasse principle \cite{hasse22}.

\subsection*{Method of proof} 

The proof of Theorems \ref{thm:SelmerAverages}--\ref{XABC} rely on parametrizations of Selmer elements of elliptic curves having marked points, via orbits in certain {\em coregular representations}, which we obtained in \cite{BH}. A representation of an algebraic group $G$ on a
vector space $V$ is called {\it coregular} if the ring of invariants
is a free polynomial ring.  In \cite{BH}, we classified the orbits of
$G(K)$ on $V(K)$ for a field $K$ for many coregular representations,
and showed that the $K$-orbits  in these cases 
correspond to genus one curves over $K$ together with 
extra data, such as line bundles on these curves and marked points on their Jacobians. 

Specifically, we proved that the nondegenerate $\Q$-orbits having integral
invariants in these coregular representations involving ``Rubik's cubes'', ``hypercubes'', and their symmetrizations (see Table~\ref{table:Invariants}) naturally correspond to line bundles of degree 2 or 3 on principal
homogeneous spaces for elliptic curves in these families.  By a result of Cassels~\cite{Cassels}, 
elements of the $p$-Selmer group of an elliptic curve $E$ may
be realized as isomorphism classes of {\it locally soluble}
principal homogeneous spaces $C$ for $E$ along with a degree $p$ line bundle on
the curve $C$.  We conclude that the ``locally soluble'' $\Q$-orbits of these representations correspond
to elements of the $2$- or $3$-Selmer groups for elliptic curves
in these families. (Note that for the family $F_0$ these parametrizations were classically known 
and were used in \cite{BS} and \cite{BS2} to prove Theorems \ref{thm:SelmerAverages}--\ref{thm:AverageRanksBounded}
for $F_0$.) In order to obtain the averages in Theorem \ref{thm:SelmerAverages},
we are thereby reduced to counting $\Q$-orbits in these representations that are
both locally soluble and have bounded integral invariants.
 
To successfully count these $\Q$-orbits, we select a representative integral orbit,
i.e., an orbit of $G(\Z)$ on $V(\Z)$, for each $\Q$-orbit.  In
particular, we show that for any locally soluble element of $V(\Q)$ with integral invariants,
there exists an integral representative in $V(\Z)$ with essentially the same
invariants (i.e., up to absolutely bounded factors).  Thus, to count rational orbits corresponding
to Selmer elements of elliptic curves having bounded invariants, it suffices to count
integer orbits corresponding to Selmer elements having essentially those same
invariants.

By a suitable adaptation of the counting techniques of \cite{dodpf}
and \cite{BS}, we first carry out a count of the {total
  number} of integral orbits in these representations having bounded
height satisfying the appropriate {\it irreducibility} conditions (to
be defined).  The primary obstacle in this counting, as in previous
representations encountered in, e.g., \cite{dodpf,BS}, is that the
fundamental region in which one has to count points is not bounded but
has a complex system of cusps going off to infinity.  A priori, it
could be difficult to obtain exact counts of points of bounded height
in the cusps of these fundamental regions.  We show, however, that most
of the points in the cusps corresponding to Selmer group elements
lie on subvarieties given by certain algebraic conditions on
the associated geometric data, and we show that these {\it
reducible} points can be counted by a different argument.  The
problem then reduces to counting points, corresponding to Selmer
elements, in the {\it main body} of the fundamental region, which we show
are predominantly irreducible.

Since not all of the genus one curves correspond to Selmer elements, to
obtain the averages in Theorem \ref{thm:SelmerAverages} requires a 
sieve to restrict to the locally soluble $\Q$-orbits.  
The upper bound sieves are easier; the lower bounds are much harder, and we accomplish them in cases (a), (b), (d), and (e) of Theorem~\ref{thm:SelmerAverages}.  We note that only the upper bounds are required for the upper bounds on average rank.  By carrying out these sieves, we prove that
each of the averages or bounds in Theorem \ref{thm:SelmerAverages} arises naturally as the sum of two contributions: one from the ``main body'' of the fundamental region, which is essentially the
Tamagawa number of the group acting; and one from the cusp of the
fundamental region, which is essentially a count of the number of elements 
corresponding to a certain subgroup $S'$ of the Selmer group $S$, namely, the image
in $S$ of the subgroup of the Mordell--Weil group generated by the marked points.

We summarize these cusp ($S'$) and Tamagawa (non-$S'$) contributions in Table~\ref{table:Invariants}.  The second and third columns of Table~\ref{table:Invariants}
give (the isomorphism class of) the algebraic group $G$ acting and the representation $V$ of $G$, respectively.  The
fourth column gives the interpretation of orbits of $G(K)$ on $V(K)$
for a field $K$, as classified in \cite{BH}, namely, as an elliptic curve $E$ in the appropriate family with an $E$-torsor $C$ and a line bundle $L$ (on $C$), whose degree is given in the fifth column.  The sixth column lists
generators for the ring of invariants,
and the degrees of these invariants, for the action
of $G$ on $V$.
(These generators are also the coefficients of the corresponding elliptic
curves in the family $F$.)  The seventh column gives the dimension of $V$, and the
eighth gives the degree of the height function on $V$.  The ninth
column gives the number of elements in the subgroup $S'$ of $S$,
and the tenth column lists the Tamagawa number of $G$.  Finally, the
sum of the ninth and tenth columns yields the averages or upper bounds given in the
eleventh column of the table, and in Theorems~\ref{thm:SelmerAverages} and~\ref{congmain}. 

We explain more precisely the details of our strategy to prove the
contents of Table \ref{table:Invariants} (and thus Theorems \ref{thm:SelmerAverages}
and \ref{thm:AverageRanksBounded}) in \S \ref{sec:outline}.

\begin{landscape}
\begin{table} \label{table:Invariants}
\begin{center}
\begin{tabular}{|r|l|l|l|c|c|c|c||c|c|c|}\hline
\multicolumn{11}{|c|}{\bf Coregular spaces and Selmer groups}\\ \hline
\,\#\hspace{.14in}\!
& \multicolumn{1}{l|}{Group $G$} & \multicolumn{1}{|l|}{Representation $V$} & 
\multicolumn{1}{|c|}{$\Q$-orbits $\leftrightarrow$} & 
\multicolumn{1}{|c|}{deg.} &
\multicolumn{1}{|l|}  {Invariants $a_i,$ $a'_i$} &
\multicolumn{1}{|c|}{dim.} &
\multicolumn{1}{|c||}{\!\!ht.\!\! deg.\!\!} &
\multicolumn{1}{|c|}{cusp} &
\multicolumn{1}{c|}{\!\!Tamagawa\!\!} &
\multicolumn{1}{|c|}{\!Total\!} \\
& & & \multicolumn{1}{|c|}{$(C,E,L)$} & \multicolumn{1}{|c|}{of $L$} &
\multicolumn{1}{|l|}  {and their degrees} &
\multicolumn{1}{|c|}{$n$} &
\multicolumn{1}{|c||}{$k$} &
\multicolumn{1}{|c|}{\!contr.\!} &
\multicolumn{1}{c|}{number} &
\multicolumn{1}{|c|}{} 
\\ \hline\hline 
1.\; & $\PGL_2$     &  $\Sym^4(2)$   & $E\in F_0$   & 2 &  $a_4$, $a_6$ & 5 & 6 & 1 & 2 & 3\\
     &                          &  binary quartic forms                             &  & &  2, 3                 &    &   &    &    &  \\ \hline 
2.\; & $\PGL_3$     &  $\Sym^3(3)$   & $E\in F_0$   & 3 & $a_4$, $a_6$ & 10 & 12 & 1 & 3 & 4\\
     &                          & ternary cubic forms                              & & &  4, 6                 &    &   &    &    &  \\ \hline 
3.\; & $\PGL_2^2$   &  $\Sym^2(2) \otimes \Sym^2(2)$  & $E\in F_1$  & 2 & $a_2$, $a_3$, $a_4$ & 9 & 12 & 2 & 4 & 6\\
     &                       & bidegree $(2,2)$ forms                              &  &  & 2, 3, 4                 &    &   &    &    &  \\ \hline 
4.\; & $\SL_3^3/\mu_3^2$  &  $3 \otimes 3 \otimes 3$ &  $E\in F_1$   &  3 & $a_2$, $a_3$, $a_4$ & 27 & 36 & 3 & 9 & 12\\
     &                       & Rubik's cubes                              & & & 6, 9, 12                 &    &   &    &    &  \\ \hline 
5.\; & $\SL_3^2/\mu_3$  &  $3 \otimes\Sym_2(3)$   &  $E\in F_1(2)$   &  3 & $a_2$, $a_4$ & 18 & 36 & 1 & 3 & 4 \\
     &                       & doubly symmetric Rubik's cubes                              &  & & 6, 12                &    &   &    &    &  \\ \hline 
6.\; & $\SL_2^2/\mu_2$  &  $2 \otimes \Sym_3(2)$ & $E\in F_1(3)$  & 2 & $a_1$, $a_3$ & 8 & 24 &  1 & 2 & 3\\
     &                       & triply symmetric hypercubes                              &  &  & 2, 6              &    &   &    &    &  \\ \hline 
7.\; & $\SL_2^4/\mu_2^3$  &  $2 \otimes 2 \otimes 2 \otimes 2$  & $E\in F_2$   & 2 & $a_1$, $a_2$, $a_2'$, $a_3$ &  16 & 24 & 4 & 8 & 12\\
     &                       & hypercubes                              &  & & 2, 4, 4, 6                 &    &   &    &    &  \\ \hline 
\end{tabular}
\end{center}
\caption{Coregular spaces and Selmer groups} % can move caption to before the table by just moving this line
\vspace{0.5in}
{\it Notation:} For each row, the nondegenerate $\Q$-orbits (that is, elements of $G(\Q)\setminus V(\Q)$) correspond to data in the fourth column, denoted $(C,[E],L)$, where
$C$ is a genus one curve, $L$ is a line bundle of degree $2$ or $3$
(as specified) on $C$, and $E$ is an elliptic curve with a model in 
the specified family $F$ with $\Jac(C) \cong E$.  (For a more precise description of the action of $G$ on $V$, see  \S \ref{sec:parametrizations}.)

\vspace{.1in}
In the language of \cite{BH}, Lines 4 and 5 of Table~\ref{table:Invariants} are the representations corresponding to Rubik's cubes and doubly symmetric Rubik's cubes, respectively; meanwhile, Lines 7 and 6 of Table~\ref{table:Invariants} are the representations corresponding to hypercubes and triply symmetric hypercubes, respectively. Note that the spaces in Lines 1, 2, and 3 may also be viewed as subspaces of these, namely quadruply symmetric hypercubes, triply symmetric Rubik's cubes, and doubly doubly symmetric hypercubes, respectively.

\end{table}
\end{landscape}

\section{Outline of proof} \label{sec:outline}

Let $(G,V)$ denote any one of the representations listed in Table \ref{table:Invariants}.
Let $K$ be a field not of characteristic $2$ or $3$. In Section~\ref{sec:parametrizations}, 
to any (nondegenerate) orbit of $G(K)$ on $V(K)$, we
describe how to associate an elliptic curve $E$
in the corresponding family $F$, but where the coefficients of $E$ lie in $K$
rather than $\Z$; we denote this family of elliptic curves $F(K)$.  
Namely, for a particular $v \in V(K)$, the model of $E$ in $F(K)$ has coefficients $a_i$ (and $a_2'$ for $F = F_2$) given by the generators of the invariant ring for the action of $G$ on $V$.
These invariants $a_i$ are fixed polynomials in $\Z[V]$.

The {\it discriminant} $\Delta(v)$ of $v\in V(K)$ is the
usual discriminant polynomial $\Delta$ in the coefficients $a_i$,
whose nonvanishing is equivalent to the curve $E$ being
nonsingular.  The discriminant on $V(K)$ is thus invariant under
$G(K)$.
We use $V^\s(K) \subset V(K)$ to denote the subset of {\it stable} (or
{\it nondegenerate}) points in $V(K)$, i.e., the points in $V(K)$
where the discriminant is nonzero.
The height $H(v)$ of a point
$v\in V(\Z)$ (or $v\in V(\R)$, using the same formulas) 
is the height $H(E)$ of the corresponding elliptic curve.

We review in \S \ref{sec:parametrizations} the theorems from
\cite{BH} which state that the $G(\Q)$-orbits of
points $v\in V^\s(\Q)$ are in canonical bijection with triples $(E,C,L)$,
where $C$ is a principal homogeneous space for an elliptic curve $E$ in
$F(\Q)$ for one of the families $F = F_0,\ldots, F_2$, and $L$ is a line bundle on $C$ of degree 2 or 3.
The aforementioned elliptic curve is isomorphic to the Jacobian of the genus one curve $C$ obtained 
via these bijections.  Given a specific $v \in V^\s(\Q)$, there is a
natural choice for a model of this elliptic curve with coefficients given by the generators $a_i$, $a_i'$ of the ring of
invariants of the action of $G$ on $V$.
Thus, given $v \in V^\s(\Q)$, there exists a model $E$ of the
elliptic curve in $F(\Q)$ such that the discriminant $\Delta$, height $H$, and invariants
$a_i$, $a'_i$ of $v$ and $E$ agree.

Among the $G(\Q)$-orbits on $V^\s(\Q)$, only the {\it
locally soluble} orbits yield ($2$- or $3$-) Selmer group elements for
the corresponding elliptic curve $E$.  Let $V^\ls(\Q) \subset V^\s(\Q)$
denote those $G(\Q)$-orbits on $V(\Q)$ corresponding to triples $(C,E,L)$ 
that are in the $2$- or $3$-Selmer group of the elliptic curve $E$.  
Thus, the set of all $G(\Q)$-orbits on $V^\ls(\Q)$ yielding a
given elliptic curve $E$ naturally has the structure of a finite
abelian $2$- or $3$-group.  We denote this group by
$S(E)$.  The group $S(E)$ naturally has a subgroup $S'(E)$, namely,
the subgroup generated by the images in $S(E)$ of the marked points on $E$.
We also define a notion of {\it irreducible} for points $v\in V(\Z)$, 
and we show that the points in $$V^\ls(\Z) := V(\Z)\cap V^\ls(\Q)$$ that do
not correspond to elements of $S'(E)$ are all
irreducible. 

Let us now restrict ourselves to those orbits in $V^\s(\Q)$ having
{\it integer} invariants, so that the elliptic curve $E$ associated to a
point $v$ in this subset of $V^\s(\Q)$ is genuinely in
the family $F$.  In order to count the orbits of $G(\Q)$ on $V^\ls(\Q)$
having integral invariants and bounded height --- thereby yielding a count
of the total number of $2$- or $3$-Selmer group elements for elliptic curves of
bounded height in $F$ --- we carry out the following steps:

\begin{enumerate}[(1)]
\item We show that for each $v' \in V^\ls(\Q)$ with integral invariants,
  there exists an {\it integral representative} $v \in V(\Z)$ having
  the identical integral invariants (up to absolutely bounded factors).
  In this way, we may associate an integral orbit (the $G(\Z)$-orbit of $v \in V(\Z)$)
  to the subset in a rational orbit with fixed integral invariants. (In previous
  work like \cite{BS, BS2}, this step was a consequence of minimization
  results of Cremona, Fisher, and Stoll \cite{CFS, fisher-minred5coverings}. See also
  recent work of Fisher and Radi\u{c}evi\'{c} \cite{fisherradicevic} for minimization
  algorithms of some of the spaces considered in this paper.)
\item We construct fundamental domains for the action of $G(\Z)$ on $V(\R)$:
\begin{enumerate}[(a)]
\item First, we show that we may construct a fundamental set $L$
  for the action of $\R^\times \times G(\R)$ on $V(\R)$ that is an
  absolutely bounded set in $V(\R)$.
\item Second, we construct a fundamental domain $\FF$ for the
  action of $G(\Z)\backslash G(\R)$ that is contained in a Siegel
  domain.  Then for any $g\in G(\R)$, we see that $\FF g L\subset
  V(\R)$ yields a finite union of fundamental domains for the action of
  $G(\Z)$ on $V(\R)$.  

  We choose $g$ to vary in a compact set
  $G_0\subset G$ that is the closure of some open set in $G$.  This
  yields a compact but continuously-varying set of fundamental domains
  $\FF g L$ for the action of $G(\Z)$ on $V(\R)$.
\end{enumerate}
\item  As we want to count the irreducible $G(\Z)$-orbits on $V^\ls(\Z)$, we give some sufficient conditions
for reducibility.

\item We give an asymptotic count of {\it all} irreducible
  $G(\Z)$-orbits on $V^\s(\Z)$.  First, by adapting the techniques of
  \cite{dodpf} and \cite{BS}, we show that: (a) the cusps of the
  fundamental regions $\FF gL$ have a negligible number of irreducible
  points; (b) the main body of the fundamental regions $\FF gL$ have a
  negligible number of points that are reducible; and (c) the total number
  of irreducible points in the main bodies of the fundamental regions
  $\FF gL$ having height less than $X$ is asymptotically equal to the
  Euclidean volume of $\FF gL\cap \{v\in V(\R):H(v)<X\}$.
\item The elements of $V(\Z)$ that are in $V^\ls(\Z)$ are defined
  by infinitely many congruence conditions.  In order to count just
  those irreducible $G(\Z)$-orbits on $V^\s(\Z)$ that are locally
  soluble, one must perform a sieve.  The sieve relies in particular on a
  certain ``geometric sieve'', which originates in the work of
  Ekedahl~\cite{Ekedahl} and was further developed by Poonen~\cite{Po} and
  in \cite{geosieve}; it also relies on the application of certain transformations that change ``mod $p^2$'' conditions to ``mod $p$'' conditions as developed in \cite{geosieve,BS2}.  This sieve allows us to obtain a count  of just the irreducible $G(\Z)$-orbits of bounded height in
  $V^\ls(\Z)$.

\item In order to obtain the averages in Theorem \ref{thm:SelmerAverages}, we
  must count the total number of curves of height less than $X$
  in the family $F$. For almost all families, these counts are fairly straightforward;
  however, for the family $F_2$, we establish the required uniformity estimate by embedding the  family into the cusp region of a larger space and use additional invariants, extending the ``Q-invariant method'' from \cite{squarefree-BSW,squarefree-BSW2}.
  
  By using the next step (7), we may add back to the count in
  (5) the number of Selmer elements in $S(E)$, over all curves $E\in
  F$ of height $<X$, that reduce to the identity in $S'(E)$ (i.e., the number of reducible orbits on $V^\ls(\Z)$ having
  height $<X$).  The total number of Selmer elements in $S(E)$
  over all elliptic curves $E\in F$ of height less than $X$, divided by the
  total number of elliptic curves $E\in F$ of height less than $X$, as $X$ tends
  to infinity, then yields the desired averages in Theorem \ref{thm:SelmerAverages}.

\item We prove an auxiliary lemma which shows that, when 
  elliptic curves in the families $F_1$ and $F_2$ are ordered by
  height, a density of $100\%$ of the marked points on these curves 
  have infinite order.  This implies that, for $100\%$ of the curves
  $E$ in  $F_i$, we have that $|S'(E)|=p^i$, where
  $p=2$ or 3.  In other words, for $100\%$ of the elliptic curves $E$
  in the family $F=F_i$, there are $p^i$
  reducible orbits on $V^\ls(\Z)$ corresponding to the elliptic curve
  $E$.  In the cases corresponding to the families $F_1(\cdot)$, the group
  $S'(E)$ is trivial.
\end{enumerate}
We will carry out the details of item ($j$) above in Section~$j+3$ below.  

\section{Parametrizations of Selmer elements}
\label{sec:parametrizations}

We recall (and appropriately modify) the relevant results from \cite{BH}.  In particular, we describe the bijections between the nondegenerate $G(\Q)$-orbits of $V(\Q)$ and the principal homogeneous spaces $C$ for elliptic curves in the corresponding families, including how these parametrization results simplify when only considering {\em locally soluble} $C$ and {\em irreducible} orbits.  In this paper, we only need these results over $\Q$, $\Qp$, and $\mathbb{R}$, but the bijections hold over any base field not of characteristic $2$ or $3$ (as shown in \cite {BH}), and the statements we make in this section about the case of locally soluble curves also extend to other number fields.

We first make a few definitions to state the main parametrization theorem. Recall that a genus one curve over $\Q$ is {\em locally soluble} if it has points over $\Q_p$ for every prime $p$ and over $\R$. For each line in Table \ref{table:Invariants}, let $V^\ls(\Q)$ denote the elements of $V(\Q)$ where the associated genus one curve is locally soluble; as we will see, this set is preserved by $G(\Q)$.

Also, for each case, we show in \cite{BH} that an element of $V(\Q)$ gives rise to a number of binary quartic forms or ternary cubic forms over $\Q$ (for $d = 2$ or $3$, respectively). For $v\in V(\Q)$, we say that $v$ is {\em irreducible} if each of its covariant binary quartic forms (resp., covariant ternary cubic forms) has no rational linear factor (resp., defines a smooth cubic curve in $\P^2$ with no rational flex); we say that $v$ is {\em reducible} otherwise.

We obtain the following bijections for each representation (the descriptions of the group actions may be found at the end of this section):

\begin{theorem} \label{thm:parametrizations}
Consider one of the following lines from Table $\ref{table:Invariants}$:
\begin{center}
\begin{tabular}{c|c|c|c|c|c}
number & group $G$ & representation $V$ & degree $d$ & family $F$ & marked point(s)\\
\hline
$1.$ & $\PGL_2$ & $\Sym^4(2)$ & $2$ & $F_0$ & none \\
$2.$ & $\PGL_3$ & $\Sym^3(3)$ & $3$ & $F_0$ & none \\
$3.$ & $\PGL_2^2$ & $\Sym^2(2) \otimes \Sym^2(2)$ & $2$ & $F_1$ & $(0,0)$\\
$4.$ & $\SL_3^3 / \mu_3^2$ & $3 \otimes 3 \otimes 3$ & $3$ & $F_1$ & $(0,0)$\\
$5.$ & $\SL_3^2 / \mu_3$ & $3 \otimes \Sym_2(3)$ & $3$ & $F_1(2)$ & $(0,0)$ \\
$6.$ & $\SL_2^2 / \mu_2$ & $2 \otimes \Sym_3(2)$ & $2$ & $F_1(3)$ & $(0,0)$ \\
$7.$ & $\SL_2^4 / \mu_2^3$ & $2 \otimes 2 \otimes 2 \otimes 2$ & $2$ & $F_2$ & $(a_2,0)$, $(a_2',0)$
\end{tabular}
\end{center}
\begin{enumerate}
\item[{\rm (a)}] There exists a bijection between the nondegenerate $G(\Q)$-orbits of $V(\Q)$ and isomorphism classes of triples $(E,C,L)$, where $E$ is an elliptic curve in the family $F$ with marked point(s) as indicated, $C$ is an $E$-torsor over $\Q$, $L$ is a line bundle of degree $d$ on $C$, and the marked point(s) represent rational degree $0$ divisors of $C$. Two such triples $(E,C,L)$ and $(E',C',L')$ are isomorphic if $E$ and $E'$ refer to the same elliptic curve in the family $F$ and there is an isomorphism between $C$ and $C'$ preserving both the line bundles and the torsor structure.
\item[{\rm (b)}] There exists a bijection between the nondegenerate $G(\Q)$-orbits of $V^\ls(\Q)$ and pairs $(E,\xi)$, where $E$ is an elliptic curve in $F$ and $\xi$ is an element of the $d$-Selmer group $S_d(E)$ of $E$.
\item[{\rm (c)}] For an elliptic curve $E \in F$, let $S'(E)$ denote the subgroup of $S_d(E)$ generated by the images of the marked points on $E$ under the natural map $E(\Q)/dE(\Q) \to S_d(E)$. Then the nondegenerate $G(\Q)$-orbits of the irreducible elements of $V^\ls(\Q)$ are in bijection with pairs $(E,\xi)$, where $E$ is an elliptic curve in $F$ and $\xi \in S_d(E) \setminus S'(E)$.
\end{enumerate}
\end{theorem}

Theorem \ref{thm:parametrizations}(a) is proved for each case in \cite{BH}, over any fields not of characteristic $2$ or $3$, although the exact statements are slightly different.  There, we study the action of a larger group $G'$ on the corresponding representation $V$ for which the $a_i$ are {\em relative} invariants.  The geometric data parametrized includes the elliptic curve in the family {\em only up to isomorphism}.  Here, the group $G$ is the subgroup of $G'$ that fixes these invariants $a_i$, which correspond to the coefficients of the elliptic curve.

Part (b) of Theorem \ref{thm:parametrizations} follows directly from part (a) by the following lemma, whose proof may be found in, e.g., \cite[Theorem 1.2]{Cassels}:

\begin{lemma}
Let $C$ be a genus one curve over $\Q$.  If $C$ is locally soluble, then any rational point $P$ on the Jacobian of $C$ represents a rational divisor on $C$, not just a rational divisor class.
\end{lemma}

Finally, reducibility or irreducibility for an element in $V(\Q)$ is a $G(\Q)$-invariant notion. An element $v \in V(\Q)$ is {\em reducible} if and only if at least one of the covariant binary quartic forms (resp., ternary cubic forms) is isomorphic to its Jacobian elliptic curve, in other words, if and only if it corresponds to the trivial element in the $2$-Selmer (resp., $3$-Selmer) group of the elliptic curve associated to $v$. We thus obtain part (c) of Theorem \ref{thm:parametrizations}.

\subsection*{The groups and representations} \label{sec:reps}

We more precisely describe the group actions for each of the spaces in Table \ref{table:Invariants} and Theorem \ref{thm:parametrizations}, over any field of characteristic not $2$ or $3$.
\begin{enumerate}[1.]
\item {\bf Binary quartic forms.} For a $2$-dimensional vector space $V$, the $\GL(V)$-action on the representation $\Sym^4(V) \otimes (\wedge^2 V)^{-2}$ factors through a $\PGL(V)$-action on $\Sym^4(V)$.

\item {\bf Ternary cubic forms.} For a $3$-dimensional vector space $W$, the $\GL(W)$-action on the space $\Sym^3(W) \otimes (\wedge^3 W)^{-1}$ factors through an action of $\PGL(V)$ on the space $\Sym^3(W)$.

\item {\bf Bidegree $(2,2)$ forms.} Let $V_1$ and $V_2$ be $2$-dimensional vector spaces. The $\GL(V_1) \times \GL(V_2)$-action on $\Sym^2(V_1) \otimes (\wedge^2 V_1)^{-1} \otimes \Sym^2(V_2) \otimes (\wedge^2 V_2)^{-1}$ factors through an action of $\PGL(V_1) \times \PGL(V_2)$ on $\Sym^2(V_1) \otimes \Sym^2(V_2)$.

\item {\bf Rubik's cubes.} Let $W_1, W_2,$ and $W_3$ by $3$-dimensional vector spaces.  The group $\GL(W_1) \times \GL(W_2) \times \GL(W_3)$ naturally acts on $W_1 \otimes W_2 \otimes W_3$. Consider the subgroup $G' \subset \prod_{i=1}^3 \GL(W_i)$ consisting of triples $(g_1, g_2, g_3)$ with $\prod_{i=1}^3 \det (g_i) = 1$. Let $\Gm$ denote the center of each $\GL(W_i)$. Then each element of the representation is stabilized by the kernel of the multiplication map $\Gm^3 \cap G' \to \mu_3$. We let $G$ be the quotient of $G'$ by this kernel and its representation $W_1 \otimes W_2 \otimes W_3$. Note that $G$ is isomorphic to $\SL_3^3 / \mu_3^2$.

\item {\bf Doubly symmetric Rubik's cubes.} Let $W_1$ and $W_2$ be $3$-dimensional vector spaces.  Let $G'$ be the subgroup of $\GL(W_1) \times \GL(W_2)$ consisting of pairs $(g_1, g_2)$ such that $(\det g_1)(\det g_2)^2 = 1$. The stabilizer of the action of $G'$ on $W_1 \otimes \Sym_2 (W_2)$ is the kernel of the multiplication map sending $(\gamma_1, \gamma_2) \in \Gm^2 \cap G'$ to $\gamma_1 \gamma_2^2$. We will consider the quotient $G$ of $G'$ by this kernel and its action on $W_1 \otimes \Sym_2(W_2)$. Note that $G$ is isomorphic to $\SL_3^2 / \mu_3$.

\item {\bf Triply symmetric hypercubes.} Let $V_1$ and $V_2$ be $2$-dimensional vector spaces.  As before, we consider the subgroup $G'$ of $\GL(V_1) \times \GL(V_2)$ of pairs $(g_1, g_2)$ such that $(\det g_1)(\det g_2)^3 = 1$, and we consider the quotient of $G'$ by the stabilizer of its action on $V_1 \otimes \Sym_3(V_2)$. The group is isomorphic to $\SL_2^2/\mu_2$.

\item {\bf Hypercubes.} Let $V_i$ be a $2$-dimensional vector space for $i = 1, 2, 3, 4$ with the usual $\GL(V_i)$-action on each factor.  Let $G'$ be the subgroup of $\prod_{1 \leq i \leq 4} \GL(V_i)$ consisting of tuples $(g_1, g_2, g_3, g_4)$ with $\prod_{i=1}^4 \det g_i = 1$. Let $G$ be the quotient of $G'$ by the stabilizer of its action on $V_1 \otimes V_2 \otimes V_3 \otimes V_4$, i.e., by the kernel of the multiplication map $\Gm^4 \cap G' \to \mu_2$. The group $G$ is isomorphic to $\SL_2^4 / \mu_2^3$.
\end{enumerate}

In each case, we show in \cite{BH} that the covariant binary quartic forms or ternary cubic forms correspond to the genus one curve $C$ and line bundles $L$ or $L \otimes P$, where $P$ refers to the marked point(s) as elements of $\Pic^0(C)$.  These are obtained in general by determinantal constructions, and there is one binary quartic or ternary cubic form for each factor of $\GL_2$ or, respectively, $\GL_3$.

%%%%%%%%%%%%%%%%%%%%%%%%%%%%%%%%%%%%%%%%%%%%%%%%%%%%%%%%%%%%%%%%%%%%%%
% Integral representatives for Selmer elements
%%%%%%%%%%%%%%%%%%%%%%%%%%%%%%%%%%%%%%%%%%%%%%%%%%%%%%%%%%%%%%%%%%%%%%

\section{Integral representatives for Selmer elements} \label{sec:integralreps}

In this section, we describe how to find representatives in $V(\Z)$ for rational orbits with integral invariants.  More precisely, given $v' \in V^\ls(\Q)$ such that the $G(\Q)$-invariants of $v'$ are integers, we show that there exists an element $v \in V(\Z)$ with the same invariants, up to absolutely bounded factors.  To find such an integral representative $v$, we first consider the question locally, i.e., we find an element of $V(\Zp)$ in the same $G(\Qp)$-orbit as $v'$.  Then strong approximation implies that these local representatives may be glued together to obtain an integral element of $V(\Z)$. Our method to find $v$ will heavily rely on the parametrizations of \S \ref{sec:parametrizations}.

Let $v'$ be an element of $V^\ls(\Q)$ such that the $G(\Q)$-invariants of $v'$ are integral, i.e., the generators of the $G(\Q)$-invariant ring, when evaluated on $v'$, are integral.  By Theorem \ref{thm:parametrizations}, the $G(\Q)$-orbit of $v'$ corresponds to (the equivalence class of) an elliptic curve $E$ in the family $F$ and a torsor $C$ for $E$ with a line bundle $L$ of degree $d$ on $C$.  In particular, the elliptic curve $E$ has an affine integral Weierstrass model $\mathcal{E}^\circ$, as an element of $F$ with coefficients equal to the invariants of $v'$. 

Let $p$ be any prime.  By assumption, the genus one curve $C$ has a $\Qp$-point, so $C$ is isomorphic to its Jacobian $E$ over $\Qp$. We may write the line bundle $L_{\Qp}$ on $C_{\Qp} \cong E_{\Qp}$ as $\mathcal{O}(D)$ for a divisor $D = (d-1) \cdot O + Q$ with $Q \in E(\Qp)$.  There exists an automorphism of $C_{\Qp}$ (as a genus one curve) taking $D = (d-1) \cdot O + Q$ to $(d-1) \cdot O + (Q + d Q')$ for any point $Q' \in E(\Qp)$, so for the equivalence class of $(C,L)$, the point $Q$ in the divisor $D$ is only determined up to $d E(\Qp)$. We claim that we may choose a representative of the $dE(\Qp)$-coset of any $Q\in E(\Qp)$ that is either the point $O$ at infinity or (almost) integral as a point on $\mathcal{E}^\circ(\Qp)$.

\begin{lemma} \label{lem:almostintegral}
Let $p$ be a prime and let $d = 2$ or $3$. Let $E$ be an elliptic curve over $\Qp$ with an affine integral Weierstrass model $\mathcal{E}^\circ$ $($with coefficients in $\Zp$$)$.
\begin{enumerate}
\item[{\rm (a)}] If $p \nmid d$, then every coset $E(\Qp)/dE(\Qp)$ has a representative that is either the point $O$ at infinity or a point $(x,y) \in \mathcal{E}^\circ(\Qp)$ with $x, y \in \Zp$.
\item[{\rm (b)}] If $p = d = 2$, then every coset $E(\Qp)/dE(\Qp)$ has a representative that is either $O$ or a point $(x,y) \in \mathcal{E}^\circ(\Qp)$ with $2^4 x, 2^6 y \in \Zp$.
\item[{\rm (c)}] If $p = d = 3$, then every coset $E(\Qp)/dE(\Qp)$ has a representative that is either $O$ or a point $(x,y) \in \mathcal{E}^\circ(\Qp)$ with $3^2 x, 3^3 y \in \Zp$.
\end{enumerate}
\end{lemma}

\begin{proof}
Let $\mathcal{E}$ denote the projective closure of the affine Weierstrass model $\mathcal{E}^\circ$ in $\P^2$, and let $\widetilde{\mathcal{E}}$ denote the reduction of $\mathcal{E}$ modulo $p$.

If $p \nmid d$, then we only need consider the standard reduction map
$$\mathcal{E}(\Qp) \to \widetilde{\mathcal{E}}(\Fp).$$
All of the points in $\mathcal{E}(\Qp)$ that are not in the preimage of the reduction of $O$ are integral (i.e., the coordinates on $\mathcal{E}^\circ$ are in $\Zp$) by definition of the reduction map. On the other hand, the kernel consists of exactly the $\Zp$-points of the formal group $\hat{E}$ for $E$, which is $d$-divisible (in itself), and all of these points are in the $dE(\Qp)$-coset of $O$.

If $p = d$, then this kernel is no longer $d$-divisible, but we only need to slightly modify this argument to allow for bounded denominators. For positive integers $i$, let $\hat{E}_i$ denote the $\Z/p^i\Z$-points of the formal group $\hat{E}$, so the canonical reduction maps $\hat{E}_{i+1} \to \hat{E}_i$ are surjective with kernels $\Z/p\Z$. Note that $\hat{E}(\Zp) = \varprojlim \hat{E}_i$, and $\hat{E}_1$ is trivial by definition. Let $K_i$ denote the kernel of the surjective map $\hat{E}(\Zp) \to \hat{E}_i$, which corresponds precisely to the subset of points in $\mathcal{E}(\Qp)$ that reduce to $O$ modulo $p^i$. Let $\nu_p$ denote the $p$-adic valuation of an element of $\Qp$. If $(x,y) \in \mathcal{E}^\circ(\Qp)$ with $\nu_p(y) \leq -3i$, then $\nu_p(x) = \frac{2}{3} \nu_p(y) \leq -2i$ and the point $(x,y)$ corresponds to a point in $K_i$. Therefore, for all $(x,y) \not\in K_i$, we have that 
\begin{equation} \label{eq:valpxy}
\nu_p(x) \geq - 2(i-1)  \qquad \mathrm{ and } \qquad \nu_p(y) \geq - 3(i-1)
\end{equation}
since valuations are integers and $3\nu_p(x) = 2\nu_p(y)$ when either valuation is negative.

The argument above for $p \nmid d$ uses the fact that $K_1 = \hat{E}(\Zp)$ is $d$-divisible. For $p = d$, we claim that $K_i$ is $p$-divisible in $K_{i-1}$ for $i \geq 3$ when $p = 2$ and $i \geq 2$ when $p = 3$. This follows from the fact that for all such $i$, the formal logarithm induces an identification of $K_{i-1}$ with $p^{i-1}\Zp$ in a manner compatible with inclusion as $i$ changes \cite[Theorem IV.6.4(b)]{Sil}.
Thus, $K_i$ is $p$-divisible in $K_{i-1}$, so the points in $K_i$ are in the coset of $O$ in $E(\Qp)/dE(\Qp)$.

Therefore, for $p = d = 2$ or $3$, we see that the natural map
$$(\mathcal{E}(\Qp) \setminus K_i) \cup \{O\} \hookrightarrow E(\Qp) \to E(\Qp) / p E(\Qp)$$
is surjective for $i = 3$ or $2$, respectively. When $p = 2$, the inequalities \eqref{eq:valpxy} imply that $2^4 x, 2^6 y \in \Z_2$ for $(x,y) \in \mathcal{E}^\circ(\Q_2)$ not in $K_3$; similarly, for $p = 3$, all $(x,y) \not\in K_2$ have $3^2 x, 3^3 y \in \Z_3$, as desired.
\end{proof}

For the remainder of this section, we call a point $(x,y) \in \mathcal{E}^\circ(\Qp)$ as in Lemma \ref{lem:almostintegral} {\em almost integral}. 
Given the elliptic curve $E$ over $\Qp$ with affine integral model $\mathcal{E}^\circ$ and the line bundle $d \cdot O$ or $(d-1) \cdot O + Q$, where $Q$ is an almost integral point of $\mathcal{E}^\circ(\Qp)$, we need to exhibit a corresponding element of $V(\Zp)$ (or more precisely, an element of $V(\Qp)$ with bounded denominators at $2$ and $3$). We list below explicit such elements in each case; note that we may control the powers of $2$ and $3$ in the denominators of all the entries:

\begin{enumerate}
\item[1.] {\bf Binary quartic forms.} We slightly modify the formulas found in \cite[Section 3.2]{CFS}.  Let $\mathcal{E}$ be the elliptic curve $y^2 = x^3 + a_4 x + a_6$, where $a_4, a_6 \in \Zp$ and the discriminant $\Delta = -16 (4 a_4^3 + 27 a_6^2)$ is nonzero.
	\begin{enumerate}
	\item Suppose $L \cong \mathcal{O}(O+Q)$ for an almost integral $Q = (x_1, y_1) \in \mathcal{E}^\circ(\Qp)$.  Set the
	binary quartic form
		\begin{equation} \label{eq:BQformula-OQ}
			f(w_1,w_2) = \frac{1}{4} w_1^4 - \frac{3}{2} x_1 w_1^2 w_2^2 - 2 y_1 w_1 w_2^3 + (-\frac{3}{4} x_1^2 - a_4) w_2^4.
		\end{equation}
	\item Suppose $L \cong \mathcal{O}(2O)$.  Set the binary quartic form
		\begin{equation} \label{eq:BQformula-2O}
			f(w_1,w_2) = w_1^3 w_2 + a_4 w_1 w_2^3 + a_6 w_2^4.
		\end{equation}
	\end{enumerate}
	In both cases, the binary quartic form $f$ determines a genus one curve isomorphic to $\mathcal{E}$ (i.e., the normalization of $z^2 = f(w_1,w_2)$),
	along with the appropriate line bundle: in the first case, $\mathcal{O}(O+Q)$ where $O$ and $Q$ are taken to
	$(w_1, w_2, z) = (1, 0, \frac{1}{2})$ and $(1,0,-\frac{1}{2})$, respectively, and in the second case, 
	$\mathcal{O}(2O)$ where $O$ is taken to the point $(z,w_1,w_2) = (0,1,0)$ on this model.  
	The usual degree $2$ and $3$ invariants $I$ and $J$ of $f$ are $-3 a_4$ and $-27 a_6$, respectively, implying
	that the Jacobian of these models is the original elliptic curve.
\item[2.] {\bf Ternary cubic forms.} Again, we modify the formulas found in \cite[Section 3.2]{CFS}.  Let $\mathcal{E}$ be the elliptic curve given by $y^2 = x^3 + a_4 x + a_6$, where $a_4, a_6 \in \Zp$ and the discriminant is nonzero.
	\begin{enumerate}
	\item Suppose $L \cong \mathcal{O}(2O+Q)$ for an almost integral $Q = (x_1, y_1) \in \mathcal{E}^\circ(\Qp)$. The ternary cubic form
		\begin{equation} \label{eq:TCformula-2OQ}
			Y^2 Z - X^2 Y + 3 x_1 Y Z^2 + 2 y_1 X Z^2 + (3 x_1^2 + a_4) Z^3
		\end{equation}
	where the pullback of $\mathcal{O}(1)$ from $\mathbb{P}^2$ is isomorphic to $\mathcal{O}(2O+Q)$, where $O$ and $Q$ are taken to the points
	$[X,Y,Z] = [0,1,0]$ and $[1,0,0]$, respectively.
	\item Suppose $L \cong \mathcal{O}(3O)$.  Then the ternary cubic form
		\begin{equation} \label{eq:TCformula-3O}
			- Y^2 Z + X^3 + a_4 X Z^2 + a_6 Z^3
		\end{equation}
		clearly cuts out a curve in $\mathbb{P}^2$ that is isomorphic to $\mathcal{E}$, and the pullback of $\mathcal{O}(1)$
		from $\mathbb{P}^2$ is isomorphic to $\mathcal{O}(3O)$, where $O$ is taken to the point $[X,Y,Z] = [0,1,0]$.
	\end{enumerate}
\item[3.] {\bf Bidegree $(2,2)$ forms.} Let $\mathcal{E}$ be the elliptic curve $y^2 + a_3 y = x^3 + a_2 x^2 +a_4 x$ with the point $P = (0,0)$ on $\mathcal{E}^\circ(\Qp)$.  We first show that there exist divisors for both the line bundle $L$ and the bundle $L \otimes P$ that are the sum of almost integral points and/or the point $O$ at infinity.

We may assume from the earlier argument that the line bundle $L$ either is $\mathcal{O}(O+Q)$ for some almost integral point $Q = (x_1, y_1)$ on $\mathcal{E}^\circ(\Qp)$ or is $\mathcal{O}(2O)$.  In the former case, the line bundle $L \otimes P$ then is isomorphic to $\mathcal{O}(O+Q')$, where $Q' = (x_2, y_2)$ is the sum of $Q$ and $P$; we have
\begin{align} \label{eq:coordsforadding00}
x_2 = \frac{y_1^2}{x_1^2} - a_2 - x_1 && \mathrm{and} && y_2 = -\frac{y_1^3}{x_1^3} + a_2 \frac{y_1}{x_1} + y_1 - a_3.
\end{align}
If $Q'$ is not in $2 \mathcal{E}(\Qp)$, then $Q'$ is also almost integral as a point on $\mathcal{E}^\circ(\Qp)$, which implies that $\frac{y_1}{x_1} \in \Zp$ if $p \neq 2$ and $\frac{y_1}{x_1} \in \frac{1}{4}\Z_2$ if $p = 2$.  Otherwise, we may translate the elliptic curve by half of $Q'$ to obtain isomorphisms $L \otimes P \cong \mathcal{O}(2O)$ and $L \cong \mathcal{O}(O + [-P])$; note that $-P = (0,-a_3)$ is also integral.  We may use the automorphism of $\mathcal{E}$ fixing $O$ and sending $P$ to $-P$ to reduce to the case where $L \cong \mathcal{O}(2O)$ (though the roles of $L$ and $L \otimes P$ will be switched).  We thus reduce to two cases:
	\begin{enumerate}
	\item Suppose the line bundle $L$ is isomorphic to $\mathcal{O}(O+Q)$, and $L \otimes P$ is isomorphic to $\mathcal{O}(O+Q')$, where $Q = (x_1,y_1)$ and $Q' = (x_2,y_2) = Q + P$ are almost integral points on $\mathcal{E}^\circ(\Qp)$.  The coordinates $x_2$ and $y_2$ are computed by \eqref{eq:coordsforadding00}, and we have that $y_1 / x_1 \in \Zp$ (or $\frac{1}{4}\Z_2$).  Then the following bidegree $(2,2)$ form, with coordinates $([w_1,w_2],[z_1,z_2])$, recovers the elliptic curve $\mathcal{E}$ with the same line bundles, while having the correct polynomial invariants:
	\begin{equation}
	\begin{pmatrix} w_1^2 & w_1 w_2 & w_2^2 \end{pmatrix} \cdot
	\begin{pmatrix}
		0 & \frac{1}{2} & -\frac{y_1}{2 x_1} \\
		-\frac{1}{2} & \frac{y_1}{x_1} & -\frac{a_2}{2}-x_1-\frac{x_2}{2} \\
		-\frac{y_1}{2x_1} & \frac{a_2}{2} + \frac{x_1}{2} + x_2 & \frac{y_2 - y_1}{2}
	\end{pmatrix}
	\cdot \begin{pmatrix} z_1^2 \\ z_1 z_2 \\ z_2^2 \end{pmatrix}.
	\end{equation}
	\item Suppose the line bundle $L$ is isomorphic to $\mathcal{O}(2O)$, and $L \otimes P$ is isomorphic to $\mathcal{O}(O+P)$.
	Then the bidegree $(2,2)$ form below
	\begin{equation}
	\begin{pmatrix} w_1^2 & w_1 w_2 & w_2^2 \end{pmatrix} \cdot
	\begin{pmatrix}
		0 & 0 & -\frac{1}{2} \\
		\frac{1}{2} & 0 & -\frac{a_2}{2} \\
		0 & -\frac{a_3}{2} & -\frac{a_4}{2}
	\end{pmatrix}
	\cdot \begin{pmatrix} z_1^2 \\ z_1 z_2 \\ z_2^2 \end{pmatrix}
	\end{equation}
	has the desired properties.
	\end{enumerate}

\item[4.] {\bf Rubik's cubes.} The geometric data here is almost identical to bidegree $(2,2)$ forms (case 3), except that the degree of the line bundle is different.  Let $\mathcal{E}$ be the elliptic curve $y^2 + a_3 y = x^3 + a_2 x^2 + a_4 x$ with the point $P = (0,0)$.  By the same arguments as before, there are two cases:
	\begin{enumerate}
	\item Suppose the line bundle $L$ is isomorphic to $\mathcal{O}(2O+Q)$ and $L \otimes P$ is isomorphic to $\mathcal{O}(2O+Q')$, where $Q = (x_1,y_1)$ and 
	$Q' = (x_2,y_2) = Q + P$ are almost integral points on $\mathcal{E}^\circ(\Qp)$.  Recall that $y_1/x_1$ is integral (or in $\frac{1}{3}\Z_3$ for $p = 3$). Then the following element of 
	$\Q_p^3 \otimes \Q_p^3 \otimes \Q_p^3$ recovers this geometric data with the correct polynomial invariants:
	\begin{equation}
	\begin{pmatrix}
	y_2 - y_1 & x_2 - x_1 & 0 \\
	x_2 - x_1 & 0 & 1 \\
	0 & -1 & 0	
	\end{pmatrix} \qquad  
	\begin{pmatrix}
	x_2 - x_1 & 0 & 1 \\
	0 & 0 & 0 \\
	-1 & 0 & 0
	\end{pmatrix} \qquad 
	\begin{pmatrix}
	-a_2 - 2 x_1 - x_2 & y_1 / x_1 & 0 \\
	- y_1 / x_1 & 1 & 0 \\
	-1 & 0 & 0
	\end{pmatrix}
	\end{equation}
	where these three matrices represent the $3 \times 3 \times 3$ array of coefficients.
	One checks that two of the ternary cubics obtained from this element are
	$$Y^2 Z  - X^2 Y + (a_2 + 3 x_i) Y Z^2 + (a_3 + 2 y_i) X  Z^2 +  (a_4 + 2 a_2 x_i + 3 x_i^2) Z^3$$
	for $i = 1$ and $2$, which are the embeddings of $\mathcal{E}$ by $\mathcal{O}(2O +Q)$ and $\mathcal{O}(2O+Q')$, respectively.
	\item Suppose the line bundle $L$ is isomorphic to $\mathcal{O}(3O)$ and $L \otimes P \cong \mathcal{O}(2O+P)$.
	Then the following element of $\ZZp^3 \otimes \ZZp^3 \otimes \ZZp^3$ recovers $\mathcal{E}$ with these line bundles and the correct polynomial invariants:
	\begin{equation} \label{eq:RCformula-3O}
	\begin{pmatrix}
	1 & 0 & 0 \\
	0 & 0 & -1 \\
	0 & 0 & -a_3	
	\end{pmatrix} \qquad  
	\begin{pmatrix}
	0 & 1 & 0 \\
	-1 & 0 & 0 \\
	0 & a_2 & a_4
	\end{pmatrix} \qquad 
	\begin{pmatrix}
	0 & 0 & 1 \\
	0 & 0 & 0 \\
	0 & -1 & 0
	\end{pmatrix}
	\end{equation}
	where these three matrices represent the $3 \times 3 \times 3$ array of coefficients.  
	It is easy to check that two of the ternary cubics obtained from this Rubik's cube are
	exactly those giving the embeddings of $\mathcal{E}$ by $\mathcal{O}(3O)$ and $\mathcal{O}(2O+P)$.
	\end{enumerate}
\item[5.] {\bf Doubly symmetric Rubik's cubes.} Let $\mathcal{E}$ be the elliptic curve $y^2 = x^3 + a_2 x^2 + a_4 x$ with the $2$-torsion point $P = (0,0)$.  We again have two cases, where either $L \cong \mathcal{O}(2O+Q)$, for an almost integral point $Q = (x_1,y_1)$ or $L \cong \mathcal{O}(3O)$.
	\begin{enumerate}
	\item Suppose $L \cong \mathcal{O}(2O+Q)$ for an almost integral point $Q = (x_1, y_1)$ on $\mathcal{E}^\circ(\Qp)$,
	and recall that we may assume $y_1 / x_1$ is also integral (or in $\frac{1}{3}\Z_3$ for $p = 3$).  Then the following
	element of $\Q_p^3 \otimes \Sym_2(\Q_p^3)$, represented as a triple of symmetric $3 \times 3$ matrices, has (almost) integral coefficients
	and the desired minimal invariants $a_2$ and $a_4$:
	\begin{equation}
	\begin{pmatrix}
	0 & 1 & 0 \\
	1 & 0 & 0 \\
	0 & 0 & 0
	\end{pmatrix} \qquad
	\begin{pmatrix}
	-x_1 & 0 & 0 \\
	0 & 1 & -1 \\
	0 & -1 & 1
	\end{pmatrix} \qquad
	\begin{pmatrix}
	-a_2 - x_1 & \frac{y_1}{x_1} & - \frac{y_1}{x_1} \\
	\frac{y_1}{x_1} & 0 & 0 \\
	- \frac{y_1}{x_1} & 0 & -1
	\end{pmatrix}.
	\end{equation}
	The ternary cubic form $$-X^2 Y + X^2 Z + x_1 Y^2 Z + (a_2 + x_1) Y Z^2 + 2 \frac{y_1}{x_1} X Z^2 + \frac{y_1^2}{x_1^2} Z^3$$ obtained
	from this element is $\PGL_3(\Qp)$-equivalent to the ternary cubic form
	$$ Y^2 Z -X^2 Y + (a_2 + 3 x_1) Y Z^2 + 2 y_1 X Z^2 + (a_4 + 2 a_2 x_1 + 3 x_1^2) Z^3,$$
	which by \eqref{eq:TCformula-2OQ} corresponds to the embedding of $\mathcal{E}$ via the line bundle $\mathcal{O}(2O+Q)$. 
	Therefore, the genus one curve and degree $3$ line bundle obtained via Theorem \ref{thm:parametrizations} is exactly $\mathcal{E}$ with
	the line bundle $\mathcal{O}(2O+Q)$.
	\item Suppose $L \cong \mathcal{O}(3O)$. Then the following element of $\ZZp^3 \otimes \Sym_2(\ZZp^3)$, 
	represented as a triple of symmetric $3 \times 3$ matrices, has the desired minimal invariants $a_2$ and $a_4$:
	\begin{equation} \label{eq:2SymRCformula-3O}
	\begin{pmatrix}
	a_6 & 0 & 1 \\
	0 & 1 & 0 \\
	1 & 0 & 0	
	\end{pmatrix} \qquad  
	\begin{pmatrix}
	0 & 1 & 0 \\
	1 & 0 & 0 \\
	0 & 0 & 0
	\end{pmatrix} \qquad 
	\begin{pmatrix}
	a_4 & 0 & 0 \\
	0 & 0 & 0 \\
	0 & 0 & -1
	\end{pmatrix}
	\end{equation}
	The genus one curve and degree $3$ line bundle obtained via Theorem \ref{thm:parametrizations} corresponds to the ternary cubic form
	$X^3 + a_2 X^2 Z - Y^2 Z + a_4 X Z^2$, which by \eqref{eq:TCformula-3O} is exactly $\mathcal{E}$ embedded using
	the line bundle $\mathcal{O}(3O)$.
	\end{enumerate}

\item[6.] {\bf Triply symmetric hypercubes.} Let $\mathcal{E}$ be the elliptic curve $y^2 + a_1 x y + a_3 y = x^3$; note that the point $P = (0,0)$ on $\mathcal{E}^\circ$ has order $3$ and is an element of $2\mathcal{E}(\Qp)$.  Just as before, we may reduce to the cases where $L$ is isomorphic to $\mathcal{O}(O + Q)$, for some almost integral point $Q = (x_1,y_1)$ not equal to $O$ or $P$, or where $L$ is isomorphic to $\mathcal{O}(2O)$.
	\begin{enumerate}
	\item For $L \cong \mathcal{O}(O+Q)$, with $Q = (x_1, y_1)$, the pair of binary cubic forms with coefficients
		\begin{equation} \label{eq:3symhypminfirst}
		(y_1, 0, 0, 1) \qquad \textrm{and}  \qquad (0, -y_1, -x_1, -a_1)
		\end{equation}
	is close to the desired integral element, but the invariants of this pair have extra factors of $2$ and $y_1$ (namely, $2^i y_1^i a_i$ instead of $a_i$).
	
	By applying transformations of the form $\left(\begin{smallmatrix} p^\alpha & 0 \\ 0 & p^\beta \end{smallmatrix}\right)$ in the two copies of $\GL_2(\Qp)$ to \eqref{eq:3symhypminfirst},
	we obtain the following element with the desired minimal invariants, up to powers of $2$ and units in $\Zp$:
		\begin{equation} \label{eq:3symhypmin}
		(y_1 p^{-v}, 0, 0, p^{3w-v}) \qquad \textrm{and} \qquad (0,-y_1p^{-2w}, -x_1 p^{-w}, -a_1)
		\end{equation}
	where $w$ is the $p$-adic valuation of $x_1$ and $v$ is the $p$-adic valuation of $y_1$. Because $y_1$ divides $x_1^3$, we have that $3w -v$ is nonnegative, and all of the coordinates
	of \eqref{eq:3symhypmin} are integral. Dividing \eqref{eq:3symhypmin} by $2 y_1 p^{-v}$, which is a unit in $\Qp$ when $p \neq 2$ or $3$, produces a pair of binary cubic forms with
	fundamental polynomial invariants $a_1$ and $a_3$ and gives rise to the desired binary quartic form
	corresponding to the elliptic curve $\mathcal{E}$ with the marked point $P$ and the line bundle $L$. (If $p = 2$ or $3$, then since the valuation of $2y_1 p^{-v}$ is bounded, 
	we still obtain an element of $V(\Qp)$ with bounded denominators.)
	
	\item For $L \cong \mathcal{O}(2O)$, the pair of binary cubics
		\begin{equation}
		  (1, 0, 0, a_3) \qquad \textrm{and} \qquad (0, 0, 1, a_1)
		\end{equation}
	is an element of $V(\Zp)$ with the invariants $a_1$ and $a_3$, and via Theorem \ref{thm:parametrizations}
	and \eqref{eq:BQformula-2O}, it is easy to compute that
	it gives rise to the elliptic curve $\mathcal{E}$ with the marked point $P$ and the line bundle $\mathcal{O}(2O)$.
	\end{enumerate}

\item[7.] {\bf Hypercubes.} Let $\mathcal{E}$ be the elliptic curve $y^2 + a_1 x y + a_3 y = (x-a_2)(x-a_2')(x-a_2'')$ with the three points $(a_2,0)$, $(a_2',0)$, and $(a_2'',0)$ summing to the identity point. Again, we consider the two cases where $L \cong \O(O+Q)$, for an almost integral point $Q = (x_1, y_1) \in \mathcal{E}^\circ(\Qp)$, or $L \cong \mathcal{O}(2O)$.
	\begin{enumerate}
	\item Suppose $L \cong \O(O+Q)$ for an almost integral point $Q = (x_1, y_1)$ of $\mathcal{E}^\circ(\Qp)$. Then the following element of $\Q_p^2 \otimes \Q_p^2 \otimes \Q_p^2 \otimes \Q_p^2$, represented as a $2 \times 2$ matrix of $2 \times 2$ matrices, gives the elliptic curve $\mathcal{E}$ (up to isomorphism) and the line bundle $L$:
	\begin{equation} \label{eq:hypminfirst}
	\begin{array}{cc}
	\left(
		\begin{array}{cc}
		y_1 & 0 \\
		0 & 0
		\end{array}
	\right) &
	\left(
		\begin{array}{cc}
		0 & 0 \\
		0 & 1
		\end{array}
	\right) \\[.175in]
	\left(
		\begin{array}{cc}
		0 & -y_1 \\
		-y_1 & a_2 - x_1
		\end{array}
	\right) &
	\left(
		\begin{array}{cc}
	-y_1 & a_2' - x_1 \\
	a_2'' - x_1 & -a_1
		\end{array}
	\right)
	\end{array}
	\end{equation}
	In other words, the natural binary quartic arising from \eqref{eq:hypminfirst}, via Theorem \ref{thm:parametrizations}, corresponds 
	to the map of $\mathcal{E}$ to $\P^1$ via the sections of $L$.
	The invariants coming from \eqref{eq:hypminfirst} are not quite the desired minimal invariants $a_1, a_2, a_2', a_2'', a_3$, however, 
	as they are scaled by powers of $y_1$ (and $2$); namely, instead of $a_i$, we obtain $2^i y_1^i a_i$.
	By applying transformations of the form $\left(\begin{smallmatrix} p^\alpha & 0 \\ 0 & p^\beta \end{smallmatrix}\right)$ in the four copies of $\GL_2(\Qp)$ to \eqref{eq:hypminfirst},
	we obtain the following element with the desired minimal invariants, up to powers of $2$ and units in $\Zp$:
	\begin{equation*}
	\begin{array}{cc}
	\left(
		\begin{array}{cc}
		y_1 p^{-v} & 0 \\
		0 & 0
		\end{array}
	\right) &
	\left(
		\begin{array}{cc}
		0 & 0 \\
		0 & p^{w+w'+w''-v}
		\end{array}
	\right) \\[.175in]
	\left(
		\begin{array}{cc}
		0 & -y_1 p^{-w-w'} \\
		-y_1 p^{-w-w''} & (a_2 - x_1) p^{-w}
		\end{array}
	\right) &
	\left(
		\begin{array}{cc}
	-y_1 p^{-w'-w''} & (a_2' - x_1)p^{-w'} \\
	(a_2'' - x_1)p^{-w''} & -a_1
		\end{array}
	\right)
	\end{array}
	\end{equation*}
	where $v$ is the $p$-adic valuation $v_p(y_1)$ and $w = v_p(a_2 - x_1)$, $w' = v_p(a_2'-x_1)$, and $w'' = v_p(a_2''-x_1)$. Because $y_1$ divides $(x_1-a_2)(x_1-a_2')(x_1-a_2'')$,
	the inequality $v \leq w + w' + w''$ holds, and every entry of the above array is indeed integral.
	Dividing the top two $2 \times 2$ matrices by $2 y_1 p^{-v}$, which is a unit in $\Qp$ when $p \neq 2$ or $3$, gives a hypercube with the precise invariants $a_1, a_2, a_2', a_2'', a_3$.
	(If $p = 2$ or $3$, like in case 4, we see that the $p$-adic valuation of $2 y_1 p^{-v}$ is still bounded, so the same hypercube still gives the correct invariants and has bounded denominators.)

	\item Suppose $L \cong \O(2O)$. Then the following element of $\ZZp^2 \otimes \ZZp^2 \otimes \ZZp^2 \otimes \ZZp^2$,
	represented as a $2 \times 2$ matrix of $2 \times 2$ matrices, recovers the elliptic curve $\mathcal{E}$ and $L$ and has the desired invariants:
	\begin{equation} \label{eq:hypmin2O}
	\begin{array}{cc}
	\left(
		\begin{array}{cc}
		1 & 0 \\
		0 & -a_2
		\end{array}
	\right) &
	\left(
		\begin{array}{cc}
		0 & -a_2' \\
		-a_2'' & a_3
		\end{array}
	\right) \\[.175in]
	\left(
		\begin{array}{cc}
		0 & 0 \\
		0 & 1
		\end{array}
	\right) &
	\left(
		\begin{array}{cc}
		0 & 1 \\
		1 & a_1
		\end{array}
	\right)
	\end{array}
	\end{equation}
	The binary quartic obtained from \eqref{eq:hypmin2O} via Theorem \ref{thm:parametrizations} is
	  $$w_1^3 w_2 + \frac{1}{4} a_1^2 w_1^2 w_2^2 + (\frac{1}{2} a_1 a_3 + a_2 a_2' + a_2' a_2'' + a_2 a_2'') w_1 w_2^3 + (-a_2 a_2' a_2'' + \frac{a_3^2}{4}) w_2^4,$$
	which matches \eqref{eq:BQformula-2O} in this case after changing the equation of $\mathcal{E}$ into short Weierstrass form and applying a change of variables to the quartic.
	\end{enumerate}

\end{enumerate}

Therefore, in each case, from an element of $V^\ls(\Qp)$, we obtain a $G(\Qp)$-equivalent element of $V^\ls(\Qp)$ with absolutely bounded denominators and the same invariants. In other words, we may find a $G(\Qp)$-equivalent {\em integral} element in $V(\Zp)$ with invariants up to absolutely bounded factors of $2$ and $3$. In each case, a standard argument on strong approximation (see, e.g., \cite[Lemmas 3.2 and 3.3] {fisher-minbqtc} and \cite[Theorem 4.17]{CFS}) allows us to patch together these local integral representatives into an integral model over $\Q$:

\begin{theorem}
For each of the cases in Table $\ref{table:Invariants}$, for an element of $V^\ls(\Q)$ such that its $G(\Q)$-invariants are in $\Z$, there is a $G(\Q)$-equivalent element of $V^\ls(\Z)$ with the same invariants, up to absolutely bounded factors of $2$ and $3$.
\end{theorem}

The exact factors of $2$ and $3$ differ for each case but may be explicitly computed by the formulas in this section, though they are not needed for the later counting arguments. We obtain the following:

\begin{corollary}
Consider any of the cases in Table $\ref{table:Invariants}$.  Let $E$ be an elliptic curve in the family $F$ and let $d$ be the degree of the associated line bundle. Then the elements in the $d$-Selmer group $S_d(E)$ of $E$ are in bijection with $G(\Q)$-equivalence classes on the set $V^\ls(\Z)$ of locally soluble elements having invariants equal to the coefficients $M^i a_i$ $($and $M^i a_i')$ of $E$, for some absolutely bounded integer $M$.
\end{corollary}

%%%%%%%%%%%%%%%%%%%%%%%%%%%%%%%%%%%%%%%%%%%%%%%%%%%%%%%%
%%%%%%%%%%%%%%%%%%%%%%%%%%%%%%%%%%%%%%%%%%%%%%%%%%%%%%%%
%%%%%%%%%%%%%%   Fundamental domains      %%%%%%%%%%%%%%
%%%%%%%%%%%%%%%%%%%%%%%%%%%%%%%%%%%%%%%%%%%%%%%%%%%%%%%%
%%%%%%%%%%%%%%%%%%%%%%%%%%%%%%%%%%%%%%%%%%%%%%%%%%%%%%%%

\section{Construction of fundamental domains}

For each case in Table~\ref{table:Invariants}, we wish to count
irreducible $G(\Z)$-orbits on $V(\Z)$ having bounded
height.  We will accomplish this by counting suitable integer points
in fundamental domains for the action of $G(\Z)$ on $V(\R)$.

\subsection{Fundamental sets for the action of $G(\R)$ on $V(\R)$}

In this subsection, we would like to find a bounded fundamental set for the action of the group $G(\R) \times \R^\times_{> 0}$ on $V(\R)$, or equivalently, a bounded fundamental set for the action of $G(\R)$ on the elements of height at most $1$ in $V(\R)$. We break up $V^\s(\R)$, considered with its real topology, into its connected components $V^{(i)}$ for $i\in\{1,\ldots,N\}$.  We prove the following:

\begin{theorem} \label{thm:fundsets}
There exists an absolutely bounded fundamental set for the action of the group $G(\R) \times \R^\times_{>0}$ on each component $V^{(i)}$.
\end{theorem}

A point $v \in V^\s(\R)$ is called {\em $\R$-soluble} (resp., {\em $\R$-insoluble}) if the genus one curve $C$ arising from $v$ under the parametrization theorems of \S \ref{sec:parametrizations} has a real point (resp., does not have a real point). Note that when $d = 2$, a point $v$ is $\R$-insoluble if and only if the corresponding binary quartic form is negative definite, and for $d = 3$, all points are $\R$-soluble. In each component $V^{(i)}$, the points $v \in V^{(i)}$ are either all $\R$-soluble or all not, and we need only study the $\R$-soluble components.

Let $m$ denote the number of independent polynomial invariants for the action of $G$ on $V$, i.e., the real dimension of $V(\R)/G(\R)$ or the complex dimension of $V(\C)/G(\C)$.  There exists a $G(\R)$-invariant map $\pi: V^{(i)} \to \R^m$ sending an element to its invariant vector $\vec{a}$, or equivalently, a map $\overline{\pi}: V^{(i)}/G(\R) \to \R^m$ from orbits to invariant vectors.

For invariant vectors $\vec a \in \R^m$, we define several properties, such as height and discriminant, in the same way as for the corresponding elliptic curves $E_{\vec a}$ with coefficients $\vec a$ in the family $F$ as in \S \ref{sec:intro}. For example, let the height of $\vec{a} \in \R^m$ be the maximum of the values $|a_j|^{12/j}$, where $j$ denotes the degree of the invariant $a_j$. Let $\R^m_{H  < 1}$ denote the set of nondegenerate invariant vectors of height at most $1$, where nondegenerate means that the corresponding discriminant form is nonzero (so the vector yields a genuine elliptic curve $E_{\vec a}$). It is clear that $\R^m_{H  < 1}$ is a bounded set in $\R^m$.

In \S \ref{sec:integralreps}, for a given vector $\vec{a} \in \R^m$ of invariants, we gave explicit formulas for points $v(\vec a)\in V^{(i)}$ having $\pi(v(\vec a)) = \vec a$, in other words, sections of $\pi$. Note that these $v(\vec a)$ are all $\R$-soluble by construction. Because $\R^m_{H  < 1}$ is bounded, we would like to use the algebraic formulas for these sections to construct bounded fundamental sets in $V^{(i)}$ for the elements of height at most $1$; the only complication is that the fibers of $\overline{\pi}$ are not always a single orbit.

Recall that in the parametrization theorems of \S \ref{sec:parametrizations}, for an element $v \in V^{(i)}$, the invariants $\pi(v)$ precisely correspond to the coefficients of the elliptic curve $E$ arising from $v$. Thus, given a nondegenerate vector $\vec{a}$ of invariants, the $G(\R)$-orbits in $V(\R)$ with invariant vector $\vec{a}$ correspond to isomorphism classes of pairs $(C,L)$ associated to the elliptic curve $E_{\vec a}$ with coefficients $\vec{a}$. For $d = 2$, there are either one or two $G(\R)$-orbits associated to any nondegenerate $\vec{a}$, depending on whether the elliptic curve $E_{\vec a}$ has one or two real components, respectively; since this condition is locally constant, the nonempty fibers of $\overline{\pi}$ are of constant size (either $1$ or $2$) for each $V^{(i)}$. For $d = 3$, the situation is simpler, since there is always only one $G(\R)$-orbit for each nondegenerate invariant vector $\vec{a}$.

In the cases where $d = 3$, because the nonempty fibers of $\overline{\pi}$ have size $1$, we may directly use the algebraic formulas from \S \ref{sec:integralreps} to find representatives $v(\vec{a}) \in V^{(i)}$ for each orbit corresponding to nondegenerate $\vec{a} \in \R^m_{H  <  1}$. For simplicity, we may use the formulas in part (b) of each case, where $L \cong \O(3O)$; it is clear that applying the formulas (which are polynomial in the invariants) to $\R^m_{H < 1}$ produces a bounded fundamental set. In other words, we may break up $G(\R)\backslash V^\s(\R)$ into $R^{(i)}\subset V^{(i)}$ for
$i\in\{1,\ldots,N\}$, where
\begin{equation}\label{eqrij}
   R^{(i)}:=\{\lambda\cdot v(\vec a):\lambda\in\R_{>0},\;\Delta(\vec a)>0\}
\end{equation}
or
\begin{equation}\label{eqrij2}
   R^{(i)}:=\{\lambda\cdot v(\vec a):\lambda\in\R_{>0},\;\Delta(\vec a)<0\}
\end{equation}
in accordance with whether discriminants are positive or negative on
$V^{(i)}$, respectively.

For $d = 2$, an identical argument constructs the sets $R^{(i)}$ for the components $V^{(i)}$ for which the fibers of $\overline{\pi}$ have size $1$, by using the formulas from \S \ref{sec:integralreps} where $L \cong \O(2O)$. However, for the components $V^{(i)}$ where the fibers have size $2$, we need to find a representative $w(\vec{a}) \in V^{(i)}$ that represents the other orbit in $\overline{\pi}^{-1}(\vec{a})$, for any $\vec{a} \in \R^m_{H < 1}$. In terms of the geometric data, these orbits correspond to elliptic curves $E$ with two real components (and a trivial torsor $C$), where the degree $2$ line bundles $L$ are isomorphic to $\O(O+Q)$ for any point $Q$ on the non-identity component of $E(\R)$. Such elliptic curves over $\R$ are exactly those with positive discriminant and may be written in the form
$$E_{I,J}: y^2 = x^3 - \frac{I}{3} x - \frac{J}{27}$$
with $4I^3 - J^2 > 0$ for $I, J \in \R$. The point $Q$ may be taken to be\footnote{One may derive the formula \eqref{eq:Qforfunddomain} for the point $Q$ from the binary quartic forms in $L^{(1)}$ in \cite[Table 1]{BS}, since those quartics yield genus one curves isomorphic to the elliptic curves $E_{1,J}$ with hyperelliptic map given by the line bundle $\O(O+Q)$ for $I = 1$.}
\begin{equation} \label{eq:Qforfunddomain}
(x_1,y_1) = \left(-\frac{\sqrt{I}}{3}, \frac{I^{3/4} \sqrt{2-I^{-3/2} J}}{3\sqrt{3}}\right).
\end{equation}
Note that the positive discriminant of $E$ implies that $-2 < I^{-3/2} J < 2$, so $Q \in E_{I,J}(\R)$.

Thus, in each of the remaining $V^{(i)}$, for a vector $\vec a \in \R^m_{H  <  1}$ where $\overline{\pi}^{-1}(\vec a)$ consists of two orbits, the elliptic curve $E_{\vec a}$ is isomorphic to an elliptic curve $E_{I,J}$ for some $I, J \in \R$ with $4I^3 - J^2 > 0$. A simple change of variables to put $E_{\vec a}$ into short Weierstrass form yields polynomial formulas for $I$ and $J$ in terms of the $a$-invariants. Then using the formulas in \S \ref{sec:integralreps} with $L \cong \O(O+Q)$ for $Q$ as in \eqref{eq:Qforfunddomain}, we obtain $w(\vec a) \in V^{(i)}$ as desired. It remains to check that this process gives a bounded set of $w(\vec a)$ in $V^{(i)}$; the formulas are algebraic in the $a$-invariants but there are some negative exponents. Observe, however, that the only denominators occur in the definition of $y_1$ in \eqref{eq:Qforfunddomain} and in the expression $y_1/x_1$ in some of the formulas. For the first, as noted above, the term under the squareroot in the formula for $y_1$ is bounded between $0$ and $4$, so the possible values for $y_1$ itself are bounded (since $I$ is also bounded). For the second, observe that $y_1/x_1 = -I^{1/4}\sqrt{2-I^{-3/2} J}/\sqrt{3}$, which is also bounded for similar reasons. Therefore, in these cases, the union of the set of $w(\vec{a})$ and the set of $v(\vec{a})$, for $\vec a \in \R^m_{H < 1}$ is a bounded fundamental set for $G(\R) \times \R^\times$ acting on $V(\R)$. For these components $V^{(i)}$, in analogy with \eqref{eqrij} and \eqref{eqrij2}, we thus have 
\begin{equation*} \label{eqrij3}
R^{(i)} := \{ \lambda \cdot v(\vec a): \lambda \in \R_{> 0}, \Delta(\vec a) > 0 \}\cup \{ \lambda \cdot w(\vec a): \lambda \in \R_{> 0}, \Delta(\vec a) > 0 \}
\end{equation*}
as fundamental sets.

Let $R^{(i)}(X)$ be the set of all elements in $R^{(i)}$ having height less than
$X$. Since $H(\lambda\cdot v)=\lambda^{k}H(v)$, we see that the
coefficients of all the elements in $R^{(i)}(X)$ are bounded by $O(X^{1/k})$. Note also that for any $g\in G(\R)$,
the set $g\cdot R^{(i)}$ is also a fundamental domain for the action of
$G(\R)$ on $R^{(i)}$. Moreover, for any compact set $G_0\subset G(\R)$ and $g\in G_0$,
the coefficients of elements in $g\cdot R^{(i)}$ having
height less than $X$ are bounded by $O(X^{1/k})$, where the implied
constant depends only on $G_0$.

Note that in each component, the groups $E[d](\R)$ and $E(\R)/d E(\R)$ are locally constant, where $E$ refers to the elliptic curve arising from a given point $v$. By the parametrization theorems, the generic stabilizer in $G(\R)$ for an element in $V^{(i)}$ corresponds to $E[d](\R)$.  We denote the cardinality of the generic stabilizer in $G(\R)$ for an element in $V^{(i)}$ by $n_i$.

\subsection{Arithmetic reduction theory: fundamental domains for
$G(\Z)$ acting on $V(\R)$}

Let $\FF$ denote a fundamental domain for the left action of $G(\Z)$ on
$G(\R)$ that is Haar-measurable and contained in a standard Siegel set~\cite[\S2]{BoHa}.
As the group $G$ is a finite index quotient of the product of finitely many
$\SL_2$'s or $\SL_3$'s, we may write $\FF$ naturally as a subset of a
product of $\FF_2$'s or $\FF_3's$, where $\FF_2$ (resp., $\FF_3$) denotes a
fundamental domain for the left action of $\SL_2(\Z)$ on $\SL_2(\R)$
(resp., $\SL_3(\Z)$ on $\SL_3(\R)$).  To arrange $\FF$ to lie in a
Siegel domain, we simply take $\FF_2$ and $\FF_3$
to lie in Siegel domains; explicitly, we may take $\FF_2= \{\nu\alpha
\kappa:\nu(x)\in N'(\alpha),\alpha(s)\in A',\kappa\in K\}$, where
\begin{equation*}\label{nak}
N'(\alpha)= \left\{\left(\begin{array}{cc} 1 & {} \\ {x} & 1 \end{array}\right):
        x\in I(\alpha) \right\}    , \;\;
A' = \left\{\left(\begin{array}{cc} s^{-1} & {} \\ {} & s \end{array}\right):
       s\geq \sqrt3/2 \right\}, \;\;
\end{equation*}
and $K$ is as usual the (compact) real orthogonal group ${\rm SO}_2(\R)$;
here $I(\alpha)$ is a union of one or two subintervals of
$[-\frac12,\frac12]$ depending only on the value of $\alpha\in A'$.  
Similarly, $\FF_3=\{\nu\alpha\kappa:\linebreak \nu(x,x',x'')\in N'(\alpha),\,\alpha(t,u)\in A',\,k\in K\}$, 
where
\begin{eqnarray*}
  &K&\!=\!\;\;\;{\mbox{subgroup $\SO_3(\R)\subset\GL_3^+(\R)$
of orthogonal transformations; }} 
\\[.05in]
  &A'&\!=\!\!\;\;\,\,\{\alpha(t,u):t,u>c\},\\
  &&\;\;\;\;\;\;\;\;{\rm where}\;\alpha(t,u)=\left(\begin{array}{ccc}
{t^{-2}u^{-1}} & {} & {} \\ {} & {tu^{-1}} & {} \\ {} & {} & {tu^2}
\end{array}\right);\\[.025in]
&N'&\!=\!\!\;\;\,\,\{\nu(x,x',x''):(x,x',x'')\in I'(a)\},\;\\
&&\;\;\;\;\;\;\;\;{\rm where}\; n(x,x',x'')=\left(\begin{array}{ccc} 
{1} & {} & {} \\ {x} & {1} & {} \\ {x'} & {x''} & {1}
\end{array}\right);
\end{eqnarray*}
here $I'(a)$ is a measurable subset of $[-1/2,1/2]^3$ dependent only
on $\alpha\in A'$, and $c>0$ is an absolute constant.

When $\FF$ is a subset of a product of multiple $\FF_2's$ or
$\FF_3$'s, then we use subscripts to distinguish the coordinates on
these $\FF_2$'s or $\FF_3$'s.  Thus, for example, in the case
$V=2\otimes 2\otimes2\otimes2$, the coordinates on $G(\Z)\backslash
G(\R)$ will be the $s_i$
($i\in\{1,2,3,4\}$), $x_i$ ($i\in\{1,2,3,4\}$), and the $\kappa_i\in
K_i$ ($i\in\{1,2,3,4\}$).

We may now construct fundamental sets for the action of $G(\Z)$ on
$V(\R)$.  Namely, for any $g\in G(\R)$ we see that $\FF hR^{(i)}$ is
the union of $n_i$ fundamental domains for the action of $G(\Z)$ on
$V^{(i)}$; here, we regard $\FF hR^{(i)}$ as a multiset, where the
multiplicity of a point $x$ in $\FF hR^{(i)}$ is given by the
cardinality of the set $\{g\in\FF\,:\,x\in ghR^{(i)}\}$.  (See
\cite[\S2.1]{BS} for a more detailed explanation.) Thus a
$G(\Z)$-equivalence class $x$ in $V^{(i)}$ is represented in this
multiset $\sigma(x)$ times, where
$\sigma(x)=\#\Stab_{G(\R)}(x)/\#\Stab_{G(\Z)}(x)$. In particular,
$\sigma(x)$ for $x\in V^{(i)}$ is always a number between 1 and $n_i$.

For any $G(\Z)$-invariant set $S\subset V^{(i)}\cap V(\Z)$, let $N(S;X)$
  denote the number of $G(\Z)$-equivalence classes of irreducible
  elements $B\in S$ satisfying $H(B)<X$.
  Then we conclude that, for any $h\in G(\R)$, the product ${n_i}\cdot
N(S;X)$ is exactly equal to the number of irreducible integer
points in $\FF hR^{(i)}$ having height less than $X$,
with the slight caveat that the (relatively rare---see
Lemma~\ref{gzbigstab}) points with $G(\Z)$-stabilizers of
cardinality $r$ ($r>1$) are counted with weight $1/r$.

As mentioned earlier, the main obstacle to counting integer points of
bounded height in a single domain $\FF hR^{(i)}$ is that the relevant
region is not bounded, but rather has cusps going off to infinity.  We
simplify the counting in this cuspidal region by ``thickening'' the
cusp; more precisely, we compute the number of integer points of
bounded height in the region $\FF hR^{(i)}$ by averaging over lots of
such fundamental regions, i.e., by averaging over the domains $\FF
hR^{(i)}$ where $h$ ranges over a certain compact subset $G_0\in
G(\R)$.  The method, which is an extension of the method
of~\cite{dodpf}, is described in Section~7.

\section{Some sufficient conditions for reducibility in $V(\Z)$}

A simple sufficient condition for a binary quartic form
$ax^4+bx^3y+cx^2y^2+dxy^3+ey^4$ to be reducible is that its
coefficient $a$ of $x^4$ is $0$.  Similarly, a ternary cubic form $f(x,y,z)$ is reducible if the coefficients of $x^3, x^2y,$ and $xy^2$ all
simultaneously vanish, or if the coefficients of $x^3, x^2y,$ and $x^2z$ all
simultaneously vanish;  indeed, in both of these cases, we see that $[1:0:0]$ is then a flex of $f$, i.e., both $f$ and the Hessian of $f$ vanish at $[1:0:0]$.

For each of the representations $V$ in Cases 3--7 of Table \ref{table:Invariants}, we now provide some analogous
sufficient conditions which guarantee that
a point in $V(\Q)$ is reducible.  We begin with the space of hypercubes $(b_{ijk\ell})$ 
($i,j,k,\ell\in\{1,2\}$) in $V(\Q)=\Q^2\otimes\Q^2\otimes\Q^2\otimes\Q^2$.  Note that $V(\Q)$ has a natural action by $S_4$ given by permuting the tensor factors.

\begin{lemma}\label{hyperred}
Let $B=(b_{ijk\ell})$ be an element in $V(\Q)=\Q^2\otimes\Q^2\otimes\Q^2\otimes\Q^2$ such that, after a suitable action by an element of $S_4$, all the coordinates in at least one of the following sets vanish:
\begin{itemize}
\item[{\rm (i)}]
$\{b_{1111},b_{1112},b_{1121},b_{1122}\}$
\item[{\rm (ii)}]
$\{b_{1111},b_{1112},b_{1121},b_{1211}\}$
\end{itemize}
Then $B$ is reducible.
\end{lemma}

\begin{proof}
In both cases (i) and (ii), one checks that $y$ is a linear factor of $\Disc(b_{1ijk}x+b_{2ijk}y)$, and hence $B$ is reducible.
\end{proof}

Note that the two spaces $V(\Q)=\Q^2\otimes\Sym_3\Q^2$ and $V(\Q)=\Sym^2\Q^2\otimes\Sym^2\Q^2$ may be viewed as the spaces of triply symmetric and doubly doubly symmetric hypercubes $(b_{ijk\ell})$ over $\Q$, respectively.  Thus the reducibility criteria given in Lemma~\ref{hyperred} apply to these spaces too, by viewing each as a subspace of the space $\Q^2\otimes\Q^2\otimes\Q^2\otimes\Q^2$ of hypercubes.  However, there are some cases for these two spaces that are not quite covered by Lemma~\ref{hyperred}, and hence we state lemmas for these two spaces separately below.

\begin{lemma}\label{hyper3red}
Let $B=(b_{ijk\ell})$ be an element in $V(\Q)=\Q^2\otimes\Sym_3\Q^2$ such that all the coordinates in at least one of the following sets vanish:
\begin{itemize}
\item[{\rm (i)}]
$\{b_{1111},b_{1112}\}$
\item[{\rm (ii)}]
$\{b_{1111},b_{2111}\}$
\end{itemize}
Then $B$ is reducible.
\end{lemma}

\begin{proof}
In case (i), we see that $y$ is a factor of $\Disc(b_{1ijk}x+b_{2ijk}y)$.  In case (ii), by replacing the cubical matrix $b_{1ijk}$ by a suitable $\Q$-linear combination of $b_{1ijk}$ and $b_{2ijk}$, we may transform $B$ (by an element of $G(\Q)$) so that $b_{1112}$ is zero.  Since $b_{1111}$ will remain zero, we are then in case (i).  Hence $B$ is reducible in either case.
\end{proof}

The space $V(\Q)=\Sym^2\Q^2\otimes\Sym^2\Q^2$ has a natural action by $S_2$, given by permuting the tensor factors.

\begin{lemma}\label{hyper22red}
Let $B=(b_{ijk\ell})$ be an element in $V(\Q)=\Sym^2\Q^2\otimes\Sym^2\Q^2$ such that, after a suitable action by an element of $S_2$, we have $b_{1111} = b_{1112} = 0$.
Then $B$ is reducible.
\end{lemma}

\begin{proof}
The covariant binary quartics have no quartic or cubic terms and thus are reducible.
\end{proof}

We next turn to the space of Rubik's cubes $(b_{ijk})$ 
($i,j,k\in\{1,2,3\}$) in $V(\Q)=\Q^3\otimes\Q^3\otimes\Q^3$.  The space $V(\Q)$ has a natural action by $S_3$, again given by permuting the tensor factors.

\begin{lemma}\label{rubikred}
Let $B=(b_{ijk})$ be an element in $V(\Q)=\Q^3\otimes\Q^3\otimes\Q^3$ such that, after a suitable action by an element of $S_3$, all the coordinates in at least one of the following sets vanish:
\begin{itemize}
\item[{\rm (i)}]
$\{b_{111},b_{112},b_{113},b_{121},b_{122},b_{123}\}$
\item[{\rm (ii)}]
$\{b_{111},b_{112},b_{113},b_{121},b_{122},b_{131},b_{211},b_{212},b_{221}\}$
\item[{\rm (iii)}]
$\{b_{111},b_{112},b_{113},b_{121},b_{131},b_{211},b_{311}\}$
\item[{\rm (iv)}]
$\{b_{111},b_{112},b_{121},b_{122},b_{211},b_{212},b_{221},b_{222}\}$
\end{itemize}
Then $B$ is reducible.
\end{lemma}

\begin{proof}
In all cases, we see that the curve in $\P^2$ defined by $\det(b_{1ij}x+b_{2ij}y+b_{3ij}z)=0$ is not smooth or has a flex at the point $(1:0:0)$.  Hence $B$ is reducible in all cases.
\end{proof}

Finally, the space $V(\Q)=\Q^3\otimes\Sym^2\Q^3$ may be viewed as the space $(b_{ijk})$ of doubly symmetric  Rubik's cubes over $\Q$, and thus the reducibility criteria in Lemma~\ref{rubikred} apply also to this space.  However, again, there is a case that we will need that is not quite covered by Lemma~\ref{rubikred}, and so we state the corresponding lemma for $\Q^3\otimes\Sym^2\Q^3$ separately.

\begin{lemma}\label{rubik2red}
Let $B=(b_{ijk})$ be an element in $V(\Q)=\Q^3\otimes\Sym_2\Q^3$ such that all the coordinates in at least one of the following sets vanish:
\begin{itemize}
\item[{\rm (i)}]
$\{b_{111},b_{112},b_{113},b_{122},b_{123}\}$
\item[{\rm (ii)}]
$\{b_{111},b_{112},b_{113},b_{122},b_{211},b_{212}\}$
\item[{\rm (iii)}]
$\{b_{111},b_{112},b_{113},b_{211},b_{212},b_{213}\}$
\item[{\rm (iv)}]
$\{b_{111},b_{112},b_{122},b_{211},b_{212},b_{222}\}$
\item[{\rm (v)}]
$\{b_{111},b_{211},b_{311}\}$
\end{itemize}
Then $B$ is reducible.
\end{lemma}

\begin{proof}
In cases (i)--(iv), we see that the curve in $\P^2$ defined by
$\det(b_{1ij}x+b_{2ij}y+b_{3ij}z)=0$ is not smooth or has a flex at
the point $(1:0:0)$.  In case (v), by replacing the matrix $b_{1ij}$
by a suitable $\Q$-linear combination of $b_{1ij}$, $b_{2ij}$, and
$b_{3ij}$, we may transform $B$ (by an element of $G(\Q)$) so that
$b_{112}$ and $b_{113}$ are zero.  Since $b_{111}$ will remain zero,
we are then in case (iii) of Lemma~\ref{rubikred}.  Hence $B$ is
reducible in all cases (i)--(v).
\end{proof}

\section{Counting irreducible elements of bounded height}\label{countsec}

In this section, we derive asymptotics for the number of
$G(\Z)$-equivalence classes of irreducible elements of $V(\Z)$ having
bounded invariants.  We also describe how these asymptotics change
when we restrict to counting elements in $V(\Z)$ satisfying a finite
set of congruence conditions.

Let $V^{(i)}$ ($i\in\{1,\ldots,N\}$) denote again the components of $V^\s(\R)$,
and let 
\[c_{i} = \frac{\Vol(\FF R^{(i)}\cap \{v\in V(\R):H(v)<1\})}{n_i}.\]
Then in this section we prove the following theorem:

\begin{theorem}\label{thmcount}
 Fix $i\in\{1,\ldots,N\}$.  
 For any $G(\Z)$-invariant set $S\subset V(\Z)^{(i)}:=V(\Z)\cap V^{(i)}$, let $N(S;X)$
  denote the number of $G(\Z)$-equivalence classes of irreducible
  elements $B\in S$ satisfying $H(B)<X$. Then
$$N(V(\Z)^{(i)};X)=c_{i}X^{n/k}+o(X^{n/k}).$$
\end{theorem}

\subsection{Averaging over fundamental domains}\label{avgsec}

Let $G_0$ be a compact, semialgebraic, left $K$-invariant set in
$G(\R)$ that is the closure of a nonempty open set and in which every
element has determinant greater than or equal to $1$. Let $V(\Z)^\irr$
denote the subset of elements of $V(\Z)$ that are irreducible.  Then
for any $i\in\{1,\ldots,N\}$, we may write
\begin{equation*}
N(V(\Z)^{(i)};X)=\frac{\int_{h\in G_0}\#\{x\in \FF hR\cap V(\Z)^{\irr}:
  H(x)<X\}dh\;}{n_i\cdot \int_{h\in G_0}dh},
\end{equation*}
where $V(\Z)^\irr$ denotes the set of irreducible elements in
$V(\Z)^\irr$ and $R$ is equal to $R^{(i)}$.  The denominator of the
latter expression is an absolute constant $C_{G_0}^{(i)}$ greater than
zero.

More generally, for any $G(\Z)$-invariant subset $S \subset
V(\Z)^{(i)}$, let $N(S;X)$ denote the number of irreducible
$G(\Z)$-orbits in $S$ having height less than $X$. Let $S^{\irr}$
denote the subset of irreducible points of $S$. Then $N(S;X)$ can be
similarly expressed as

\begin{equation}\label{eq9}
N(S;X)=\frac{\int_{h\in G_0}\#\{x\in \FF hR\cap S^{\irr}: H(x)<X\}dh\;}{C_{G_0}^{(i)}}.
\end{equation}
We use (\ref{eq9}) to define $N(S;X)$ even for
sets $S\subset V(\Z)$ that are not necessarily $G(\Z)$-invariant.

As in \cite[Thm.\ 2.5]{BS}, we may write $N(S;X)$ alternatively as
\begin{equation*}
N(S;X) =  \frac1{C_{G_0}^{(i)}}\int_{g\in N'(\alpha)A' K}
\#\{x\in S^\irr\cap \nu \alpha\kappa G_0R:H(x)<X\}\,dg
\end{equation*}
where $dg$ is a Haar measure on $G(\R)$.  Explicitly, if we write $G$
as a finite quotient of $\prod_i\SL_{2\mbox{ or }3}$, where the $i$-th
factor of $\SL$ has Iwasawa decomposition $N_iA_iK_i$, then we have
\begin{equation*}
dg \,=\,  \prod_i s_i^{-2}\,du_i\, d^\times s_i \,d\kappa_i \;\mbox{ or }\; 
\prod_i t_i^{-6}u_i^{-6}\,d\nu_i \,d^\times t_i\, d^\times u_i\,
d\kappa_i 
\,,\end{equation*} respectively,
where $d\nu_i$ and $d\kappa_i$ are invariant measures on $N_i$ and $K_i$, respectively.  
We normalize the invariant measure $d\kappa_i$ on $K_i$ so that $\int_{K_i} d\kappa_i=1.$

Let us write $E(\nu,\alpha,X) = \nu \alpha G_0R\cap\{x\in
V^{(i)}:H(x)<X\}$, again viewed as a multiset.  
As $KG_0=G_0$ and $\int_K dk =1 $, we have
\begin{equation}\label{avg}
N(S;X) = \frac1{C_{G_0}^{(i)}}\int_{g\in N'(a)A'}                              
\#\{x\in S^\irr\cap E(\nu,\alpha,X)\}\,dg.
\end{equation}

We note that the same counting method may be used even if we are interested in counting both
reducible and irreducible orbits in $V(\Z)$. For any set $S\subset V^{(i)}$, let $N^\ast(S;X)$ be defined by (\ref{avg}), but
where the superscript ``irr'' is removed:
\begin{equation}\label{avg2}
N^\ast(S;X) = \frac1{C_{G_0}^{(i)}}\int_{g\in N'(s)A'}                              
\#\{x\in S\cap E(\nu,\alpha,X)\}\,dg.
\end{equation}
 Thus for a $H(\Z)$-invariant set $S\subset V^{(i)}$, $N^\ast(S;X)$ counts 
the total (weighted) number of $H(\Z)$-orbits in $S$ having height less than $X$ (not just the irreducible ones). 

The expression (\ref{avg}) for $N(S;X)$, and its analogue (\ref{avg2}) for $N^\ast(S,X)$, will be useful to us in the sections
that follow.

\subsection{An estimate from the geometry of numbers}

To estimate the number of lattice points in the multiset $E(\nu,\alpha,X)$, we
use the following result due to Davenport~\cite{Davenport1}.

\begin{proposition}\label{davlem}
  Let $\RR$ be a bounded semialgebraic multiset in $\R^n$
  having maximum multiplicity $m$ and defined by at most $k$
  polynomial inequalities each having degree at most $\ell$.  Let $\RR'$
  denote the image of $\RR$ under any $($upper or lower$)$ triangular
  unipotent transformation of $\R^n$.  Then the number of integral
  lattice points, counted with multiplicity, contained in the
  region $\mathcal R'$ is
\[\Vol(\mathcal R)+ O(\max\{\Vol(\bar{\mathcal R}),1\}),\]
where $\Vol(\bar{\mathcal R})$ denotes the greatest $d$-dimensional 
volume of any projection of $\mathcal R$ onto a coordinate subspace
obtained by equating $n-d$ coordinates to zero, for any $d$ between
$1$ and $n-1$.  The implied constant in the second summand depends
only on $n$, $m$, $k$, and $\ell$.
\end{proposition}
Although Davenport states Lemma \ref{davlem} only for compact
semialgebraic sets, his proof adapts without
significant change to the more general case of a bounded semialgebraic
{\em multiset} $\mathcal R\subset\R^n$, with the same estimate applying also to
any image $\mathcal R'$ of $\mathcal R$ under a unipotent triangular
transformation.

\subsection{Cutting off the cusps}

The following proposition shows that the number of points in $B\in\FF
hR^{(i)}\cap V(\Z)$ having bounded height, where the coordinate of
lowest weight (namely, $b_{1111}$ or $b_{111}$) vanishes, is
negligible.

\begin{proposition}\label{hard}
Let $h$ take a random value in $G_0$ uniformly with respect to the Haar measure
$dg$.  Then the expected number of 
irreducible elements $B\in\FF h R^{(i)}\cap V(\Z)$ such that 
$H(B)< X$  and $b_{1111}=0$ $($resp., $b_{111}=0)$ 
is $O_\varepsilon(X^{(n-1)/k+\varepsilon})$.
\end{proposition}

\begin{proof}
We follow the method in \cite{dodpf}.  Namely, we divide the set of all $B\in V(\Z)$
into a number of cases depending on which initial coordinates are zero
and which are nonzero.  These cases are described in the second columns
of Tables 2(a)--(e).  The vanishing conditions in the various subcases of
Case~$m+1$ are obtained by setting equal to 0---one at a time---each
variable that was assumed to be nonzero in Case~$m$.  If such a
resulting subcase satisfies the reducibility conditions of the corresponding lemma among
Lemmas~\ref{hyperred}--\ref{rubik2red}, it is not listed.  In this way, it becomes clear that
any irreducible element in $V(\Z)$ must satisfy precisely
one of the conditions enumerated in the second column of the corresponding table.

 Let $T$ denote the set of all $n$ variables
$b_{ijk\ell}$ (or $b_{ijk}$) corresponding to the coordinates on $V(\Z)$.  
For a subcase $\CC$ of the corresponding table among Tables~2(a)--(e), we use
$T_0=T_0(\CC)$ to denote the set of variables in $T$ assumed to be 0
in Subcase~$\CC$, and $T_1$ to denote the set of variables in $T$
assumed to be nonzero.  

Each variable $b\in T$ has a {\it weight}, defined as follows. The
action of $r=(s_1,s_2,\ldots)$ or $r=(t_1,u_1,t_2,u_2,\ldots)$ on $B\in
V$ causes each variable $b$ to multiply by a certain weight which we
denote by $w(b)$.  These weights $w(b)$ are evidently rational
functions in $s_1,s_2,\ldots$ or $t_1,u_1,t_2,u_2,\ldots$.  

Let $V(\CC)$ denote the set of $B\in V(\R)$ such that
$B$ satisfies the vanishing and nonvanishing conditions of
Subcase~$\CC$.  For example, in Subcase~3b of Table 2(d) we have
$T_0(\mbox{3b})=\{b_{111}, b_{112}, b_{121}\}$ and
$T_1(\mbox{3b})=\{b_{113}, b_{122}, b_{131},b_{211}\}$; thus $V(\mbox{3b})$
denotes the set of all $B\in V(\Z)=\Z^3\otimes\Z^3\otimes\Z^3$ such that $b_{111}=b_{112}= b_{121}=0$
but $b_{113}, b_{122}, b_{131},b_{211}\neq 0$. 

For each subcase $\CC$ of Case $m$ ($m>0$), we wish to show that
$N(V(\CC);X)$ is $O_\varepsilon(X^{(n-1)/k+\varepsilon})$.  
Since $N'(\alpha)$ is absolutely bounded, the equality (\ref{avg2}) implies that
\begin{equation*}
N^*(V(\CC);X)\ll
\int_{s_1,s_2,\cdots=c}^\infty
\sigma(V(\CC))
\, s_1^{-2}s_2^{-2}\cdots\,d^\times s_2 \,d^\times s_1 
\end{equation*}
or 
\begin{equation*}
N^*(V(\CC);X)\ll
\int_{t_1,u_1,t_2,u_2,\cdots=c}^\infty
\sigma(V(\CC))
\, t_1^{-6}u_1^{-6}t_2^{-6}u_2^{-6}\cdots\,d^\times u_2 \,d^\times
t_2\,d^\times u_1 \,d^\times t_1, 
\end{equation*}
where $\sigma(V(\CC))$ denotes the number of integer points
in the region $E(\nu,\alpha,\lambda,X)$ that also satisfy the conditions 
\begin{equation}\label{cond}
\mbox{$b=0$ for $b\in T_0$ and $|b|\geq 1$ for $b\in T_1$}.
\end{equation}

Now for an element $B\in H(\nu,\alpha,X)$, we evidently have
\begin{equation}\label{condt}
|b|\leq J{w(b)}X^{1/k}
\end{equation}
for some absolute constant $J>0$, and therefore the number of integer points in $H(\nu,\alpha,\lambda,X)$
satisfying (\ref{cond}) will be nonzero only if we have
\begin{equation}\label{condt1}
J{w(b)}X^{1/k}\geq 1
\end{equation}
for all weights $w(b)$ such that $b\in T_1$.  Now the sets $T_1$ in
each subcase of Table~2 have been chosen to be precisely the set of
variables having the minimal weights $w(b)$ among the variables $t\in
T\setminus T_0$; by ``minimal weight'' in $T\setminus T_0$, we mean
that there is no other variable $b\in T\setminus T_0$ with weight having
equal or smaller exponents for all parameters $
s_1,s_2, \ldots$ (resp., $t_1,u_1,t_2,u_2, \ldots$).  Thus if
the condition \eqref{condt1} holds for all weights $w(b)$
corresponding to $b\in T_1$, then---by the very choice of $T_1$---we
will also have $Jw(b)\gg 1$ for all weights $w(b)$ such that $b\in
T\setminus T_0$.

Therefore, if the region $\H=\{B\in
H(\nu,\alpha,X):b=0\;\;\forall b\in T_0;\;\; |b|\geq 1\;\; \forall
b\in T_1\}\subset\R^{n-|T_0|}$ contains an integer point, then (\ref{condt1}) and
Lemma~\ref{davlem} together imply that the number of integer points
in $\H$ is $O(\Vol(\H))$, since the volumes of all the projections of
$\nu^{-1}\H$ will in that case also be $O(\Vol(\H))$.  
Now clearly
\[\Vol(\H)=O\Bigl(J^{n-|T_0|}X^{\frac{n-|T_0|}k}\prod_{b\in T\setminus T_0} w(b)\Bigr),\]
so we obtain
\begin{equation}\label{estv1s}
N(V(\CC);X)\ll
\int_{s_1,s_2,\cdots=c}^\infty
X^{\frac{n-|T_0|}k}
\prod_{b\in T\setminus T_0}w(b) 
\,\, s_1^{-2}s_2^{-2}\cdots \,\,\,
d^\times\! s_2\,d^\times\!s_1 
\end{equation}
or
\begin{equation}\label{estv1s2}
N(V(\CC);X)\ll 
\int_{t_1,u_1,t_2,u_2,\cdots=c}^\infty
X^{\frac{n-|T_0|}k}
\prod_{b\in T\setminus T_0} w(b)
\, t_1^{-6}u_1^{-6}t_2^{-6}u_2^{-6}\cdots\,d^\times u_2 \,d^\times
t_2\,d^\times u_1 \,d^\times t_1.
\end{equation}

The latter integral can be explicitly carried out for each of the
subcases in Table~2.  It suffices, however, to have a simple
estimate of the form $O_\varepsilon(X^r)$, with $r\leq (n-1)/k+\varepsilon$, for the integral
corresponding to each subcase.  For example, if the total exponent of
$s_i$ (resp., $t_i,u_i$) in \eqref{estv1s} or \eqref{estv1s2} is negative for all $i$,
then it is clear that the resulting integral will be at most
$O(X^{(n-|T_0|)/k})$ in value.  This condition holds for many of the
subcases in Table 2 (indicated in the fourth column by ``-''),
immediately yielding the estimates given in the third column.

For cases where this negative exponent condition does not hold, the
estimate given in the third column can be obtained as follows.  The
factor $\pi$ given in the fourth column is a product of variables in
$T_1$, and so it is at least one in absolute value.  The integrand in
\eqref{estv1s} or \eqref{estv1s2} may thus be multiplied by $\pi$ without harm, and the
estimates \eqref{estv1s} and \eqref{estv1s2} will remain true; we may then apply the
inequalities \eqref{condt} to each of the variables in $\pi$, yielding
\begin{equation}\label{estv2s}
N(V(\CC);X)\ll
\int_{s_1,s_2,\cdots=c}^\infty
X^{\frac{n-|T_0| + \#\pi}{k}}
\prod_{b\in T\setminus T_0} w(b)\;w(\pi)
\,\, s_1^{-2}s_2^{-2}\cdots \,\,\,
d^\times\! s_2\,d^\times\!s_1 
\end{equation}
or
\begin{equation}\label{estv2s2}
N(V(\CC);X)\ll
\int_{t_1,u_1,t_2,u_2,\cdots=c}^\infty
X^{\frac{n-|T_0| + \#\pi}{k}}
\prod_{b\in T\setminus T_0} w(b)\;w(\pi)
\, t_1^{-6}u_1^{-6}t_2^{-6}u_2^{-6}\cdots\,d^\times u_2 \,d^\times
t_2\,d^\times u_1 \,d^\times t_1 
\end{equation}
where $\#\pi$ denotes the total number of variables of $T$ appearing
in $\pi$ (counted with multiplicity), and we extend the notation $w$
multiplicatively, i.e., $w(ab)=w(a)w(b)$.  In each subcase of Table~2,
we have chosen the factor $\pi$ so that the total exponent of each
$s_i$ (resp., each $t_i$ and each $u_i$) in \eqref{estv2s} (resp.,
\eqref{estv2s2}) is negative.  Thus we obtain from \eqref{estv2s} and \eqref{estv2s2} that
$N(V(\CC);X)=O(X^{(n-\#T_0(\CC)+\#\pi)/k})$, 
and this is precisely the estimate given in the third
column of Table~2.  In every subcase, aside from Case~0, we see that
$n-\#T_0+\#\pi<n$, as desired.
\end{proof}

\pagebreak
\begin{center}
\begin{tabular}{|r|l|c|c|}\hline
Case & The set $S\subset V(\Z)$ defined by & $N(S;X)\ll$ & Use factor \\ 
\hline\hline 
0. & ${b_{1111}}\neq0\,$ & $X^{16/24}$ &  --  \\
\hline
1. & ${b_{1111}}=0\,;$ & $X^{15/24}$ &  --  \\
  & ${b_{1112}, b_{1121},b_{1211},b_{2111}}\neq 0$ &  &  \\ \hline
2. & ${b_{1111}, b_{1112}}=0\,;$ & $X^{14/24 + \varepsilon}$ &  --  \\
  & ${b_{1121}, b_{1211}, b_{2111}}\neq 0$ &  &  \\ \hline
3. & ${b_{1111}, b_{1112}, b_{1121}}=0\,;$ & $X^{14/24 + \varepsilon}$ &  $b_{1122}$ \\
  & ${b_{1122}, b_{1211}, b_{2111}}\neq 0$ &  &  \\ \hline
\end{tabular}
\end{center}

\vspace{-.1in}
\begin{center}
{\bf Table 2(a).} Estimates for $2 \otimes 2 \otimes 2 \otimes 2$.
\end{center}

\vspace{.2in}

\begin{center}
\begin{tabular}{|r|l|c|c|}\hline
Case & The set $S\subset V(\Z)$ defined by & $N(S;X)\ll$ & Use factor \\ 
\hline\hline 
0. & ${b_{1111}}\neq0\,$ & $X^{8/24}$ &  --  \\
\hline
1. & ${b_{1111}}=0\,;$ & $X^{22/72+\varepsilon}$ &  $(b_{2111})^{1/3}$  \\
  & ${b_{1112}, b_{2111}}\neq 0$ &  &  \\ \hline
\end{tabular}
\end{center}

\vspace{-.1in}
\begin{center}
{\bf Table 2(b).} Estimates for $2 \otimes \Sym_3(2)$.
\end{center}

\vspace{.2in}

\begin{center}
\begin{tabular}{|r|l|c|c|}\hline
Case & The set $S\subset V(\Z)$ defined by & $N(S;X)\ll$ & Use factor \\ 
\hline\hline 
0. & ${b_{1111}}\neq0\,$ & $X^{9/12}$ &  --  \\
\hline
1. & ${b_{1111}}=0\,;$ & $X^{8/12+\varepsilon}$ &  --  \\
  & ${b_{1112}, b_{1211}}\neq 0$ &  &  \\ \hline
\end{tabular}
\end{center}

\vspace{-.1in}
\begin{center}
{\bf Table 2(c).} Estimates for $\Sym^2(2)\otimes\Sym^2(2)$.
\end{center}

\vspace{.2in}

\begin{center}
\begin{tabular}{|r|l|c|c|}\hline
Case & The set $S\subset V(\Z)$ defined by & $N(S;X)\ll$ & Use factor \\ 
\hline\hline 
0. & ${b_{111}}\neq0\,$ & $X^{27/36}$ &  --  \\
\hline
1. & ${b_{111}}=0\,;$ & $X^{26/36}$ &  --  \\
  & ${b_{112}, b_{121},b_{211}}\neq 0$ &  &  \\ \hline
2. & ${b_{111}, b_{112}}=0\,;$ & $X^{25/36}$ &  --  \\
  & ${b_{113}, b_{121}, b_{211}}\neq 0$ &  &  \\ \hline
3a. & ${b_{111}, b_{112},b_{113}}=0\,;$ & $X^{24/36+\varepsilon}$ &  --  \\
  & ${b_{121}, b_{211}}\neq 0$ &  &  \\ \hline
3b. & ${b_{111}, b_{112}, b_{121}}=0\,;$ & $X^{24/36+\varepsilon}$ &  --  \\
  & ${b_{113}, b_{122}, b_{131},b_{211}}\neq 0$ &  &  \\ \hline
4a. & ${b_{111}, b_{112}, b_{113},b_{121}}=0\,;$ & $X^{24/36+\varepsilon}$ &  $b_{122}$  \\
  & ${b_{122}, b_{131},b_{211}}\neq 0$ &  &  \\ \hline
4b. & ${b_{111}, b_{112}, b_{121},b_{122}}=0\,;$ & $X^{24/36+\varepsilon}$ &  $b_{113}$  \\
  & ${b_{113}, b_{131}, b_{211}}\neq 0$ &  &  \\ \hline
4c. & ${b_{111}, b_{112}, b_{121},b_{211}}=0\,;$ & $X^{23/36}$ &  --  \\
  & ${b_{113}, b_{122},b_{131},b_{212},b_{221},b_{311}}\neq 0$ &  &  \\ \hline
\end{tabular}
\end{center}

\vspace{-.1in}
\begin{center}
{\bf Table 2(d).} {Subcases 0--4c of estimates for $3 \otimes 3 \otimes 3$.}
\end{center}

\pagebreak

\begin{center}
\begin{tabular}{|r|l|c|c|}\hline
Case & The set $S\subset V(\Z)$ defined by & $N(S;X)\ll$ & Use factor \\ 
\hline\hline 
5a. & ${b_{111}, b_{112}, b_{113}, b_{121},b_{122}}=0\,;$ & $X^{24/36+\varepsilon}$ & $b_{123}^2$ \\
  & ${b_{123}, b_{131},b_{211}}\neq 0$ &  &  \\ \hline
5b. & ${b_{111}, b_{112}, b_{113},b_{121},b_{131}}=0\,;$ & $X^{24/36+\varepsilon}$ & $b_{122}^2$ \\
  & ${b_{122}, b_{211}}\neq 0$ &  &  \\ \hline
5c. & ${b_{111}, b_{112}, b_{113},b_{121},b_{211}}=0\,;$ & $X^{24/36+\varepsilon}$ & $b_{122} b_{212}$ \\
  & ${b_{122}, b_{131}, b_{212}, b_{221},b_{311}}\neq 0$ &  &  \\ \hline
5d. & ${b_{111}, b_{112}, b_{121},b_{122},b_{211}}=0\,;$ & $X^{24/36+\varepsilon}$ & $b_{113}b_{131}$  \\
  & ${b_{113},b_{131},b_{212},b_{221},b_{311}}\neq 0$ &  &  \\ \hline
6a. & ${b_{111}, b_{112}, b_{113}, b_{121},b_{122},b_{131}}=0\,;$ & $X^{26/36}$ &  $b_{123}^2 b_{132}^2 b_{211}$  \\
  & ${b_{123}, b_{132}, b_{211}}\neq 0$ &  &  \\ \hline
6b. & ${b_{111}, b_{112}, b_{113}, b_{121},b_{122},b_{211}}=0\,;$ & $X^{24/36+\varepsilon}$ & $b_{123} b_{131} b_{212}$  \\
  & ${b_{123}, b_{131}, b_{212},b_{221},b_{311}}\neq 0$ &  &  \\ \hline
6c. & ${b_{111}, b_{112}, b_{113},b_{121},b_{131},b_{211}}=0\,;$ & $X^{24/36+\varepsilon}$ & $b_{122}^2 b_{311}$ \\
  & ${b_{122},b_{212},b_{221},b_{311}}\neq 0$ &  &  \\ \hline
6d. & ${b_{111}, b_{112}, b_{113},b_{121},b_{211},b_{221}}=0\,;$ & $X^{21/36+\varepsilon}$ &  --  \\
  & ${b_{122}, b_{131}, b_{212},b_{311}}\neq 0$ &  &  \\ \hline
6e. & ${b_{111}, b_{112}, b_{121},b_{122},b_{211},b_{212}}=0\,;$ & $X^{21/36+\varepsilon}$ &  --  \\
  & ${b_{113}, b_{131}, b_{221},b_{311}}\neq 0$ &  &  \\ \hline
7a. & ${b_{111}, b_{112}, b_{113}, b_{121},b_{122},b_{131},b_{211}}=0\,;$ & $X^{24/36+\varepsilon}$ &  $b_{123}^2 b_{132} b_{311}$  \\
  & ${b_{123}, b_{132}, b_{212},b_{221},b_{311}}\neq 0$ &  &  \\ \hline
7b. & ${b_{111}, b_{112}, b_{113}, b_{121},b_{122},b_{211},b_{212}}=0\,;$ & $X^{24/36+\varepsilon}$ &  $b_{123} b_{213} b_{131} b_{311}$  \\
  & ${b_{123}, b_{131}, b_{213},b_{221},b_{311}}\neq 0$ &  &  \\ \hline
7c. & ${b_{111}, b_{112}, b_{113}, b_{121},b_{122},b_{211},b_{221}}=0\,;$ & $X^{21/36+\varepsilon}$ &  $b_{123}$  \\
  & ${b_{123}, b_{131}, b_{212},b_{311}}\neq 0$ &  &  \\ \hline
7d. & ${b_{111}, b_{112}, b_{113},b_{121},b_{131},b_{211},b_{212}}=0\,;$ & $X^{24/36+\varepsilon}$ & $b_{122}^2 b_{213} b_{311}$ \\
  & ${b_{122},b_{213},b_{221},b_{311}}\neq 0$ &  &  \\ \hline
7e. & ${b_{111}, b_{112}, b_{121},b_{122},b_{211},b_{212},b_{221}}=0\,;$ & $X^{21/36+\varepsilon}$ &  $b_{222}$  \\
  & ${b_{113}, b_{131}, b_{222},b_{311}}\neq 0$ &  &  \\ \hline
8a. & ${b_{111}, b_{112}, b_{113}, b_{121},b_{122},b_{131},b_{211},b_{212}}=0\,;$ & $X^{24/36+\varepsilon}$ &  $b_{123} b_{132}^2 b_{213} b_{311}$  \\
  & ${b_{123}, b_{132}, b_{213},b_{221},b_{311}}\neq 0$ &  &  \\ \hline
8b. & ${b_{111}, b_{112}, b_{113}, b_{121},b_{122},b_{211},b_{212},b_{221}}=0\,;$ & $X^{24/36+\varepsilon}$ &  $b_{123} b_{132} b_{213} b_{221} b_{311}$  \\
  & ${b_{123}, b_{131}, b_{213},b_{222},b_{311}}\neq 0$ &  &  \\ \hline
8c. & ${b_{111}, b_{112}, b_{113},b_{121},b_{131},b_{211},b_{212},b_{221}}=0\,;$ & $X^{24/36+\varepsilon}$ & $b_{122}^2 b_{213} b_{231} b_{311}$ \\
  & ${b_{122},b_{213},b_{222},b_{231},b_{311}}\neq 0$ &  &  \\ \hline
\end{tabular}
\end{center}

\vspace{-.1in}
\begin{center}
{\bf Table 2(d) cont'd.} {Subcases 5a--8c of estimates for $3 \otimes 3 \otimes 3$.} 
\end{center}

\pagebreak

\begin{center}
\begin{tabular}{|r|l|c|c|}\hline
Case & The set $S\subset V(\Z)$ defined by & $N(S;X)\ll$ & Use factor \\ 
\hline\hline 
0. & ${b_{111}}\neq0\,$ & $X^{18/36}$ &  --  \\
\hline
1. & ${b_{111}}=0\,;$ & $X^{17/36}$ &  --  \\
  & ${b_{112}, b_{211}}\neq 0$ &  &  \\ \hline
2a. & ${b_{111}, b_{112}}=0\,;$ & $X^{16/36}$ &  --  \\
  & ${b_{113}, b_{122},b_{211}}\neq 0$ &  &  \\ \hline
2b. & ${b_{111}, b_{211}}=0\,;$ & $X^{17/36}$ &  $b_{311}$  \\
  & ${b_{112}, b_{311}}\neq 0$ &  &  \\ \hline
3a. & ${b_{111}, b_{112},b_{113}}=0\,;$ & $X^{15/36+\varepsilon}$ &  --  \\
  & ${b_{122}, b_{211}}\neq 0$ &  &  \\ \hline
3b. & ${b_{111}, b_{112},b_{122}}=0\,;$ & $X^{15/36+\varepsilon}$ &  --  \\
  & ${b_{113},b_{211}}\neq 0$ &  &  \\ \hline
3c. & ${b_{111}, b_{112},b_{211}}=0\,;$ & $X^{16/36}$ &  $b_{311}$  \\
  & ${b_{113}, b_{122},b_{212},b_{311}}\neq 0$ &  &  \\ \hline
4a. & ${b_{111}, b_{112},b_{113},b_{122}}=0\,;$ & $X^{15/36+\varepsilon}$ &  $b_{123}$  \\
  & ${b_{123}, b_{211}}\neq 0$ &  &  \\ \hline
4b. & ${b_{111}, b_{112},b_{113},b_{211}}=0\,;$ & $X^{15/36+\varepsilon}$ &  $b_{311}$  \\
  & ${b_{122}, b_{212},b_{311}}\neq 0$ &  &  \\ \hline
4c. & ${b_{111}, b_{112},b_{122},b_{211}}=0\,;$ & $X^{15/36+\varepsilon}$ &  $b_{311}$  \\
  & ${b_{113},b_{212},b_{311}}\neq 0$ &  &  \\ \hline
4d. & ${b_{111}, b_{112},b_{211},b_{212}}=0\,;$ & $X^{15/36+\varepsilon}$ &  $b_{311}$  \\
  & ${b_{113},b_{122},b_{311}}\neq 0$ &  &  \\ \hline
5a. & ${b_{111}, b_{112},b_{113},b_{122},b_{211}}=0\,;$ & $X^{15/36+\varepsilon}$ &  $b_{123} b_{311}$  \\
  & ${b_{123}, b_{212},b_{311}}\neq 0$ &  &  \\ \hline
5b. & ${b_{111}, b_{112},b_{113},b_{211},b_{212}}=0\,;$ & $X^{15/36+\varepsilon}$ &  $b_{213} b_{311}$ \\
  & ${b_{122}, b_{213},b_{311}}\neq 0$ &  &  \\ \hline
5c. & ${b_{111}, b_{112},b_{122},b_{211},b_{212}}=0\,;$ & $X^{15/36+\varepsilon}$ &  $b_{222} b_{311}$  \\
  & ${b_{113},b_{222},b_{311}}\neq 0$ &  &  \\ \hline
\end{tabular}
\end{center}

\vspace{-.1in}
\begin{center}
{\bf Table 2(e).} {Estimates for $3 \otimes\Sym_2(3)$.} 
\end{center}

\subsection{The number of irreducible points in the main body} \label{sec:irrmain}

We now give an estimate on the number of 
reducible elements $B\in \FF h R \cap V(\Z)$, on average, satisfying
$b_{1111}\neq 0$ (resp., $b_{111}\neq 0$):

\begin{proposition}\label{hard2}
Let $h$ take a random value in $G_0$ uniformly with respect to the measure
$dg$.  Then the expected number of 
reducible elements $B\in\FF h R^{(i)}\cap V(\Z)$ such that 
$H(B)< X$  and $b_{1111}\neq 0$ $($resp., $b_{111}\neq 0)$
is $o(X^{n/k})$.
\end{proposition}
\noindent We defer the proof of this proposition to the end of the section.

 We also have the following proposition which bounds the number of
$G(\Z)$-equivalence classes of integral elements in $\FF hR^{(i)}$ with height less than $X$ that have
large stabilizers inside $G(\Q)$; we again
defer the proof to the end of the section.

\begin{proposition}\label{gzbigstab}
  Let $h\in G_0$ be any element, where $G_0$ is any fixed compact
  subset of $G(\R)$.  Then the number of integral
  elements $B\in \RR_X(h)$ whose stabilizer in $G(\Q)$ has size greater
  than $1$ (for $B \in V^{(i)}$) is $o(X^{n/k})$.
\end{proposition}

\subsection{The main term} \label{sec:mainterm}

Fix again $i\in\{1,\ldots,N\}$ and let $R=R^{(i)}$.  The results of \S \ref{sec:irrmain} show that, in order to obtain
Theorem~\ref{thmcount}, it suffices to count those integral elements
$B\in \FF h R$ of bounded height for which $b_{1111}\neq 0$ (resp.,
$b_{111}\neq 0$), as $h$ ranges over $G_0$.

Let $\RR_X(h)$ denote the region $\FF h R\cap \{B\in
V(\R):H(B)<X\}$; let $\RR_X:=\RR_X(1)$.  Then we have the following result counting the
number of integral points in $\RR_X(h)$, on average, satisfying
$b_{1111}\neq 0$ (resp., $b_{111}\neq0$):

\begin{proposition}\label{nonzerob11}
  Let $h$ take a random value in $G_0$ uniformly with
  respect to the Haar measure $dg$.  Then the expected
  number of elements $B\in\FF hR\cap V(\Z)^{(i)}$ such that
  $|H(B)|< X$ and $b_{1111}\neq0$ $($resp., $b_{111}\neq0)$ is $\Vol(\RR_X)
  + O(X^{(n-1)/k})$.
\end{proposition}

\begin{proof}
Following the proof of Lemma~\ref{hard}, let $V^{(i)}(\phi)$ denote the subset of $V(\R)$ such that $b_{1111}\neq 0$ (resp., $b_{111}\neq 0$).  
We wish to show that
\begin{equation*}
N^*(V^{(i)}(\phi);X)=\Vol(\RR_X) + O(X^{(n-1)/k}).
\end{equation*}
We have
\begin{equation*}
  N^*(V^{(i)}(\phi);X)=\frac{1}{C^{(i)}_{G_0}}
\int_{s_1,s_2,\cdots=c}^\infty
\int_{\nu\in    N'(s)} 
  \sigma(V(\phi))
\, s_1^{-2}s_2^{-2}\cdots\,d^\times s_2 \,d^\times s_1 
\end{equation*}
or 
\begin{equation*}
  N^*(V^{(i)}(\phi);X)=\frac{1}{C^{(i)}_{H_0}}
\int_{t_1,u_1,t_2,u_2,\cdots=c}^\infty
\int_{\nu\in    N'(s)} 
  \sigma(V(\phi))
t_1^{-6}u_1^{-6}t_2^{-6}u_2^{-6}\cdots\,d^\times u_2 \,d^\times
t_2\,d^\times u_1 \,d^\times t_1, 
\end{equation*}
where $\sigma(V(\phi))$ denotes the number of integer points
in the region $E(\nu,\alpha,X)$ satisfying $|b_{\min}|\geq 1$, where $b_{\min}$ denotes $b_{1111}$ (resp., $b_{111})$.
Evidently, the number of integer points in $E(\nu,\alpha,X)$
with $|b_{\min}|\geq 1$ can be nonzero only if we have
\begin{equation}\label{condt2}
J{w(b_{\min})}X^{1/k}\geq 1.
\end{equation}
Therefore, if the region $\BB=\{B\in
E(\nu,\alpha,X):|b_{\min}|\geq 1\}$
 contains an integer point, then \eqref{condt2} and
Lemma~\ref{davlem} imply that the number of integer points in $\BB$
is $\Vol(\BB)+O(\Vol(\BB)/(w(b_{\min})X^{1/k}))$, since all
smaller-dimensional projections of $u^{-1}\BB$ are clearly bounded by
a constant times the projection of $\BB$ onto the hyperplane $b_{\min}=0$ (since
$b_{\min}$ has minimal weight).

Therefore, since $\BB=E(\nu,\alpha,X)-\bigl(E(\nu,\alpha,X)-\BB\bigr)$, 
we may write
\begin{eqnarray}\label{bigint}\nonumber
\!\!\!\!\!\!\!N^\ast(V^{(i)}(\phi);X) &\!\!\!\!\!\!\!=\!\!\!\!\!\!& \!\frac1{C^{(i)}_{H_0}}
\int_{s_1,s_2,\cdots=c}^{\infty} \int_{\nu\in N'(\alpha)}\!\!
\Bigl(\Vol\bigl(E(\nu,\alpha,X)\bigr)\!-\!\Vol\bigl(B(\nu,\alpha,X)\!-\!\BB\bigr) 
 \\[.085in] & & 
+O(\max\{X^{\frac{n-1}k}s_1^2s_2^2\cdots ,1\})
\Bigr)   
\, s_1^{-2}s_2^{-2}\cdots\,d\nu\,d^\times s_2 \,d^\times s_1 
\end{eqnarray} 
or
\begin{eqnarray} \label{bigint2}
\!\!\!\!\!\!\!N^\ast(V^{(i)}(\phi);X) &\!\!\!\!\!\!\!=\!\!\!\!\!\!& \!\frac1{C^{(i)}_{H_0}}
\int_{t_1,u_1,t_2,u_2,\cdots=c}^{\infty} \int_{\nu\in N'(\alpha)}\!\!
\Bigl(\Vol\bigl(E(\nu,\alpha,X)\bigr)\!-\!\Vol\bigl(E(\nu,\alpha,X)\!-\!\BB\bigr) 
 \\[.085in] & & \nonumber
+O(\max\{X^{\frac{n-1}k}t_1^6u_1^6t_2^6u_2^6\cdots ,1\})
\Bigr)   
t_1^{-6}u_1^{-6}t_2^{-6}u_2^{-6}\cdots\,d\nu\,d^\times\! u_2
\,d^\times\! t_2\,d^\times \!u_1 \,d^\times\! t_1.
\end{eqnarray}
The integral of the first term in \eqref{bigint} or \eqref{bigint2} is $\int_{h\in G_0}
\Vol(\RR_X(h))dg$.
Since $\Vol(\RR_X(h))$ does not
depend on the choice of $h\in G_0$,
the latter integral is simply $C^{(i)}_{G_0}\cdot \Vol(\RR_X)$.

To estimate the integral of the second term in \eqref{bigint} or \eqref{bigint2}, let $\BB'=
E(\nu,\alpha,X)-\BB$, and for each $|b_{\min}|\leq 1$, let $\BB'(b_{\min})$ be
the subset of all elements $B\in\BB'$ with the given value of
$b_{\min}$.  Then the $(n-1)$-dimensional volume of $\BB'(b_{\min})$ is at most 
$O\Bigl(X^{\frac{n-1}k}\prod_{b\in T\setminus\{b_{\min}\}}w(b)\Bigr)$, and so we
have the estimate 
\[\Vol(\BB') \ll \int_{-1}^1 X^{\frac{n-1}k}\prod_{b\in
  T\setminus\{b_{\min}\}}w(b)
\,\,db_{\min} = O\Bigl(X^{\frac{n-1}k}\prod_{b\in T\setminus\{b_{\min}\}}w(b)\Bigr).\]
The second term of the integrand in \eqref{bigint} or \eqref{bigint2} can thus be
absorbed into the third term.
 
Finally, one easily computes the
integral of the third term in \eqref{bigint} and \eqref{bigint2} to be
$O(X^{(n-1)/k})$.  We thus obtain
\begin{equation*}
N^\ast(V^{(i)};X) = \Vol(\RR_X) + O(X^{(n-1)/k}),
\end{equation*}
as desired.
\end{proof}

Combining Propositions~\ref{hard}, \ref{hard2}, \ref{gzbigstab}, and \ref{nonzerob11} yields Theorem~\ref{thmcount}.

\subsection{Computation of the volume}\label{secvol}

In this subsection, we describe how to compute
the volume of $\mathcal R_X$ in $V^{(i)}\subset V(\R)$.

Let $m$ denote the number of independent invariants for the action of
$G_1$ on $V$.  Given an element $v\in V(\R)$, we may attach to $v$ a
vector $\vec a=\vec a(v)\in \R^m$ whose coordinates are the
independent invariants of $v$.  For example, in the case of
$V=2\otimes2\otimes 2\otimes2$, we have $\vec
a(v)=(a_2(v),a_4(v),a_4'(v),a_6(v))$.  For each $\vec a\in \R^m$, the
set $R^{(i)}$ contains at most one point $p^{(i)}(\vec a)$ having
invariant vector $\vec a$. Let $R^{(i)}(X)$ denote the set of all
those points in $R^{(i)}$ having height less than $X$. Then
$\Vol(\mathcal R_X)=\Vol(\FF\cdot R^{(i)}(X))$

The set $R^{(i)}$ is in canonical one-to-one correspondence with the
set $\{\vec a\in \R^{m}:\Delta(\vec a)>0\}$ or $\{\vec a\in \R^{m}:\Delta(\vec a)<0\}$ in accordance with whether 
$\Delta$ takes positive or negative values on $R^{(i)}$.
There is thus a
natural measure $d\vec a$ on each of these sets $R^{(i)}$, given by the standard Euclidean measure on $\{\vec a(v):v\in R^{(i)}\}$ viewed as a subset of $\R^m$.

We then have the following proposition, whose proof is identical to
that of \cite[Thm.~2.8]{BS}:

\begin{proposition}\label{vjac}
There exists a rational constant $\mathcal J$ such that, for any measurable function $\phi$
  on $V(\R)$, we have
\begin{equation}\label{Jac}
\frac{|\mathcal J|}{n_i}\int_{R^{(i)}}
\int_{G(\R)}\phi(g\cdot V^{(i)}(\vec a))\,dg\,d\vec a=\int_{V^{(i)}}\phi(v)dv.
\end{equation}
\end{proposition}

\noindent We may use Proposition \ref{vjac} 
to give a convenient expression for the volume of the multiset $\mathcal R_X$:
\begin{eqnarray}\nonumber
\!\!\!\!\int_{\mathcal R_X}\!\!\!\!\!dv=\int_{\FF\cdot R^{(i)}(X)}\!\!\!\!\!dv&=&
|\mathcal J|\cdot\int_{R^{(i)}(X)}\int_{\FF}dg\,d\vec  a \\ 
\label{volexp} &=&|\mathcal J|\cdot \Vol(G(\Z)\backslash G(\R))\cdot \int_{R^{(i)}(X)}d\vec a.
\end{eqnarray}

\subsection{Congruence conditions}\label{sec:cong}

In this subsection, we prove the following version of Theorem \ref{thmcount} where we
count elements of $V(\Z)$ satisfying a {\em finite} set of congruence conditions:

\begin{theorem}\label{cong2}
Suppose $S$ is a subset of $V(\Z)$ defined by 
congruence conditions modulo finitely many prime powers. Then 
\begin{equation}\label{ramanujan}
N(S\cap V^{(i)};X)
  \;=\; N(V(\Z)^{(i)};X)\cdot
  \prod_{p} \mu_p(S)+o(X^{n/k}),
\end{equation}
where $\mu_p(S)$ denotes the $p$-adic density of $S$ in $V(\Z)$.
\end{theorem}

\begin{proof}
Suppose $S$ is defined by congruence conditions modulo some integer $m$. Then $S$ may be viewed as the union of (say) $r$ translates ${\mathcal L}_1,\ldots,{\mathcal L}_r$ of the lattice $m\cdot V(\Z)
$. For each such lattice translate ${\mathcal L}_j$, we may use formula \eqref{avg} and the discussion following that formula to compute $N(S;X)$, but where each $d$-dimensional volume is scaled by a factor of $1/m^d$ to reflect the fact that our new lattice has been scaled by a factor of $m$. For a fixed value of $m$, we thus obtain
\begin{equation}\label{lat}
N({\mathcal L}_j\cap V^{(i)};X) = m^{-n} \Vol(R_X) +o(X^{n/k}).
\end{equation}
Summing (\ref{lat}) over $j$, and noting that $rm^{-n} = \prod_p \mu_p(S)$, yields Theorem~\ref{cong2}.
\end{proof}

We will also have occasion to use the following weighted version of
Theorem \ref{cong2}; the proof is identical.

\begin{theorem}\label{cong3}
  Let $p_1,\ldots,p_r$ be distinct prime numbers. For $j=1,\ldots,r$,
  let $\phi_{p_j}:V(\Z)\to\R$ be a $G(\Z)$-invariant function on
  $V(\Z)$ such that $\phi_{p_j}(x)$ depends only on the congruence
  class of $x$ modulo some power $p_j^{a_j}$ of $p_j$.  Let
  $N_\phi(V(\Z)\cap V^{(i)};X)$ denote the number of irreducible
  $G(\Z)$-orbits in $V(\Z)\cap V^{(i)}$ having height bounded by $X$,
  where each orbit $G(\Z)\cdot B$ is counted with weight
  $\phi(B):=\prod_{j=1}^r\phi_{p_j}(B)$. Then 
\begin{equation}
N_\phi(V(\Z)\cap V^{(i)};X)
  = N(V(\Z)\cap V^{(i)};X)
  \prod_{j=1}^r \int_{B\in V({\Z_{p_j}})}\tilde{\phi}_{p_j}(B)\,dB+o(X^{n/k}),
\end{equation}
where $\tilde{\phi}_{p_j}$ is the natural extension of ${\phi}_{p_j}$
to $V(\Z_{p_j})$ by continuity, and $dB$ denotes the additive
measure on $V(\Z_{p_j})$ normalized so that $\int_{B\in
  V(\Z_{p_j})}dB=1$.
\end{theorem}

\subsection{Proof of Propositions~\ref{hard2} and \ref{gzbigstab}}\label{sec:reducibility}

We may use the results of \S \ref{sec:cong} to prove Propositions~\ref{hard2} and \ref{gzbigstab} on estimates for reducible points in the main body and points with large stabilizer, respectively.  Indeed, to prove Proposition~\ref{hard2}, we note that if an element $B\in V(\Z)$ 
is reducible over $\Q$ then it also must be reducible modulo $p$ for every $p$.  

Let $S^\red$ denote the set of elements in $V(\Z)$ that are reducible over $\Q$, and let $S^\red_p$ denote the set of all elements in $V(\Z)$ that are reducible mod $p$.  Then $S^\red\subset \cap_p S^\red_p$.  Let $S^\red(Y)=\cap_{p<Y}S^\red_p$ for any positive integer $Y$, and let us use as before $V(\phi)$ to denote the set $B\in V(\Z)$ such that $b_{\min}\neq 0$.  Then the proof of Theorem~\ref{cong2} (without assuming Propositions~\ref{hard2} and \ref{gzbigstab}!) gives that
\begin{equation}\label{SYcount}
N^\ast(S^\red(Y)\cap V(\phi);X)
  \;\leq\;  N^\ast(V(\phi);X)\cdot
  \prod_{p<Y} \mu_p(S^\red_p)+o(X^{n/k}).
\end{equation}
Note that the inequality in \eqref{SYcount} also holds when the product is over subsets of primes $p < Y$.

To estimate $\mu_p(S^\red_p)$ in each case (for $p$ large enough), first recall that an element $v$ of $V(\Z)$ is reducible mod $p$ if and only if any one of the covariant binary quartic forms or ternary cubic forms coming from $v$, considered mod $p$, corresponds to a trivial Selmer element, i.e., has a root or a flex defined over $\Fp$, respectively. In each of the seven cases, we show that for infinitely many $p$, there exists $\delta \in (0,1)$ such that $$\mu_p(S^\red_p) \leq 1 - \delta + O(p^\beta)$$
where $\beta = -1/2$ or $-1$.
The latter five cases will use the first two cases, and some cases will need a weak form of the Hasse bound
$$\# E(\Fp) = p + O(\sqrt{p})$$
for elliptic curves $E$ to conclude that all of the fibers of the covariant binary quartic or ternary cubic maps are roughly the same size.

\begin{enumerate}
\item {\bf Binary quartic forms.} It is easy to check that $\frac{1}{4}(p^5-p^4+p^3-3p^2+2p)$ of the binary quartic forms in $\Fp$ are irreducible over $\Fp$ (out of $p^5$ total), so in this case, we have $\mu_p(S^\red_p) \leq 3/4 + O(1/p).$

\item {\bf Ternary cubic forms.} We will show that a positive density of smooth ternary cubic forms over $\Fp$ have Jacobians with a nontrivial rational $3$-torsion point, and most of those will not have a rational flex. We first claim that the density of the singular ternary cubics over $\Fp$ is $O(1/p)$; this follows from using the Grothendieck-Lefschetz trace formula to count points in the smooth locus of the moduli space of ternary cubics (there is no $1/\sqrt{p}$ term because the moduli space is rational and thus has vanishing $H^1$).

Note that since $\# E(\Fp)/3E(\Fp) = \#E[3](\Fp)$, each elliptic curve $E$ over $\Fp$ arises as the Jacobian of a ternary cubic form over $\Fp$ the same number of times (in fact, exactly $\# \GL_3(\Fp)$ times).
Now consider the degree $8$ forgetful map from the modular curve $Y_1(3)$ to the moduli space $\mathcal{M}_{1,1}$ of elliptic curves; over $\Fp$, the image consists of the elliptic curves over $\Fp$ with a nontrivial $3$-torsion point defined over $\Fp$, and this image must have density at least $1/8 + O(1/p)$ (the error term due to the cusps of the genus $0$ curve $X_1(3)$). Finally, if $E[3](\Fp) = 3$ or $9$, then a density of $2/3$ or $8/9$ (respectively) of the ternary cubics with Jacobian $E$ will not have a rational flex. Combining all of these proportions shows that the density of irreducible ternary cubics over $\Fp$ is at least $1/12 + O(1/p)$, so
$$\mu_p(S^\red_p) \leq 11/12 + O(1/p).$$

\item {\bf Bidegree $(2,2)$ forms.} A bidegree $(2,2)$ form is irreducible if and only if both covariant binary quartics are irreducible. Consider the map $\phi$ from bidegree $(2,2)$ forms to one of the covariant binary quartics. Given a bidegree $(2,2)$ form $v \in V(\Fp)$, if $f = \phi(v)$ is irreducible, then $f$ corresponds to a genus one curve $C$ (isomorphic to its Jacobian $E$) and a degree $2$ line bundle $L$ such that $L$ is not isomorphic to $\mathcal{O}(2Q)$ for a point $Q \in E(\Fp)$. Thus, the group $E(\Fp)/2E(\Fp)$ is nontrivial. Adding any nonzero point $P \in E(\Fp)$ in all but one (nonzero) coset of $2E(\Fp)$ to the line bundle $L$ gives a second line bundle $L'$ with $(C,L')$ corresponding to an irreducible binary quartic. Thus, a positive proportion (either $\frac{1}{2}-\frac{1}{\#E(\Fp)}$ or $\frac{3}{4}-\frac{1}{\#E(\Fp)}$) of the bidegree $(2,2)$ forms above $f$ are irreducible. Combining this with case 1 and the Hasse bound shows that for large enough $p$, we have
$$\mu_p(S^\red_p) \leq 7/8 + O(1/\sqrt{p}).$$

\item {\bf Rubik's cubes.} The argument for the case of Rubik's cubes is very similar to that for bidegree $(2,2)$ forms, replacing covariant binary quartics with covariant ternary cubics. Given an irreducible ternary cubic $f$ corresponding to a genus one curve $C$, degree $3$ line bundle $L$, and Jacobian $E$, we want to find points $P \in E(\Fp)$ such that $L+P$ and $L-P$ (as degree $3$ line bundles) are not isomorphic to $\mathcal{O}(3Q)$ for any $Q$. Again, we have that $E(\Fp)/3E(\Fp)$ is nontrivial (since $f$ is irreducible) so has size either $3$ or $9$. All nonzero points $P$ not in two of the nonzero cosets of $3E(\Fp)$ will thereby give an irreducible Rubik's cube (with three irreducible ternary cubics). For large enough $p$, we have
$$\mu_p(S^\red_p) \leq 35/36 + O(1/\sqrt{p}).$$

\item {\bf Doubly symmetric Rubik's cubes.} We combine the argument for Rubik's cubes with the observation that $2$-torsion points $P$ are always equal to $3P$. Thus, a doubly symmetric Rubik's cubes with one irreducible ternary cubic will be irreducible, and since there are either zero, one, or three nontrivial $2$-torsion points for any elliptic curve, we have $\mu_p(S^\red_p) \leq 35/36 + O(1/p).$
\item {\bf Triply symmetric hypercubes.} We combine the argument for bidegree $(2,2)$ forms with the observation that $3$-torsion points $P$ are always equal to $2(-P)$. So a triply symmetric hypercube with one irreducible binary quartic will be irreducible, and since an elliptic curve has zero, two, or eight nontrivial $3$-torsion points, we compute $\mu_p(S^\red_p) \leq  15/16 + O(1/p).$
\item {\bf Hypercubes.} We use a similar argument as the previous cases, but we now need to eliminate three cosets of $2E(\Fp)$. It is easy to check that the number of binary cubic forms over $\Fp$ with three distinct roots in $\Fp$ is $\frac{1}{6}p(p^2-1)(p-1)$, so $E(\Fp)/2E(\Fp)$ has order $4$ for $1/6 + O(1/p)$ of elliptic curves over $\Fp$. Thus, we obtain, for large enough $p$,
$$\mu_p(S^\red_p) \leq (1-1/4 \cdot 1/6 \cdot 1/4) + O(1/\sqrt{p}) = 95/96 + O(1/\sqrt{p}).$$
\end{enumerate}

Combining with (\ref{SYcount}), we see that
$$\lim_{X\to\infty}\frac{N^\ast(S^\red\cap V(\phi);X)}{X^{n/k}}
 \;\ll\;  \prod_{p<Y} \mu_p(S^\red_p) \;\ll \; \prod_{p<Y}\Bigl(1-\delta+O(p^\beta)\Bigr).
 $$
 When $Y$ tends to infinity, the product on the right tends to 0, proving Proposition~\ref{hard2}.
  
We may proceed similarly with Proposition~\ref{gzbigstab}.  If an element $B\in V(\Z)$ with nonzero discriminant has a nontrivial stabilizer in $G(\Q)$, then any of the covariant binary quartic forms or ternary cubic forms has a nontrivial stabilizer, or equivalently, the corresponding Jacobian has a rational $2$- or $3$-torsion point.
Let $S^\bigstab\subset V(\Z)$ denote the elements $B\in V(\Z)$ that have nontrivial stabilizers in $G(\Q)$ and let $S_p^\bigstab\subset V(\Z)$ denote the elements $B\in V(\Z)$ such that $B$ modulo $p$ has a nontrivial stabilizer in $G(\Fp)$.  Let $S^\bigstab(Y)=\cap_{p<Y} S^\bigstab_p$.  Then we claim that in each case, we have
\begin{equation} \label{eq:bigstab}
\mu_p(S^\bigstab_p) \leq (1-\delta') + O(p^\beta)
\end{equation}
for $\beta = -1/2$ or $-1$ and some $\delta' \in (0,1)$.

We need only compute $\delta'$ for binary quartics and ternary cubics; inequality \eqref{eq:bigstab} for the other cases will follow from the Hasse bound argument because inclusion in $V^{\bigstab}$ is determined by the stabilizer for any of the covariant forms. For binary quartic forms, note that $\frac{1}{3}p(p^2-1)^2$ of the $p^5$ binary quartic forms over $\Fp$ factor into an irreducible cubic factor and a linear factor over $\Fp$. All of the Jacobians of these curves have no $2$-torsion point over $\Fp$ since the cubic does not factor, so $\mu_p(S^\bigstab_p) \leq 2/3 + O(1/p)$. For ternary cubic forms, we want to find the density of ternary cubics whose Jacobians have no $3$-torsion point over $\Fp$. As above, consider the forgetful map ${Y}_1(3) \to \mathcal{M}_{1,1}$; both the source and the target over $\Fp$ are genus zero curves (with cusps), and since the fibers have order $0$, $2$, or $8$, we must have that at least $1/2 + O(1/p)$ of the points in $\mathcal{M}_{1,1}$ are not in the image of the map, i.e., $\mu_p(S^\bigstab_p) \leq 1/2 + O(1/p)$.

Finally, we have by the same argument as for Proposition \ref{hard2} that 
$$\lim_{X\to\infty}\frac{N^\ast(S^\bigstab\cap V(\phi);X)}{X^{n/k}}
 \;\ll\; \prod_{p<Y} \mu_p(S^\bigstab_p) \;\ll\; \prod_{p<Y}\Bigl(1-\delta'+O(p^\beta)\Bigr)
 $$
and letting $Y$ tend to infinity proves Proposition~\ref{gzbigstab}.

\section{Sieving to Selmer elements} \label{sec:sieveSelmer}

We have seen that locally soluble orbits of elements of $V(\Q)$ correspond to elements in the $d$-Selmer group $S(E)$ of elliptic curves $E$ in the family $F$, where $F=F_0$, $F_1$,
  $F_1(2)$, $F_1(3)$, or $F_2$ and $d=2$ or~$3$.  More precisely, recall from Theorem \ref{thm:parametrizations}(c) that irreducible such orbits correspond to elements of the Selmer group $S(E)$ that are not in the subgroup $S'=S'(E)$ given by the image in $S(E)$ of the marked points on $E$. 
  
Let $\Phi$ be a subfamily of $F$ that is defined by local congruence conditions modulo prime powers.  For each prime $p$ we assume that the elliptic curves over $\Z_p$, that the congruence conditions modulo powers of $p$ define, form a closed subset of $\Z_p^m$ with boundary of measure 0.
We use $\Phi^{\inv}$ to denote the set $\{\vec a\in \Z^m:\vec a = \vec a(E) \mbox{ for some } E\in \Phi\}$, and $\Phi^\inv_p$ to denote the
$p$-adic closure of $\Phi^{\inv}$ in $\Z_p^{m}$ by $\Phi^\inv_p$. We
say that such a subfamily $\Phi$ of elliptic curves over $\Q$ is {\em acceptable  at
  $p$} if $\Phi_p^\inv$ contains all elliptic curves $E$ in the family such that $p^2\nmid \Delta(E)$.  The subfamily $\Phi$ of elliptic curves is called {\em acceptable} if it is large at all but finitely
many primes $p$.  In this section, we prove Theorem~\ref{congmain} for this slightly more general definition of an acceptable subfamily, using an appropriate sieve applied to the counts of $G(\Z)$-orbits on $V(\Z)$ having bounded height as obtained in Section~7. 

\subsection{A weighted set $U(\Phi)$ in $V(\Z)$ corresponding to a large family $\Phi$}

Theorem \ref{thm:parametrizations}(c) implies that
non-$S'$ elements of the Selmer group of the Jacobian of the
elliptic curve $E(\vec a)\in F$ for $\vec a\in \Z^m$ are in bijective
correspondence with $G(\Q)$-equivalence classes of irreducible locally
soluble elements $B\in V(\Z)$ having invariants $M^ia_i$ and $M^i a'_i$ for all $i$;  
in this bijection, we have $H(B)=M^6H(C)$.  Let us write $\vec a_M$ to denote the vector $\vec a$ in which each $a_i$ and $a_i'$ are replaced by $M^ia_i$ and $M^ia_i'$, respectively.

In~\S\ref{countsec}, we computed the asymptotic number of
$G(\Z)$-equivalence classes of irreducible elements $B\in V(\Z)$
having bounded height.
In order to use this to compute the number of irreducible locally soluble $G(\Q)$-equivalence classes of elements $B\in V(\Z)$ having invariants in 
\begin{equation}\label{newinvts}
\{\vec a_M: \vec a\in \Phi^\inv\}
\end{equation}
and bounded height (where $\Phi$ is any large family), we need to count each 
$G(\Z)$-orbit $G(\Z)\cdot B$
with a weight of $1/n(B)$, where $n(B)$ is equal to the number of
$G(\Z)$-orbits inside the $G(\Q)$-equivalence class of $B$
in $V(\Z)$.  

To count the number of irreducible locally soluble $G(\Z)$-orbits having invariants in the set (\ref{newinvts}) and  bounded height, where each orbit $G(\Z)\cdot B$ is weighted by $1/n(B)$, it suffices to count the number of such $G(\Z)$-orbits of
 bounded height such that each orbit
$G(\Z)\cdot B$ is weighted instead by $1/m(B)$, where
$$m(B):=\sum_{B'\in O(f)} \frac{\#\Aut_\Q(B')}{\#\Aut_\Z(B')}\; = 
\sum_{B'\in O(f)} \frac{\#\Aut_\Q(B)}{\#\Aut_\Z(B')}\;
;$$ here $O(f)$ denotes a set of orbit representatives for the action of $G(\Z)$ on the $G(\Q)$-equivalence class of $B$ in $V(\Z)$, and $\Aut_\Q(B')$ (resp.,
$\Aut_\Z(B')$) denotes the stabilizer of $B'$ in $G(\Q)$ (resp.,
$G(\Z)$). The reason it suffices to weight by $1/m(B)$ instead of $1/n(B)$ is that we have shown in the
proof of Proposition~\ref{gzbigstab} that all but a negligible number $o(X^{n/k})$ of
$G(\Z)$-orbits having bounded height have trivial
stabilizer in $G(\Q)$ (and thus also in $G(\Z)$), while the number of elliptic curves in $\Phi$ of bounded height is $\gg X^{n/k}$.

We use $U(\Phi)$ to denote the weighted set of all locally soluble elements in $V(\Z)$ having invariants in the set (\ref{newinvts}), where each element of $B\in U(\Phi)$ is assigned a weight of $1/m(B)$.  Then we have concluded that the weighted number of irreducible $G(\Z)$-orbits of height less than $M^{6}X$ in $U(\Phi)$ is asymptotically equal to the number of elements in $S(E)\setminus S'(E)$ for elliptic curves $E$ of height less than $X$ in $\Phi$.  

The global weights $m(B)$ assigned to elements $B\in U(\Phi)$ are useful for the following reason.  
For a prime $p$ and any element $B\in V(\Z_p)$, 
define the local weight $m_p(B)$ by
$$m_p(B):=\sum_{B'\in O_p(B)} \frac{\#\Aut_{\Q_p}(B)}{\#\Aut_{\Z_p}(B')},$$ where $O_p(B)$ denotes 
a set of orbit representatives for the action of $G(\Z_p)$ on the $G(\Q_p)$-equivalence class of $B$ in $V(\Z_p)$, and 
$\Aut_{\Q_p}(B)$
(resp., $\Aut_{\Z_p}(B)$) denotes the stabilizer of $B$ in $G(\Q_p)$ (resp.,
$G(\Z_p)$). 
Using the fact that $G$ has class number one,
by an argument identical to \cite[Prop.~3.6]{BS},
we have the following identity:
 \begin{equation}\label{wtmult}
 m(B)=\prod_pm_p(B).
 \end{equation}
Thus the global weights of elements in $U(\Phi)$ are products of local weights, so we may express the global density of elements $U(\Phi)$ in $V(\Z)$ as products of local densities of the closures of the set $U(\Phi)$ in $V(\Z_p)$.  We consider these local densities next. 

\subsection{Local densities of the weighted sets $U(\Phi)$}

Suppose that $\Phi$ is a large subfamily of elliptic curves in $F$,
and for each prime $p$, let $\Phi_p$ denote the resulting family of
curves defined by congruence conditions over $\Z_p$ .  Let $U(\Phi)$
denote the associated weighted set in $V(\Z)$, and let $U_p(\Phi)$
denote the $p$-adic closure of $U(\Phi)$ in $V({\Z_p})$.  We can now
determine the $p$-adic density of $U_p(\Phi)$, where each element
$B\in U_p(\Phi)$ is weighted by $1/m_p(B)$,
in terms of a {\it local $(p$-adic$)$ mass} $M_p(V,\Phi)$ involving all
elements of $E(\Q_p))/d E(\Q_p)$ for curves $E$ in $\Phi$ over
$\Q_p$; the proof is identical to \cite[Prop.~3.9]{BS}:

\begin{proposition}\label{denel}
 Let $\mathcal J$ be the constant 
 of Proposition $\ref{vjac}$, and let $\Phi$ be any large subfamily of elliptic curves
 in $F$. Then
$$\int_{U_p(\Phi)}\frac{1}{m_p(v)}dv\,=\,|M^n\mathcal J|_p\cdot \Vol(G(\Z_p))\cdot
M_p(V,\Phi),$$
where
$$M_p(V,\Phi)\,:=\,\displaystyle{
\int_{E=E(\vec a)\in \Phi_p}\sum_{\sigma\in \textstyle\frac{E(\Q_p)}{d E(\Q_p)}}\frac{1}{\#E[d](\Q_p)}d\vec a.}$$
\end{proposition}

In the analogous manner, if $\Phi$ is a large subfamily of $F$,
then we may define $M_p(\Phi)$ to be the measure of $\Phi_p^\inv$
with respect to the measure $d\vec a$ on $\Z_p^{m}$, where the
measure $d\vec a$ on $\Z_p^m$ is normalized so that the total measure is
$1$.  That is, we have
\begin{equation}\label{ecden}
M_p(\Phi)=\int_{E=E(\vec a)\in \Phi_p}d\vec a.
\end{equation}
In Section 9, we will be interested in comparing the masses $M_p(V,\Phi)$ and $M_p(\Phi)$.  

\subsection{Squarefree conditions} \label{sec:squarefreeconditions}

In this section, we describe conditions for elements in $V(\Z)$ that will be removed in the sieve for Selmer elements. For example, we show that conditions like insolubility at $p$ imply that $p^2$ divides the discriminant (or a specific factor of it).

For two of the cases, the discriminant polynomial factors as a polynomial with repeated factors, so sieving naively for elements with squarefree discriminant would remove all elements. Specifically, for Case 5 (doubly symmetric Rubik's cubes), for an elliptic curve in $F_1(2)$, we have a factorization of the discriminant $\Delta = 16a_4^2(-4a_4+a_2^2)$ (see \cite[\S 5.2]{BH}); let $\alpha(v) = a_4$ and $\Delta'(v) = -4a_4+a_2^2$, which are both degree $12$ invariants of $V$. Similarly, for Case 6 (triply symmetric hypercubes), the discriminant polynomial for an elliptic curve in $F_1(3)$ factors as a rational multiple of $a_3^3(a_1^3-27 a_3)$ \cite[\S 6.3]{BH}; in this case, let $\alpha(v) = a_3$ and $\Delta'(v) = a_1^3-27 a_3$, which are degree $6$ invariants of $V$. Recall that we define the {\em reduced discriminant} $\Delta_{\red}(v)$ to be the squarefree part of the discriminant, so it is $\alpha(v)\Delta'(v)$ for these two cases (and just $\Delta(v)$ for the other cases).

We will use the following definition repeatedly in the sequel:

\begin{definition}
For any integer polynomial $g(t_1,\ldots, t_r)$ where $p^2$ divides $g(\vec{b})$ with $\vec{b} \in \Z^r$, we say that $g(\vec{b})$ is a multiple of $p^2$ for ``mod $p$ reasons'' if $p^2 \mid g(\vec{b}')$ for all $\vec{b}' \equiv \vec{b} \pmod{p}$, and for ``mod~$p^2$ reasons'' otherwise.
\end{definition}

\begin{proposition} \label{prop:psquareddivides}
Let $v \in V(\Z)$. If the covariant binary quartic or ternary cubic forms associated to $v$ are insoluble at $p$ $($i.e., do not have a $\Qp$-point$)$ or if $m_p(v) \neq 1$, then 
\begin{enumerate}
\item[\rm{(a)}] $p^2$ divides the discriminant $\Delta(v)$, and
\item[\rm{(b)}] in Cases $5$ and $6$, either $p^2$ divides $\alpha(v)$ for mod $p$ reasons, or $p$ divides both $\alpha(v)$ and $\Delta'(v)$, or $p^2$ divides $\Delta'(v)$ for mod $p^2$ reasons.
\end{enumerate}
\end{proposition}

\begin{proof}
For binary quartics and ternary cubics (Cases 1 and 2), this result is proved in \cite[Proposition 3.18]{BS} and \cite[Proposition 38]{BS2}, respectively.

Part (a) follows directly from Cases 1 and 2, since the discriminant of a bidegree $(2,2)$ form, a Rubik's cube, or a hypercube coincide with the discriminant of any of the covariant binary quartics or ternary cubics. For example, if a bidegree $(2,2)$ form (or a hypercube) $v$ gives rise to a covariant binary quartic $f$ that is insoluble at $p$, then $\Delta(v) = \Delta(f)$ is divisible by $p^2$. Furthermore, if one of the covariant binary quartics $f$ of $v \in V(\Z)$ has $m_p(f) \neq 1$, then there is an element of $\SL_2(\Qp) \setminus \SL_2(\Zp)$ that takes $f$ to another integral binary quartic. By taking the identity in all other factors of $\SL_2(\Qp)$, we then obtain a non-integral element of $\SL_2(\Qp)^r$ (for $r = 2$ or $4$, respectively) taking $v$ to another element of $V(\Zp)$, so $m_p(v) \neq 1$. The argument for Rubik's cubes is analogous.

For a doubly symmetric Rubik's cube $v$ (Case 5), if any of the covariant ternary cubics of $v$ is insoluble at $p$, then all the covariant cubics are. Let $f$ be the covariant cubic $\det(Ax + By + Cz)$, when we view $v$ as a triple of symmetric matrices $(A,B,C)$. By the argument for ternary cubics (see \cite[Proposition 38]{BS2}), we find that insolubility of $f$ implies that $f$ modulo $p$ factors over $\overline{\Fp}$ into linear factors. The three singularities $[x:y:z]$ of the curve $f = 0$ modulo $p$ correspond to where the rank of the matrix $Ax + By + Cz$ (modulo $p$) drops by $2$. Thus, a change of coordinates (over $\overline{\Fp}$) will take $v$ to the triple $(E_{11},E_{22},E_{33})$ modulo $p$, where $E_{ij}$ is the $3 \times 3$ matrix with a $1$ in the $ij$th entry and $0$ elsewhere. An easy explicit computation shows that $p^2$ divides $\alpha(v)$ for $v$ congruent to $(E_{11},E_{22},E_{33})$ modulo $p$ (so $p^2$ divides $\alpha(v)$ for ``mod $p$ reasons'').

If a doubly symmetric Rubik's cube $v = (A,B,C)$ has $m_p(v) \neq 1$, then there exists a nontrivial element $\gamma = (\gamma_1, \gamma_2) \in \GL_3(\Qp)^2$, not in $\GL_3(\Zp)^2$, such that $\gamma(v) \in V(\Zp)$ and $(\det \gamma_1) (\det \gamma_2)^2 = 1$ (where, say, $V = V_1 \otimes \Sym_2(V_2)$ and $\gamma_1$ acts on $V_1$ and $\gamma_2$ acts on $V_2$). 
Without loss of generality, by scaling, we may take $\gamma_2$ to have determinant $1$, $p^{-1}$, or $p^{-2}$ (so $\gamma_1$ has determinant $1$, $p^2$, or $p^4$, respectively).

First suppose $\gamma_2$ has determinant $1$. If $\gamma_1$ is nontrivial, then $\gamma_1$ also takes the covariant ternary cubic $f = \det(Ax + By + Cz)$ to an integral ternary cubic form $f'$. A change of basis puts $\gamma_1$ into the form $\left( \begin{smallmatrix} p^r &  & \\ & p^s & \\ & & p^t \end{smallmatrix} \right)$, where $r+s+t = 0$ and $r \leq s \leq t$ with at least one nonzero. Then either $f$ or $f'$ has a linear factor when reduced modulo $p$; assume without loss of generality $f$ factors into a linear and a quadratic factor modulo $p$. Then the curve $f = 0$ modulo $p$ has at least two singularities (over $\overline{\Fp}$), corresponding to where the rank of the matrix $Ax + By + Cz$ modulo $p$ drops by $2$. As above, a change of coordinates (over $\overline{\Fp}$) will take $v$ to the triple $(E_{11},E_{22},C')$ modulo $p$ for some symmetric matrix $C'$. An explicit computation\footnote{This computation may be done without having an explicit formula for $\alpha$ by computing the usual invariants for the ternary cubic $f$, which are degree $12$ and $18$ in the entries of $V$, and comparing them modulo low powers of $p$ to the degree $6$ and $12$ $G$-invariants of $V$, the latter of which is $\alpha$.} shows that $p^2$ divides $\alpha(v)$ for $v$ congruent to $(E_{11},E_{22},C')$ modulo $p$, so $p^2$ divides $\alpha(v)$ for mod $p$ reasons. Now if $\gamma_1$ is trivial, then $\gamma_2$ must be nontrivial, and $\gamma_2$ takes the other covariant ternary cubic $g$ to an integral ternary cubic. Again, we may take $\gamma_2$ to be of the form $\left( \begin{smallmatrix} p^r &  & \\ & p^s & \\ & & p^t \end{smallmatrix} \right)$, where $r+s+t = 0$ and $r \leq s \leq t$ with at least one nonzero. This implies that either $v$ or $\gamma_2(v)$, up to an appropriate change of coordinates, has the following factors of $p$ in each of the three matrices: $\left( \begin{smallmatrix} p^2 & p & p \\ p &  & \\ p & &  \end{smallmatrix} \right)$. Then the ternary cubic $f$ is a multiple of $p^2$, so both the invariants $\alpha(v)$ and $\Delta'(v)$ are divisible by $p^2$.

If $\gamma_2$ has determinant $p^{-u}$ for $u = 1$ or $2$, then we can similarly change the basis to make $\gamma_1$ a diagonal matrix $(p^r, p^s, p^t)$ with $r+s+t = 2u$ and $r \leq s \leq t$. To examine how $\gamma$ acts on the coefficients of the covariant ternary cubic form $f(x,y,z)$, first note that $\gamma_2$ sends $f$ to $p^{-u} f$. We claim that that either $\gamma(f)$ modulo $p$ is divisible by $x$, or $f$ modulo $p$ is divisible by $z$, or $\gamma(f)$ is a multiple of $p$. This is a straightforward computation: for $u = 1$, if $s \geq 1$, then under the action of $\gamma$, the coefficients of $x y z$, $y^3$, $y^2 z$, $y z^2$, and $z^3$ in $f$ are all multiplied by positive powers of $p$, so $\gamma(f)$ modulo $p$ is divisible by $x$. If $s \leq 0$, then the coefficients of $x^3$, $x^2 y$, $x y^2$, and $y^3$ are multiplied by negative powers of $p$ by the action of $\gamma$, so since $f$ is integral, it must be divisible by $z$ modulo $p$. For $u = 2$, if $r \leq 0$ and $s \geq 1$, then $x$ divides $\gamma(f)$ modulo $p$; if $r \leq 0$ and $s \leq 0$, then $z$ divides $f$ modulo $p$; and if $(r,s,t) = (1,1,2)$ (the only remaining case), then $p$ divides $\gamma(f)$. Now we may use the arguments from the case of $\det \gamma_2 = 1$, since either $f$ or $\gamma(f)$ has a linear factor when reduced modulo $p$ or $\gamma(f)$ is a multiple of $p$.

For triply symmetric hypercubes $v$ (Case 6), first suppose the covariant binary quartics arising from $v$ are insoluble at $p$. Then they must be a square of a quadratic polynomial modulo $p$ (possibly a fourth power of a linear factor). Viewing $v$ as a pair of binary cubic forms $(A,B)$ in variables $t$ and $u$, we have that the pencil of binary cubic forms has two points where the cubic is in fact the cube of a linear form modulo $p$. In other words, up to appropriate changes of coordinates, we have $A = t^3$ and $B = u^3$ modulo $p$. It is trivial to check in the case that the invariant $\alpha(v)$ is divisible by $p^2$, and for mod $p$ reasons.

If a triply symmetric hypercube $v$ has $m_p(v) \neq 1$, then there exists a nontrivial element $\gamma = (\gamma_1, \gamma_2) \in \GL_2(\Qp)^2$, not in $\GL_2(\Zp)^2$, such that $\gamma(v) \in V(\Zp)$ and $(\det \gamma_1)(\det \gamma_2)^3 = 1$ (where, say, $V = V_1 \otimes \Sym_3(V_2)$ and $\gamma_1$ acts on $V_1$ and $\gamma_2$ acts on $V_2$). A change of basis puts $\gamma_1$ into the form $\left( \begin{smallmatrix} p^r &  \\ & p^s \end{smallmatrix} \right)$ and $\gamma_2$ into the form $\left( \begin{smallmatrix} p^t &  \\ & p^u \end{smallmatrix} \right)$ for integers $r, s, t, u$. Since the action of the diagonal matrices $(p^3 \mathrm{Id}_2, p \mathrm{Id}_2)$ on $V_1 \otimes \Sym_3(V_2)$ is trivial, we may assume $\gamma_2$ has determinant $1$ or $p^{-1}$.

If $\gamma_2$ has determinant $1$, so does $\gamma_1$, and $r = -s$ and $t = -u$. First suppose $\gamma_1$ is nontrivial, so without loss of generality, we may take $r > 0$. Then $\gamma_1$ takes the covariant binary quartic $f = \Disc(Ax+By)$ to an integral binary quartic form $f'$. The binary cubic $B$ is thus a multiple of $p$, in which case it is easy to see that $p^3$ divides both $\alpha(v)$ and $\Delta'(v)$.
If instead $\gamma_1$ is trivial and $\gamma_2$ is nontrivial, then since we may take $t > 0$, we find that the binary cubic forms $A$ and $B$ must have multiple factors of $p$ in their coefficients, namely both are of the form $c_3 X^3 + c_2 X^2 Y +  p^r  c_1 X Y^2 + p^{3r}  c_0 Y^3$ for $c_i \in \Zp$. Again, this immediately implies that both $\alpha(v)$ and $\Delta'(v)$ are divisible by $p^3$.

If $\gamma_2$ has determinant $p^{-1}$, then we have $r = 3-s$ and $u = -t-1$. Then the action of $(\gamma_1, \gamma_2)$ on a pair of binary cubics $(A,B)$ with integral coefficients $((a_0,a_1,a_2, a_3),(b_0,b_1,b_2,b_3))$ produces a pair of cubics whose corresponding (integral) coefficients are scaled by the following powers of $p$:
$$((r+3t, r+t-1,r-t-2,r-3t-3),(-r+3t+3,-r+t+2,-r-t+1,-r-3t)).$$
Note that the powers for the coefficients of $B$ are negated and reversed of those for $A$. When any of the powers are negative, we find that the corresponding coefficient is divisible by the negative of that power of $p$. For example, if $r = 2$ and $t = 1$, then the integrality of both $(A,B)$ and $\gamma(A,B)$ implies that $p$ divides $a_2$, $p^4$ divides $a_3$, $p^2$ divides $b_2$, and $p^5$ divides $b_3$. Explicit computations with the invariants $\alpha$ and $\Delta'$ in this case give the following implications:
\begin{enumerate}
\item[(i)] if $p \mid a_2, a_3, b_2, b_3$, then $p^2 \mid \alpha$ and $p^2 \mid \Delta'$
\item[(ii)] if $p \mid a_3, b_1, b_2, b_3$, then $p \mid \alpha$ and $p \mid \Delta'$
\item[(iii)] if $p \mid b_0, b_1, b_2, b_3$ (i.e., $B$ is a multiple of $p$), then $p^3 \mid \alpha$ and $p^3 \mid \Delta'$
\end{enumerate}
Without loss of generality, we may assume $r \geq 2$ and $t \geq 0$. Then one of the above three cases holds unless $r = 3t + 3$ or $r = t + 2$. If $r = 3t + 3$, then $p^i$ divides $b_i$ for $i = 1, 2, 3$, so we compute that $p^2$ divides $\alpha$. If $r = t+2$, then $p$ divides $a_3$, $p$ divides $b_2$, and $p^2$ divides $b_3$; in this case, we compute that $p^2$ divides $\Delta'$, and for mod $p^2$ reasons.
\end{proof}

\subsection{Uniformity estimates and a squarefree sieve} \label{sec:uniformity}

To obtain the cases of equality in Theorem~\ref{thm:SelmerAverages}, we require a more general version of Theorem~\ref{cong3}, namely one that counts weighted elements of $V(\Z)$, where
the weight functions are defined by appropriate infinite sets of
congruence conditions. 
A function $\phi:V({\Z})\to[0,1]\in\R$ is said to be {\it defined
  by congruence conditions} if, for all primes $p$, there exist
functions $\phi_p:V({\Z_p})\to[0,1]$ satisfying the following
conditions:
\begin{itemize}
\item[(1)] For all $B\in V(\Z)$, the product $\prod_p\phi_p(B)$ converges to $\phi(B)$.
\item[(2)] For each prime $p$, the function $\phi_p$ is 
locally constant outside some closed set $S_p \subset V({\Z_p})$ of measure zero.
\end{itemize}
We say that such a function $\phi$ is {\it acceptable} if for
sufficiently large primes $p$, we have $\phi_{p}(B)=1$ whenever
$p^2$ divides $\Delta_{\red}(B)$. 

Our purpose in this section is to prove the following generalization of Theorem 
\ref{cong3}, which allows for certain infinite sets of congruence conditions:
\begin{theorem}\label{thsqfreetc}
  Let $\phi:V(\Z)\to[0,1]$ be an acceptable function that is defined by
  congruence conditions via the local functions $\phi_{p}:V({\Z_p})\to[0,1]$. Then, with
  notation as in Theorem~$\ref{cong3}$, we have:
\begin{equation}
N_\phi(V(\Z)^{(i)};X)
  \leq N(V(\Z)^{(i)};X)
  \prod_{p} \int_{B\in V({\Z_{p}})}\phi_{p}(B)\,dB+o(X^{n/k}),
\end{equation}
with equality in Cases \equalcases.
\end{theorem}

To prove Theorem~\ref{thsqfreetc}, we follow the method of \cite{geosieve} to
establish the following tail estimate:

\begin{proposition}\label{propunif} Let $\W_p(V)$ be the set of $v \in V(\Z)$ such that $p^2$ divides $\Delta_{\red}(v)$.
 In Cases \equalcases, for any fixed $\varepsilon>0$, we have
\begin{equation}\label{tailestimate}
N(\cup_{p>Y}\W_p(V);X)=O_\varepsilon(X^{n/k}/(Y\log Y)+X^{(n-1)/k})+O(\varepsilon X^{n/k}).
\end{equation}
\end{proposition}

We expect Proposition \ref{propunif} to hold also for Cases 3, 6, and 7 (which together with Theorem \ref{ufcount} would imply that the upper bounds in Theorem \ref{thm:SelmerAverages} are exact averages for those cases).

\begin{proof}

If $p^2\leq X^{1/k}$, then the counting method of \S\ref{sec:cong}, 
with the relevant congruence conditions modulo~$p^2$ imposed, immediately yields
 the individual estimate $N(\W_p(V);X)=O(X^{n/k}/p^2)$ (noting that $R^{(i)}(X)=X^
{1/k}R^{(i)}(1)$).
Hence, to prove Proposition~\ref{propunif}, it suffices to assume that $Y>X^{1/(2k)}$.

To simplify notation, let $\Delta' = \Delta$ for Cases 1, 2, and 4.
Let $\W_p^{(2)}$ denote the set of $B\in V(\Z)$ such that $p^2\mid \Delta'(B)$ for ``mod $p^2$ reasons'', i.e., such that there exists $B' \equiv B \pmod{p}$ such that $p^2 \nmid
\Delta'(B')$. Let $\W^{(1)}_p:=\W_p(V) \backslash\W^{(2)}_p$; for $B \in \W^{(1)}_p$, we have that $p^2$ divides $\Delta'(B)$ for mod $p$ reasons or (only relevant in Case 5) either $p$ divides both $\alpha(B)$ and $\Delta'(B)$ or $p^2$ divides $\alpha(B)$. In Case 5, it is easy to check that if $p^2$ divides $\alpha(B)$ for mod $p^2$ reasons, then $p$ divides $\Delta'(B)$ also.
The two sets $\W_p^{(1)}$ and $\W_p^{(2)}$ are preserved under $G(\Z)$-transformations.

For any $\varepsilon>0$, let
  $\FF^{(\varepsilon)}\subset\FF$ denote the subset of
  elements $na(t_i,u_i)k\in\FF$ such that $t_i$ and $u_i$ are
  bounded above by an appropriate constant to ensure that
  $\Vol(\FF^{(\varepsilon)})=(1-\varepsilon)\Vol(\FF)$. Then
  $\FF^{(\varepsilon)}\cdot R^{(i)}(X)$ is a bounded domain in
  $V(\R)$ that expands homogeneously with
  $X$. By~\cite[Theorem~3.3]{geosieve}, we have
\begin{equation}\label{e1}
\#\{\FF^{(\varepsilon)}\cdot R^{(i)}(X)\bigcap (\cup_{p>Y}\W^{(1)}_p)\}=O_\varepsilon(X^{n/k}/(Y\log Y)+X^{(n-1)/k}).
\end{equation}
Furthermore, the results of \S\ref{sec:mainterm} imply that
\begin{equation}\label{e2}
\#\{(\FF\backslash\FF^{(\varepsilon)})\cdot R^{(i)}(X)\bigcap V(\Z)^
\irr\}=O(\varepsilon X^{n/k}).
\end{equation}
Combining the two estimates (\ref{e1}) and (\ref{e2}) 
yields \eqref{tailestimate} with $\W_p$ replaced with $\W^{(1)}_p$.

Proposition~\ref{propunif} is already known in Cases 1 and 2 of Table \ref{table:Invariants} (cf.\ \cite{BS, BS2}), so we prove the estimate for $\W^{(2)}_p$ only for Cases 4 and 5, where the elements of $V(\Z)$ are (possibly symmetric) $3\times 3\times 3$ matrices.

Suppose $B$ belongs to $\W^{(2)}_p$. Let $f(x,y,z)=\det(B_{1jk}x+B_{2jk}y+B_{3jk}z)$ 
be the first of the three ternary cubic forms arising from $B$; then $\Delta(B) = \Delta(f)$. 
Note that the discriminant of $f$ must also be a multiple of $p^2$ for mod~$p^2$ reasons (for otherwise $B$ would then be in 
$\W_p^{(1)}$).    In \cite[Prop.~25]{BS2}, it was shown that if a ternary cubic form $f$ has discriminant a multiple of $p^2$ for mod~$p^2$ reasons, then 
there is an $\SL_3(\Z)$-transformation taking $f$ to $f'$, such that $p$ divides the $xz^2$- and $yz^2$-coefficients and $p^2$ divides the
$z^3$-coefficient of $f'(x,y,z)$.  Let $B'$ be the result of the corresponding $\SL_3(\Z)$-transformation on $B$; then $f'$ is the first ternary cubic form arising from $B'$. 

Since the $z^3$-coefficient of $f'$ is a multiple of $p^2$, we see that the determinant of the $3\times 3$ matrix $(B'_{3jk})$ is a multiple of $p^2$, and it must be so for 
mod $p^2$ reasons.  It follows that the matrix $(B'_{3jk})$ modulo $p$ has rank 2. By an $\SL_3(\Z) \times \SL_3(\Z)$-transformation in Case 5, or simply
an $\SL_3(\Z)$-transformation in Case 6, we may obtain an element $B''\in V(\Z)$ from $B'$ such that: a) the last row and column of $(B''_{3jk})$ is a multiple of $p$; b) the determinant of the $2\times 2$ matrix $(B''_{3jk})_{1 \leq j,k \leq 2}$ is coprime to $p$; and c) $B''_{333}$ is a multiple of $p^2$.  Note that the first of the associated ternary cubic forms of $B''$ remains $f'$.  The fact that $p$  divides the coefficients of $xz^2$ and $yz^2$ implies that $B''_{133}$ and $B''_{233}$ are also multiples of $p$. 

Define the element $B'''$ by
\begin{equation}\label{eqmatgamma}
    \left(\left(\begin{smallmatrix}
      1&&\\&1&\\&&p^{-1}
    \end{smallmatrix}\right),\left(\begin{smallmatrix}
      1&&\\&1&\\&&p^{-1}
    \end{smallmatrix}\right),\left(\begin{smallmatrix}
      1&&\\&1&\\&&p^{-1}
    \end{smallmatrix}\right)\right)
    \cdot pB \mbox{ or }     \left(\left(\begin{smallmatrix}
      1&&\\&1&\\&&p^{-1}
    \end{smallmatrix}\right),\left(\begin{smallmatrix}
      1&&\\&1&\\&&p^{-1}
    \end{smallmatrix}\right)\right)\cdot pB
\end{equation}
depending on whether we are in Case 5 or 6, respectively.  Then $B'''$ has the same discriminant as $B$ and is in $\W_p^{(1)}$, because its first associated 
ternary cubic form $f'''$ has its $x^3$-, $x^2y$-, $xy^2$-, and $y^3$-coefficients divisible by~$p$. 

We therefore have obtained a discriminant-preserving map $\phi$ from $G(\Z)$-orbits on
$\W_p^{(2)}$ to $G(\Z)$-orbits on $\W_p^{(1)}$. The following lemma
states that this map is at most $3$ to $1$:
\begin{lemma}
Given a $G(\Z)$-orbit on $\W_p^{(1)}$, there are at most three
$G(\Z)$-orbits on $\W_p^{(2)}$ that map to it under $\phi$.
\end{lemma}
\begin{proof}
Let $B$ be an element of $\W_p^{(2)}$ and $f$ its first associated ternary cubic form, i.e., 
$f(x,y,z)=\det(B_{1jk}x+B_{2jk}y+B_{3jk}z)$. 
 If the reduction of $f\in\W^{(2)}_p$ modulo $p$ has a nodal
singularity at $[0:0:1]\in\P^2(\F_p)$ (i.e., its $xz^2$-, $yz^2$-, and $z^3$-coefficients vanish modulo $p$), then the first associated ternary cubic form
$f'''$ of the $3\times3\times3$ matrix \eqref{eqmatgamma}, when reduced modulo $p$, has $z$ as a factor. 
Moreover, for such an $f'''$, the first associated ternary cubic form 
\begin{equation}\label{eqmatgammainverseprep}
    \left(\begin{smallmatrix}
      1&&\\&1&\\&&p
    \end{smallmatrix}\right)\cdot p^{-1}f'''
\end{equation}
of the $3\times3\times3$ matrix 
\begin{equation}\label{eqmatgammainverse}
    \left(\left(\begin{smallmatrix}
      1&&\\&1&\\&&p
    \end{smallmatrix}\right),\left(\begin{smallmatrix}
      1&&\\&1&\\&&p
    \end{smallmatrix}\right),\left(\begin{smallmatrix}
      1&&\\&1&\\&&p
    \end{smallmatrix}\right)\right)
    \cdot p^{-1}B \mbox{ or }     \left(\left(\begin{smallmatrix}
      1&&\\&1&\\&&p
    \end{smallmatrix}\right),\left(\begin{smallmatrix}
      1&&\\&1&\\&&p
    \end{smallmatrix}\right)\right)\cdot p^{-1}B
\end{equation}
can be integral only if the $x^3$-, $x^2y$-, $xy^2$-, and
$y^3$-coefficients of $f'''$ are zero modulo $p$.  Therefore, the preimages
under $\phi$ of the $G(\Z)$-orbit of $B\in\W_p^{(1)}$ are
associated to 
linear factors of the reduction of $f'''$ modulo $p$. The reduction of
$f'''$ modulo $p$ has at most $3$ linear factors, unless $f'''\equiv
0\pmod{p}$, in which case
\eqref{eqmatgammainverse} belongs to $\W_p^{(1)}$. Thus, the map
$\phi:G(\Z)\backslash\W_p^{(2)}\to G(\Z)\backslash\W_p^{(1)}$
is at most $3$ to~$1$, and the lemma follows.
\end{proof}

Therefore, since discriminants less than $X$ can have at most $2k$ distinct prime factors $p> Y > X^{1/(2k)}$, we~obtain
\begin{equation}
N(\cup_{p>Y}\W_p^{(2)};X)\leq
3\cdot 2k\cdot N(\cup_{p>Y}\W_p^{(1)};X)=O_\varepsilon(X^{n/k}/(Y\log Y)+X^{(n-1)/k})+O(\varepsilon X^{n/k}).
\end{equation}
This concludes the proof of the proposition.
\end{proof}

\subsection{Weighted count of elements in $U(\Phi)$ having bounded height}

For a large family $\Phi$, we may now describe the asymptotic number of $G(\Z)$-orbits in $U(\Phi)$ having bounded height.  

\begin{theorem}\label{ufcount}
Let $\Phi$ be any large subfamily of $F$.  Then $N(U(\Phi);M^{6}X)$, the weighted number of $G(\Z)$-orbits in $U(\Phi)$ having height less than $M^6X$, is given by
\begin{equation}\label{eval}
N(U(\Phi);M^6X) \leq M^{6n/k}\cdot\displaystyle{\sum_{i=1}^{N} \frac{\Vol({\mathcal R_X})}{n_i}\cdot
\prod_p\int_{U_p(\Phi)}\frac{1}{m_p(v)}dv\cdot X^{n/k}}+o(X^{n/k}),
\end{equation}
with equality in Cases \equalcases.
\end{theorem}

\begin{proof}
By Theorem~\ref{thmcount}, Theorem~\ref{cong3}, and the multiplicativity of weights (\ref{wtmult}), 
  it follows that for any  fixed positive integer $Y$, we have
  \[
  \lim_{X\rightarrow\infty}\frac{N(V(\Z)\cap[\cap_{p<Y} U_p(\Phi)];M^{6}X)}{(M^6X)^{n/k}}=
   \sum_{i=1}^{N} 
  \frac{\Vol({\mathcal R_1})}
  {n_i}\cdot\prod_{p<Y}\int_{U_p(\Phi)}\frac{1}{m_p(v)}dv,\]
where $V(\Z)\cap[\cap_{p<Y} U_p(\Phi)]$ is viewed as a weighted set in which each element $B$ is weighted by $1/m(B)$.
  Letting $Y$ tend to infinity, we obtain that
  \begin{equation}\label{eq371}\limsup_{X\rightarrow\infty}
  \frac{N(U(\Phi);M^{6}X)}{X^{n/k}}\leq M^{6n/k}\cdot
  \sum_{i=1}^{N} \frac{\Vol({\mathcal R_X})}{n_i}\cdot
\prod_{p}\int_{U_p(\Phi)}\frac{1}{m_p(v)}dv.
\end{equation}

  To obtain a lower bound for $N(U(\Phi);M^6X)$ in Cases \equalcases, we note
  that
  $$V(\Z)\cap \bigcap_{p<Y}U_p(\Phi) \subset \Bigl(U(\Phi)\cup
 \bigcup_{p>Y}W_p\Bigr)$$
(even as weighted sets, since all weights in $U_p(\Phi)$ are less than 1.)
  Hence, by the uniformity estimate of Proposition \ref{propunif}, we have, for any $\varepsilon > 0$, that
  \[
  \liminf_{X\rightarrow\infty}
  \frac{N(U(\Phi);M^{6}X)}{(M^6X)^{n/k}}\geq 
  \sum_{i=1}^{N} \frac{\Vol({\mathcal R_X})}{n_i}\cdot
  \prod_{p<Y}\int_{U_p(\Phi)}\frac{1}{m_p(v)}dv - O_\varepsilon \left(\frac{1}{Y \log Y}\right) - O(\varepsilon).
 \]
  Letting $Y$ tend to infinity then yields
  \begin{equation} \label{eq:lastestimate}
  \liminf_{X\rightarrow\infty}
  \frac{N(U(\Phi);M^{6}X)}{(M^6X)^{n/k}}\geq 
  \sum_{i=1}^{N} \frac{\Vol({\mathcal R_X})}{n_i}\cdot
  \prod_{p}\int_{U_p(\Phi)}\frac{1}{m_p(v)}dv - O(\varepsilon).
  \end{equation}
  Since \eqref{eq:lastestimate} holds for any $\varepsilon > 0$, we conclude that equality holds in \eqref{eval} in these cases.
\end{proof}

It remains to evaluate expression (\ref{eval}) in terms of the total number of elliptic curves in $\Phi$ having height less than $X$.

\section{Proof of Theorem \ref{thm:SelmerAverages}}

\subsection{The number of elliptic curves of bounded height in a large subfamily}

In this subsection, we give an estimate for the number of elliptic curves of bounded height in any large subfamily $\Phi\subset F$, where $F=F_0$, $F_1$, $F_1(2)$, $F_1(3)$, or $F_2$.  Counting such elliptic curves involves understanding, in particular, the count of elliptic curves of bounded height in $F$ having squarefree discriminant. These asymptotic counts of elliptic curves of bounded height in large subfamilies $\Phi$ appear in the denominator when computing the average sizes of Selmer groups in these subfamilies.  

Specifically, we prove the following theorem. 

\begin{theorem}\label{ccount}
Let $\Phi$ be any large subfamily of $F$.  Then the number of elliptic curves $E$ in $\Phi$ with $H(E)<X$ is
given by
\[\int_{{H(\vec a)<X}}d\vec a\;\cdot \prod_p M_p(\Phi)\cdot 
X^{n/k} + o(X^{n/k}) .\]
\end{theorem}

As with the proof of Theorem \ref{thsqfreetc}, to obtain Theorem \ref{ccount} it suffices to prove:

\begin{proposition} \label{prop:denombound}
For any family $F=F_0,F_1,F_1(2),F_1(3),$ or $F_2$, we have
\begin{equation*}
\#\{\vec{a} \in F : H(\vec{a})<X \text{ and } p^2 \mid \Delta'(\vec{a})
\text{ for some $p>Y$}\} = O_\varepsilon(X^{n/k+\varepsilon}/Y)+o(X^{n/k}).
\end{equation*}
\end{proposition}

\begin{proof}[Proof of Proposition $\ref{prop:denombound}$ for $F = F_0, F_1,$ and $F_1(2)$]
We embed $F$ into the space of binary quartic forms, Rubik's cubes, or doubly symmetric Rubik's cubes (Cases 1, 4, or 5 of Table \ref{table:Invariants}) according to whether 
$F= F_0$, $F_1$, or $F_1(2)$, respectively, via the Kostant sections \eqref{eq:BQformula-2O}, \eqref{eq:RCformula-3O}, or \eqref{eq:2SymRCformula-3O} from Section~\ref{sec:integralreps}.
The estimate in Proposition \ref{propunif} then yields the desired result. \end{proof}

\begin{proof}[Proof of Proposition $\ref{prop:denombound}$ for $F=F_1(3)$]
In this case, the polynomial
$\Delta'(a_1,a_3)$
is simply $a_1^3-27a_3$.  If $a_1=O(X^{1/12})$ is fixed, then the number
of values of $a_3=O(X^{1/4})$ such that $p^2 \mid (a_1^3-27a_3)$ is clearly at
most $O(X^{1/4}/p^2+1)$.  Since $p \ll X^{1/8}$, we have
\begin{align*}
\#\{\vec{a} \in F: H(\vec{a})<X \text{ and } p^2 \mid \Delta'(\vec{a})
\text{ for some $p>Y$}\}
 &= \sum_{Y<p\ll X^{1/8}} O(X^{1/12})O(X^{1/4}/p^2+1)  \\
 &= \,O\left(X^{1/3}/(Y\log Y)+X^{5/24}\right),
 \end{align*}
yielding the desired estimate in this case.
\end{proof}

For the family $F=F_2$, the proof of Proposition \ref{prop:denombound} is more involved and comprises the remainder of this subsection. We follow a method analogous to that in~\cite{squarefree-BSW,squarefree-BSW2} (i.e., the ``$Q$-invariant method''), where the density of monic integer polynomials having squarefree discriminant was determined.  Specifically, we embed elements $\vec{a}\in F_2$ whose discriminant is a multiple of $p^2$ for mod $p^2$ reasons into the space $V(\Z)$ of hypercubes so that the invariants match and so that the image contains only hypercubes having ``$Q$-invariant'' equal to~$p$.  We then bound the number of elements lying in the image of this map that have $Q$-invariant greater than $Y$ and height less than $X$, using a variant of the averaging method, thus yielding the desired uniformity estimate. 

 Let $V$ denote the space of hypercubes and $G$ the algebraic group $\SL_2^4/\mu_2^3$ as in Case~7 of Table~\ref{table:Invariants}. 
Let $V_0\subset V$ be the subspace of hypercubes  $B = (b_{ijk\ell})$ such that $b_{1111} = b_{1112} = b_{1121} = b_{1211} = 0$, i.e., those of the form
	\begin{equation} \label{eq:semidistinguishedHC}
	\begin{array}{cc}
	\left(
		\begin{array}{cc}
		\ \ 0\ \ & 0 \\
		\ \ 0\ \ & b_{1122}
		\end{array}
	\right) &
	\left(
		\begin{array}{cc}
		\ \ 0\ \ & b_{1212} \\
		b_{1221} & b_{1222}
		\end{array}
	\right) \\[.175in]
	\left(
		\begin{array}{cc}
		b_{2111} & b_{2112} \\
		b_{2121} & b_{2122}
		\end{array}
	\right) &
	\left(
		\begin{array}{cc}
		b_{2211} & b_{2212} \\
		b_{2221} & b_{2222}
		\end{array}
	\right)
	\end{array},
	\end{equation}
where $b_{ijk\ell}$ are any elements in the base ring/field. Then the discriminant polynomial on $V_0$ factors over $\Q$ as $b_{2111}^2 b_{1122}^2 b_{1221}^2 b_{1212}^2$ times an irreducible polynomial of degree $16$. We call an integer hypercube {\em distinguished} if it is $G(\Z)$-equivalent to a nondegenerate hypercube in $V_0(\Z)$. Lemma~\ref{hyperred}(ii) shows that all distinguished hypercubes are reducible. 

The space $V_0$ is fixed by $U^4 \subset \SL_2^4$, where $U$ denotes the group of lower triangular unipotent matrices. We define the {\em $Q$-invariant} (named after the analogous invariant in \cite{squarefree-BSW}) of a hypercube $B \in V_0$ to be $Q(B) := b_{2111}$; we thus have that $Q(B)^2$ divides the discriminant $\Delta$ of $B$. Note that the $Q$-invariant is a degree $1$ invariant under the action of $U^4$ on $V_0$. In order to define the $Q$-invariant for more general distinguished hypercubes in $V(\Q)$, we show that it is well-defined on ``most'' $G(\Q)$-orbits:

\begin{proposition}
Let $B \in V_0(\Q)$ be a hypercube whose associated elliptic curve $E$ satisfies \linebreak $E(\Q)[2]\neq0$. Then for any $B' \in V_0(\Q)$ that is $G(\Q)$-equivalent to $B$, we have $Q(B') = Q(B)$.
\end{proposition}

\begin{proof}
For a hypercube $B \in V_0(\Q)$, the first covariant binary quartic form $f(w_1,w_2)$ has $w_1^4$-coefficient equal to $0$. Indeed, if $C_1$ denotes the cube formed by the top two $2\times 2$ matrices in (\ref{eq:semidistinguishedHC}), and $C_2$ the cube formed by the bottom two $2\times 2$ matrices, then $f(w_1,w_2):=\Disc(C_1 w_1 + C_2 w_2)$ has $w_1^4$-coefficient 0 because $\Disc(C_1)=0$. This is a reflection of the fact that $B$ is reducible and corresponds to the trivial $2$-Selmer element of its associated elliptic curve $E$. Since $E(\Q)[2] = 0$, the binary cubic form $f(w_1,w_2)/w_2$ is irreducible over $\Q$, and so the only $\SL_2$-transformations of $f$ that preserve this initial zero coefficient are those in $U$. 
If $B' \in V_0(\Q)$ is $G(\Q)$-equivalent to $B$, then the corresponding covariant binary quartic form $f'(w_1,w_2)$ also has the property that $f'(w_1,w_2)$ has $w_1^4$-coefficient equal to $0$, and so $f$ and $f'$ must be $\SL_2(\Q)$-equivalent and thus $U$-equivalent.

In the other three directions, we claim that $B$ and $B'$ again differ only by transformations in $U$. Indeed, consider again the top cube $C_1$ of $B$. 
Its three covariant binary {\em quadratic} forms $q_i(z_1,z_2)$ for $i = 1, 2, 3$ (see \cite{hcl1} for the construction of three binary quadratics from a $2 \times 2 \times 2$ cube) are all of the form $c_i z_2^2$, i.e., have first and second coefficient both zero. Since a binary quadratic form can have only one double root, the only linear transformations that preserve the two zero coefficients of $q_i(z_1,z_2)$ are unipotent. Therefore, the two hypercubes $B$ and $B'$ must be $U^4$-equivalent, and so their $Q$-invariants are equal.
\end{proof}

We may thus define the $Q$-invariant of any hypercube $B' \in V(\Q)$ whose associated elliptic curve $E$ satisfies $E(\Q)[2]\neq 0$ by setting
$Q(B') := Q(B)$, where $B\in V_0(\Q)$ is any element that is $G(\Q)$-equivalent to $B'$.

Now recall that a curve in the family $F_2$ is of the form 
\begin{equation} \label{eq:F2eqn}
y^2+a_1 x y + a_3 y = (x-a_2)(x-a_2')(x-a_2'')
\end{equation}
where $a_i\in\Z$, $a_2 + a_2'+a_2'' = 0$, and $\Delta(\vec{a})\neq0$. 
We embed $F_2$ into the space of hypercubes (see \eqref{eq:hypmin2O} in Case 7(b) in \S \ref{sec:integralreps}) by sending
the curve with invariants $(a_1, a_2, a_2', a_2'', a_3)$ to the hypercube
	\begin{equation} \label{eq:QinvHCembed}
	\begin{array}{cc}
	\left(
		\begin{array}{cc}
		0 & 0 \\
		0 & 1
		\end{array}
	\right) &
	\left(
		\begin{array}{cc}
		0 & 1 \\
		1 & a_1
		\end{array}
	\right) \\[.175in]
	\left(
		\begin{array}{cc}
		1 & 0 \\
		0 & -a_2
		\end{array}
	\right) &
	\left(
		\begin{array}{cc}
		0 & -a_2' \\
		-a_2'' & a_3
		\end{array}
	\right)
	\end{array}
	\end{equation}
having the same invariants. In particular, the discriminant of the hypercube \eqref{eq:QinvHCembed} (equal to the discriminant of the covariant binary quartics) is the discriminant $\Delta$ of the curve \eqref{eq:F2eqn}. Note that hypercubes of the form \eqref{eq:QinvHCembed} lie in the subspace $V_0$, giving the desired map from $F_2$ to distinguished hypercubes.

\begin{proposition}
If $p^2$ divides the discriminant $\Delta$ of an elliptic curve \eqref{eq:F2eqn} for mod $p^2$ reasons, then there exists a hypercube in $V_0(\Z)$ 
that is $G(\Q)$-equivalent to \eqref{eq:QinvHCembed} 
whose $Q$-invariant is $p$.
\end{proposition}
\begin{proof}
Suppose $p^2$ divides $\Delta$ for mod $p^2$ reasons. Then all the covariant binary quartic forms also have discriminant (equal to $\Delta$) a multiple of $p^2$ for mod $p^2$ reasons. There exists a lower triangular unipotent transformation from $U(\Z)$ that transforms such a binary quartic to one with last coefficient divisible by $p^2$ and second-to-last coefficient divisible by $p$. Apply such a transformation to \eqref{eq:QinvHCembed} in the first direction, thus adding an integer multiple of the top two matrices to the bottom two matrices; then the top row is unchanged and the bottom two matrices take the form
	\begin{equation} \label{eq:Qinvcubetransf1}
	\begin{array}{cc}
	\left(
		\begin{array}{cc}
		1 & 0 \\
		0 & \ast
		\end{array}
	\right) &
	\left(
		\begin{array}{cc}
		0 & \ast \\
		\ast & \ast
		\end{array}
	\right)
	\end{array},
	\end{equation}
where $\ast$ represents any integer. Because the transformed binary quartic form $f$ has last coefficient divisible by $p^2$, the cube \eqref{eq:Qinvcubetransf1} has discriminant a multiple of $p^2$ (for mod $p^2$ reasons also). Thus, an appropriate lower triangular unipotent transformation moves any one of its three covariant binary quadratic forms $q_i$ to one with last coefficient divisible by $p^2$ and middle coefficient divisible by $p$. Thus, by using such lower triangular unipotent linear transformations in these three directions of the cube in order to render $q_1$, $q_2$, and $q_3$ in this form, we obtain a cube of the following form:
	\begin{equation} \label{eq:Qinvcubetransf2}
	\begin{array}{cc}
	\left(
		\begin{array}{cc}
		1 & \ast \\
		\ast & p\ast
		\end{array}
	\right) &
	\left(
		\begin{array}{cc}
		\text{$p$-adic unit} & p* \\
		p* & p^2*
		\end{array}
	\right)
	\end{array}
	\end{equation}
Indeed, note that if the upper left entry of the second matrix were not a unit, then the discriminant of the cube \eqref{eq:Qinvcubetransf1} would be a multiple of $p^2$ for mod $p$ reasons. Since $p^2$ divides the discriminant of the cube \eqref{eq:Qinvcubetransf2}, the lower right entry of the second matrix must then be divisible by $p^2$. And, because $p$ divides the middle coefficient of the binary quadratic form $q_1$, the lower right entry of the left matrix of \eqref{eq:Qinvcubetransf2} must be divisible by $p$. 

Finally, we note that elements of $U^4$ do not change most of the top row of \eqref{eq:QinvHCembed}; only the (1222)-entry (where $a_1$ originally was) can be changed. In fact, after all the transformations, that entry is divisible by $p$, since the binary quartic form $f$ is not changed by any of the last three transformations, and its second-to-last coefficient is now divisible by $p$.

Therefore, after the lower triangular unipotent transformations described above in the four directions, we obtain a hypercube of the form

	\begin{equation} \label{eq:Qinvtransf3}
	\begin{array}{cc}
	\left(
		\begin{array}{cc}
		0 & 0 \\
		0 & 1
		\end{array}
	\right) &
	\left(
		\begin{array}{cc}
		0 & 1 \\
		1 & p\ast
		\end{array}
	\right) \\[.175in]
	\left(
		\begin{array}{cc}
		1 & \ast \\
		\ast & p\ast
		\end{array}
	\right) &
	\left(
		\begin{array}{cc}
		\ast & p\ast \\
		p\ast & p^2\ast
		\end{array}
	\right)
	\end{array}.
	\end{equation}
Multiplying the hypercube \eqref{eq:Qinvtransf3} by $p^2$ and acting by $(\begin{smallmatrix} 1 & 0 \\ 0 & p^{-1} \end{smallmatrix})$ in all four directions gives the hypercube
	\begin{equation} \label{eq:Qinvtransfinal}
	\begin{array}{cc}
	\left(
		\begin{array}{cc}
		0 & 0 \\
		0 & 1
		\end{array}
	\right) &
	\left(
		\begin{array}{cc}
		0 & 1 \\
		1 & \ast
		\end{array}
	\right) \\[.175in]
	\left(
		\begin{array}{cc}
		p & \ast \\
		\ast & \ast
		\end{array}
	\right) &
	\left(
		\begin{array}{cc}
		\ast & \ast \\
		\ast & \ast
		\end{array}
	\right)
	\end{array}
	\end{equation}
in $V_0(\Z)$ with $Q$-invariant $p$.
\end{proof}

To complete the proof of Proposition \ref{prop:denombound}, we now show that, when $G(\Z)$-equivalence classes of distinguished hypercubes are ordered by height, $100\%$ have a well-defined $Q$-invariant.  We then asymptotically count the number of $G(\Z)$-equivalence classes of distinguished hypercubes in $V(\Z)$ having bounded height and large $Q$-invariant.

\begin{proposition}\label{prop:qestimate} \ 

\begin{enumerate}
\item[{\rm (a)}] The number of elliptic curves $E\in F_2$ such that $E(\Q)[2]\neq 0$ and $H(E)< X$ is $O_\varepsilon(X^{7/12+\varepsilon})$.

\item[{\rm (b)}]  Let $h$ take a random value in $G_0$ uniformly with respect to the Haar measure $dg$.  Then the expected number of distinguished elements $B\in\FF h R^{(i)}\cap V(\Z)$ such that
the elliptic curve $E$ corresponding to $B$ has trivial $2$-torsion over $\Q$, the $Q$-invariant of $B$ is larger than $Y$, and $H(B)< X$ is $O_\varepsilon(X^{2/3+\varepsilon}/Y)+o(X^{2/3})$.

\end{enumerate}
\end{proposition}

\begin{proof}  \ 
\begin{enumerate}
\item[(a)] For an elliptic curve in $F_2$ to have nontrivial $2$-torsion, the corresponding cubic $f(x)$ when~\eqref{eq:F2eqn} is put into short Weierstrass form over $\Z$ must have an integer root. Explicitly, 
$$ f(x) = (x-4a_2)(x-4a_2')(x-4a_2'') + (a_1x+4a_3)^2, $$
whose constant term is $$c=-64a_2a_2'(a_2+a_2')+16a_3^2.$$   We show that the total number of choices for $a_1=O(X^{2/24})$, $a_2,a_2'=O(X^{4/24})$ and $a_3=O(X^{6/24})$ such that $f(x)$ has an integer root is $o(X^{2/3})$.

We consider two cases, $c=0$ and $c\neq 0$.  If $c=0$, then $a_3=\pm2\sqrt{a_2a_2'(a_2+a_2')}$, so $a_3$ is determined up to at most choices by $a_2,a_2'$.  Hence there are at most $O(X^{2/24}\cdot X^{4/24}\cdot X^{4/24})=O(X^{5/12})$ choices for $a_1,a_2,a_2',a_3$ if $c=0$.

We now assume we are in the case where $c\neq 0$.  Fix choices of $a_2,a_2'=O(X^{4/24})$ and $a_3=O(X^{6/24})$ such that $c\neq 0$.
Then any root $r$ must divide $c$; thus there are at most $O_\varepsilon(X^\varepsilon)$ possibilities for $r$. Once $r$ is fixed, this also then determines $a_1$ up to at most two possibilities via the equation $f(r)=0$. We conclude that the total number of possibilities for $a_2,a_2',a_3,a_1$ such that $c\neq 0$ and $f(x)$ has an integer root is $O_\varepsilon(X^{4/24}\cdot X^{4/24}\cdot X^{6/24}\cdot X^\varepsilon)=O_\varepsilon(X^{7/12+\varepsilon})$. 

\item[(b)] We first observe that, by Proposition~\ref{hard2}, the expected number of distinguished elements $B\in\FF h R^{(i)}\cap V(\Z)$ such that $b_{1111}\neq 0$ is $o(X^{2/3})$. Thus we may restrict to counting only elements $B$ such that $b_{1111}=0$.  In that case, Proposition~\ref{hard} and Table 2(a) shows that the expected number of elements $B\in\FF h R^{(i)}\cap V(\Z)$ such that $b_{1111}=0$ but at least two of $b_{1112},b_{1121},b_{1211},b_{2111}$ are nonzero is $O(X^{(n-1)/k})$.  Hence we may further restrict to counting only elements $B$ such that $b_{1111}=0$ and at least three
of $b_{1112},b_{1121},b_{1211},b_{2111}$ are $0$.  In that case, for $B$ to be distinguished, we must in fact have
$b_{1111}=b_{1112}=b_{1121}=b_{1211}=0$, i.e., $B\in V_0(\Z)$; furthermore, we must have that $b_{2111},b_{1122},b_{1212},b_{1221}\neq 0$, for otherwise $\Delta(B)$ would be zero.

We must thus bound the expected number of elements $B\in\FF h
R^{(i)}\cap V(\Z)$ such that the elements of
$T_0=\{b_{1111},b_{1112},b_{1121},b_{1211}\}$ are all $0$, the
elements of $T_1=\{b_{2111},b_{1122},b_{1212},b_{1221}\}$ are all
nonzero, and $H(B)< X$. This would have been the next natural case
$\CC$ in Table 2(a), had it not led to reducibility.  Furthermore, we may 
assume that $|b_{2111}|>Y$, since we only wish to count those elements $B$
whose $Q$-invariant is larger than $Y$. We can now estimate
$N(V(\CC);X)$ exactly as in (\ref{estv1s}) and (\ref{estv2s})---but noting that we may bound
the $s_i$'s from above by a power of $X$, for otherwise $\H$ would have
no integer points with all elements of $T_1$ nonzero. In fact, we
easily calculate that we must have $s_1\ll X^{2/k}$ and
$s_2,s_3,s_4\ll X^{1/k}$ for the region $\H$ to have any such integer points.  Let us take $\pi=\prod_{t\in T_1} t = b_{2111}b_{1122}b_{1212}b_{1221}$ in
(\ref{estv2s}).  Then $|\pi|\geq Y$, and therefore
\begin{align*}\label{distcountincusp}
N(V(\CC);X) &\ll \frac1Y \int_{s_1,s_2,\cdots=c}^{X^{2/k}} X^{\frac{n-|T_0| + \#\pi}{k}} \prod_{b\in T\setminus T_0} w(b)\;w(\pi)
\,\, s_1^{-2}s_2^{-2}\cdots \,\,\, d^\times\! s_2\,d^\times\!s_1 \\
&=\frac1Y \int_{s_1,s_2,\cdots=c}^{X^{2/24}} X^{\frac{16-4+4}{24}} (s_1^4s_2^2s_3^2s_4^2)  \; (s_1^{-2}) \; s_1^{-2}s_2^{-2}s_3^{-2}s_4^{-2} d^\times\! s_4\,d^\times\! s_3\,d^\times\! s_2\,d^\times\!s_1 \\
&\ll \frac1Y X^{2/3}\log^4(X),
\end{align*}
yielding part (b).
\qedhere \end{enumerate}
\end{proof}

\noindent
Proposition~\ref{prop:denombound} for $F=F_2$ now follows from Proposition~\ref{prop:qestimate}.

\subsection{Evaluation of the average size of the $2$-Selmer group}\label{finalsec}

We now have the following theorem, from which Theorem \ref{congmain} 
(and thus Theorem~\ref{thm:SelmerAverages})
% and \ref{main22}) 
will be seen to follow.
\begin{theorem}\label{main2}
Let $\Phi$ be any large subfamily of $F$. Then we have
\begin{eqnarray*}\label{eqthsec5}
\displaystyle\lim_{X\to\infty}\frac{\displaystyle\sum_{\substack{E\in \Phi\\H(E)<X
}}\#(S(E)\setminus S'(E))}{\displaystyle\sum_{\substack{E\in \Phi\\H(E)<X}}1}&=&
\frac{|\mathcal J|\cdot \Vol(G(\Z)\backslash G(\R))\cdot \displaystyle{\sum_{i=1}^{N} \frac1{n_i}
\cdot \int_{{H(\vec a)<X}\atop{\pm\Delta(\vec a)>0}}d\vec a}}{\displaystyle{\sum_{i=1}^n\int_{{H(\vec a)<X}\atop{\pm\Delta(\vec a)>0}}d\vec a}} \\[.15in]
& & \cdot\, \displaystyle\prod_p\frac{\displaystyle{|\mathcal J|_p\cdot \Vol(G(\Z_p)) \cdot
\int_{C=C(\vec a)\in \Phi_p}\sum_{\sigma\in \textstyle\frac{E(\Q_p)}{d E(\Q_p)}}\frac{1}{\#E[d](\Q_p)}d\vec a}}{\displaystyle{\int_{C=C(\vec a)\in \Phi_p}d\vec a}} .
\end{eqnarray*}
\end{theorem}
\begin{proof}
This follows by combining Proposition~\ref{denel}, Theorem~\ref{ufcount} and expression (\ref{volexp}) for the volume $\Vol(\RR_X)$, and Theorem~\ref{ccount}.
\end{proof}

\noindent
In order to evaluate the right hand side of the expression in Theorem~\ref{main2}, we
use the following fact (see \cite[Lemma~3.1]{BK}):
\begin{lemma}\label{lembk}
Let $J$ be an abelian variety over $\Q_p$ of dimension $n$. Then
$$\#(J(\Q_p)/d J(\Q_p))=
\left\{\begin{array}{cl}
\#J[d](\Q_p) & {\mbox{\em if }} p\neq d;\\[.075in]
d^n\cdot\#J[d](\Q_p)& {\mbox{\em if }} p= d.
\end{array}\right.$$
\end{lemma}
\begin{proof}
  It follows from the theory of formal groups 
  that there exists a subgroup
  $M\subset J(\Q_p)$ of finite index that is isomorphic to $\Z_p^n$. Let
  $H$ denote the finite group $J(\Q_p)/M$.
 Then by applying the snake lemma to the following diagram
$$\xymatrix{
0\ar[d]\ar[r]&M\ar[d]^{[d]}\ar[r]&J(\Q_p)\ar[d]^{[d]}\ar[r]&H\ar[d]^{[d]}\ar[r]&0
\ar[d]\\
0\ar[r]&M\ar[r]&J(\Q_p)\ar[r]&H\ar[r]&0}$$
we obtain the exact sequence
$$0\to M[d]\to J(\Q_p)[d]\to H[d]\to M/ d M\to J(\Q_p)/d J(\Q_p)\to H/d H\to 0.$$
Since $H$ is a finite group and $M$ is isomorphic to $\Z_p^n$, Lemma \ref{lembk}  follows.
\end{proof}

The expression on the right hand side in Theorem~\ref{main2} thus
reduces simply to the Tamagawa number
$\tau(G)=\Vol(G(\Z)\backslash G(\R))\cdot \prod_p
\Vol(G(\Z_p))$ of $G$.  This gives the average number of elements in $S(E)\setminus S'(E)$ over all elliptic curves
$E\in \Phi$.  

Meanwhile, the size of $S'(E)$ is 1 except in the cases where $F=F_i$ for $i\in\{1,2\}$.  In that case, as we will prove in the next section (Prop.\ \ref{inforder}), we have for 100\% of elliptic curves $E\in F_i$ that $|S'(E)|=d^i$.  
This completes the proof of
Theorem~\ref{congmain}, and thus of Theorem~\ref{thm:SelmerAverages}.

\section{The marked points in $F_1$ and $F_2$ are independent} \label{sec:markedpts}

In this section, we prove that the marked points in the families $F_1$ and $F_2$ are non-torsion and ``independent'' asympototically $100\%$ of the time.  More precisely, we prove the following theorem. 

\begin{theorem}\label{inforder}
Let $j \in \{1,2\}$ and $d$ any positive integer.  For an elliptic curve $E\in F_j$, let $S(E)$ denote the $d$-Selmer group of $E\in F_i$ and $S'(E)$ the subgroup in $S(E)$ generated by the images of the marked points on $E$.  Then, when elliptic curves $E\in F_i$ are ordered by height, $100\%$ of these curves $E$ have the property that $|S'(E)|=d^j$.
\end{theorem}

Theorem~\ref{inforder} implies that, for $100\%$ of the curves $E$ in $F_j$, the subgroup in $E(\Q)$ generated by the marked points has rank $j$. Thus, when elliptic curves $E\in F_j$ are ordered by height, $100\%$ of these curves $E$ have rank at least $j$.  

\begin{proof}
Let $j\in\{1,2\}$.  Let $n_1=3$ and $k_1=4$, and $n_2=4$ and $k_2=6$.  Then the total number of elliptic curves $E\in F_j$ with height less than $X$ is $\sim c_j X^{n_j/k_j}$ for a positive constant $c_j$.  This is seen simply by counting lattice points in a certain bounded region in $\R^{n_j}$.

We first note that asympotically $100\%$ of the elliptic curves in $F_j$ have trivial rational torsion. The generic elliptic curve for each $F_j$ (over $\Q(\{a_i\})$) has no nonzero rational $\ell$-torsion for primes $\ell \leq 7$. By the Hilbert irreducibility theorem, the same holds for asymptotically $100\%$ of the curves in $F_j$, and Mazur's theorem implies that asymptotically $100\%$ of the curves in $F_j$ have trivial rational torsion. This immediately implies that the marked points on $100\%$ of the curves in $F_j$ have infinite order.

We now show that in $F_2$, the two marked points $P_1$ and $P_2$ are independent asymptotically $100\%$ of the time as well. The argument is similar to the argument in \S \ref{sec:reducibility} proving that reducible elements in $V(\Q)$ are rare, namely reducing modulo $p$ and showing that the desired property is rare enough modulo $p$. Here, we take dependence to mean that the two points generate a cyclic subgroup, and we restrict our attention to the set $F_2^{\textrm{nt}}$ of curves in $F_2$ that have trivial rational torsion. Note that dependence of two points on $E(\Q)$ implies dependence after reducing $E$ modulo any prime.

Let $S^\mathrm{dep}$ denote the set of curves $E$ in $F_2^{\textrm{nt}}$ where $P_1$ and $P_2$ are dependent, and let $S^\mathrm{dep}_p$ denote the set of curves $E$ in $F_2^{\textrm{nt}}$ where $P_1$ and $P_2$ are dependent modulo the prime $p$ (i.e., dependent as points of $E(\Fp)$).  Then $S^\mathrm{dep} \subset \cap_p S^\mathrm{dep}_p$.

Let $S^\mathrm{cyc}_p$ denote the set of curves in $F_2$ where the reduction modulo $p$ is a cyclic group. Let $T^\mathrm{cyc}_p$ denote the set of (isomorphism classes of) elliptic curves $E$ over $\Fp$ where $E(\Fp)$ is a cyclic group. Note that for a curve $E \in S^\mathrm{cyc}_p \cap F_2^{\textrm{nt}}$, we have $E \in S^\mathrm{dep}_p$ and the reduction of $E$ modulo $p$ lies in $T^\mathrm{cyc}_p$. For a fixed prime $p$, Vl{\v{a}}du{\c{t}} \cite{vladut} shows that the probability of $E(\Fp)$ being cyclic for $E$ an elliptic curve over $\Fp$ is
$$
\mu_p(T^{\mathrm{cyc}}_p) := \prod_{\textrm{primes } \ell \mid p-1} \left( 1 - \frac{1}{\ell(\ell^2-1)} \right) + O(p^{-1/2+\varepsilon})
$$
which is $\leq  5/6 + O(p^{-1/2+\varepsilon})$ for $p \neq 2$. Since the constant in the error term does not depend on $p$, for sufficiently large $p$, we must have $\mu_p(T_p^\mathrm{cyc}) \leq 5/6+\delta$ for any $\delta > 0$.

In order to use the above bound on $\mu_p(T^{\mathrm{cyc}}_p)$ to compute the probability that a curve $E \in F_2$ lies in $S^{\mathrm{cyc}}_p$, recall that the set of curves in $F_2$ modulo $p$ is the same as elliptic curves modulo $p$, with multiplicity given by the number of pairs of distinct non-identity points on the curve. Using (a weak version of) the Weil bound, we thus find that the density of $S^{\mathrm{cyc}}_p$ in $F_2$ is bounded above by
$$\frac{(5/6 + \delta)(p+O(\sqrt{p}))^2}{(5/6 + \delta)(p+O(\sqrt{p}))^2 + (1/6-\delta)(p+O(\sqrt{p}))^2}$$
for any $\delta > 0$. For sufficiently large $p$, we thus find that this density of $S^{\mathrm{cyc}}_p$ in $F_2$ is at most $5/6 + \delta'$ for any $\delta' > 0$.

If $E \in S^\mathrm{dep}_p$ but not in $S^\mathrm{cyc}_p$, then the group $E(\Fp)$ is the product of two nontrivial cyclic groups, and the probability that the two marked points are dependent in $E(\Fp)$ is bounded above by $1/2$. Thus, combining this with the density of $S^{\mathrm{cyc}}_p$, for sufficiently large $p$ we find that the density of $S^\mathrm{dep}_p$ in $F_2$ is at most
$$ (5/6+\delta') + 1/2(1/6-\delta') = 11/12 + \delta'/2$$
for any $\delta' > 0$.

By the Chinese remainder theorem, for any finite set $\Sigma$ of sufficiently large primes $p$, the density of curves in $F_2$ where the two marked points are dependent modulo $p$ for all $p \in \Sigma$ is at most $\prod_{p \in \Sigma} (11/12+\delta'/2)$. Letting the size of $\Sigma$ approach infinity yields the result.
\end{proof}

\subsection*{Acknowledgments}

We are grateful to Bhargav Bhatt, John Cremona, Aise Johan de Jong, Tom Fisher, Benedict 
Gross, Catherine O'Neil, Arul Shankar, Christopher Skinner, and Xiaoheng Wang
for helpful conversations. The first author was partially supported
by a Simons Investigator Grant and NSF Grant~DMS-1001828.  The second
author was partially supported by NSF Grants~DMS-1701437 and DMS-1844763, the Sloan Foundation, and the Minerva Research Foundation.

\newpage
\bibliography{cofreecountbib}
\bibliographystyle{amsalpha2}

\end{document}